%% file: 00-main.tex
\documentclass[reqno]{amsart}

\usepackage{mathrsfs}
\usepackage{enumitem}
\usepackage{amssymb}
\usepackage{enumitem}
\usepackage{thm-restate}
\usepackage{tikz}
\usetikzlibrary{shapes.geometric}
\usepackage{hyperref}
\usepackage[cmtip,all]{xy}
\usepackage{enumitem}
\usepackage{undertilde}

\usepackage{multicol}
\usepackage[utf8]{inputenc}
\usepackage{color}
\usepackage{xcolor}
\usepackage{comment}
\specialcomment{proofdetail}{\color{blue}}{\color{black}}
\usepackage{orcidlink}

\excludecomment{proofdetail}

\newcommand{\nocontentsline}[3]{}
\newcommand\stoptoc{%
  \let\origcontentsline\addcontentsline
  \let\addcontentsline\nocontentsline
}
\newcommand\resumetoc{%
  \let\addcontentsline\origcontentsline
}

\newcommand{\sspace}{\mathcal{X}}
\newcommand{\sspacealt}{\mathcal{Y}}

\newcommand{\ssvar}{x}

\newcommand{\ssvaralt}{y}

\newcommand{\ssvarn}[1]{\ssvar_{#1}}

\newcommand{\pspace}{\Omega}
\newcommand{\pspacealt}{\widetilde{\Omega}}

\newcommand{\psvar}{\theta}

\newcommand{\prps}{\Theta}

\newcommand{\prss}{X}

\newcommand{\pmap}[1]{p(\cdot \hspace{-1mm}\mid\hspace{-1mm} {#1})}
\newcommand{\pmapalt}[1]{q(\cdot \hspace{-1mm}\mid\hspace{-1mm} {#1})}

\newcommand{\pmapi}[2]{p_{#1}(\cdot \hspace{-1mm}\mid\hspace{-1mm} {#2})}

\newcommand{\pmapws}[1]{p(\cdot \mid {#1})}

\newcommand{\pmape}[2]{p(#1 \hspace{-1mm}\mid\hspace{-1mm} #2)}
\newcommand{\pmapalte}[2]{q(#1 \hspace{-1mm}\mid\hspace{-1mm} #2)}

\newcommand{\fmape}[2]{f(#1 \hspace{-1mm}\mid\hspace{-1mm} #2)}
\newcommand{\gmape}[2]{g(#1 \hspace{-1mm}\mid\hspace{-1mm} #2)}
\newcommand{\hmape}[2]{h(#1 \hspace{-1mm}\mid\hspace{-1mm} #2)}

\newcommand{\Gmape}[2]{G(#1 \hspace{-1mm}\mid\hspace{-1mm} #2)}
\newcommand{\Hmape}[2]{H(#1 \hspace{-1mm}\mid\hspace{-1mm} #2)}

\newcommand{\pmapews}[2]{p(#1 \mid #2)}

\newcommand{\prior}{p}
\newcommand{\prioralt}{q}

\newcommand{\prodspace}{S}

\newcommand{\prodm}{\mu}
\newcommand{\prodmalt}{\nu}

\newcommand{\post}[1]{\prior(\cdot \hspace{-1mm}\mid\hspace{-1mm} #1)}
\newcommand{\postalt}[1]{\prioralt(\cdot \hspace{-1mm}\mid\hspace{-1mm} #1)}

\newcommand{\postws}[1]{\prior(\cdot \mid #1)}

\newcommand{\poste}[2]{\prior(#1 \hspace{-1mm}\mid\hspace{-1mm} #2)}
\newcommand{\postalte}[2]{\prioralt(#1 \hspace{-1mm}\mid\hspace{-1mm} #2)}

\newcommand{\ssfilt}[1]{\mathscr{F}_{#1}}

\theoremstyle{definition}
\newtheorem{thm}{Theorem}[section]

\newtheorem{prop}[thm]{Proposition}

\newtheorem{defn}[thm]{Definition}
\newtheorem{ex}[thm]{Example}
\newtheorem{rmk}[thm]{Remark}

\newtheorem*{claim*}{Claim}

\numberwithin{equation}{section}

\begin{document}


\title[Schnorr randomness and Effective Bayesian Consistency]
{Schnorr randomness and Effective Bayesian Consistency and Inconsistency}

\author{Simon M. Huttegger\,\orcidlink{0000-0003-2374-3889}}

\address{Department of Logic and Philosophy of Science \\ 5100 Social Science Plaza \\ University of California, Irvine \\ Irvine, CA 92697-5100, U.S.A.}

\email{shuttegg@uci.edu}

\urladdr{http://faculty.sites.uci.edu/shuttegg/}

\author{Sean Walsh\,\orcidlink{0009-0007-5460-778X}}

\address{Department of Philosophy \\ University of California, Los Angeles \\ 390 Portola Plaza, Dodd Hall 321 \\ Los Angeles, CA 90095-1451}

\email{walsh@ucla.edu}

\urladdr{http://philosophy.ucla.edu/person/sean-walsh/}

\author{Francesca Zaffora Blando\,\orcidlink{0009-0004-9065-9198}}

\address{Department of Philosophy\\ Carnegie Mellon University \\Baker Hall 161\\
5000 Forbes Avenue\\ Pittsburgh, PA 15213}

\email{fzaffora@andrew.cmu.edu}

\urladdr{https://francescazafforablando.com/}


\date{\today}

\begin{abstract}
\makeatletter\phantomsection\def\@currentlabel{(abstract)}\makeatother
We study Doob's Consistency Theorem and Freedman's Inconsistency Theorem from the vantage point of computable probability and algorithmic randomness. We show that the Schnorr random elements of the parameter space are computably consistent, when there is a map from the sample space to the parameter space satisfying many of the same properties as limiting relative frequencies. We show that the generic inconsistency in Freedman's Theorem is effectively comeager, which implies the existence of computable parameters which are not computably consistent. Taken together, this work provides a computability-theoretic solution to Diaconis and Freedman's problem of ``know[ing] for which [parameters] $\psvar$ the rule [Bayes' rule] is consistent'' (\cite[4]{Diaconis1986-us}),
and it strengthens recent similar results of Takahashi \cite{Takahashi2023-rl} on Martin-L\"of randomness in Cantor space.

\end{abstract}

\maketitle

 \tableofcontents

\input{01-intro}

\input{02-sigma02-schnorr}

\input{04-productspace}

\input{05-proof-doob}

\input{06-paradigmatic}

\input{07-reversal}

\input{08.5-simplex}

\input{08-general-con}

\input{09-incon}

\bibliographystyle{alpha}
\bibliography{00-bib}

\end{document}

%% file: 01-intro.tex
\section{Introduction}\label{sec:intro}

A statistical model consists of a sample space and a collection of probability distributions on the sample space, one of which is hypothesized to be the data-generating process. The distributions are often indexed by a parameter space, which is an assumption that we will make throughout this paper. In a setting where the number of sample observations grows without bound, statisticians consider the question of whether estimates are {\it consistent}, where, loosely speaking, an estimate is consistent if it is guaranteed to identify the true parameter.

Consistency is of major importance in classical statistics. For instance, Barnett \cite[p. 135]{Barnett1982} remarks that ``[c]onsistency is generally regarded as an essential property of a reasonable estimator.'' As Diaconis and Freedman \cite{Diaconis1986-us} point out, it also plays an important role in Bayesian statistics and thus constitutes a point of contact between classical and Bayesian statistics.\footnote{There is a large literature on the consistency of estimates in the statistics literature. We plan to put it into perspective, especially in light of our results, in a future paper.} 

Suppose we have a prior probability measure on the parameter space. Doob \cite{Doob1949-hg} showed that, under mild conditions, the posteriors are consistent for {\it almost all} parameters\textemdash that is, with increasing sample size, the posterior concentrates around the true parameter whenever it is in a set of prior probability one. Further, consistency can often be achieved for {\it all} parameters, in the finite-dimensional parametric case (\cite{Freedman1963-aa,schwartz1965bayes}).

Things are much more complicated in the nonparametric case. Freedman demonstrated that inconsistency can be {\it topologically generic} in this setting \cite{Freedman1965-aa}. The literature following Freedman has generalized this result, establishing the pervasiveness of inconsistency for nonparametric models (\cite{Diaconis1986-us}).

In the present paper, we investigate statistical consistency from the point of view of computable probability theory and algorithmic randomness (see \cite{Gacs2005}, \cite{Hoyrup2009-pl}, \cite{Rut:2013}, \cite{Hoyrup2021-qj}, and \cite{Huttegger2024-zb}).
Before stating our main results, we begin by recalling elements of the classical statistical setup in more detail.\footnote{The notation for the different parts of the classical setup can vary between authors. We are broadly following the notation set out in Chapter 1 of \cite{Schervish2012-cn}.}

\emph{Parameter space and prior}. We suppose that the parameter space $\pspace$ is a Polish space and that the prior $\prior$ is a probability measure on $\pspace$.

\emph{Sample space}. We suppose that the sample space $\sspace$ is a closed subset of Baire space~$\mathbb{N}^{\mathbb{N}}$, and hence is representable as the set of paths $[T]$ through a tree $T\subseteq \mathbb{N}^{<\mathbb{N}}$. As Diaconis and Freedman say: ``all the paradoxes of inconsistency already appear'' in the case of this kind of sample space''.\footnote{\cite[2]{Diaconis1986-us}. It is not obvious whether our results generalize to the sample space being $\mathbb{R}^{\mathbb{N}}$.} We use $\ssvarn{n}$ for the finite sequence of natural numbers which results from restricting the element $\ssvar$ of the sample space to its first $n$ digits.\footnote{One often writes $\ssvarn{n}$ as $\ssvar\upharpoonright n$. The notation $\ssvarn{n}$ makes vivid the statistical interpretation.} Each string of natural numbers $\sigma$ in $T$ of length $n$ determines a basic clopen subset $[\sigma]=\{\ssvar\in \sspace: \ssvarn{n}= \sigma\}$ of the sample space $\sspace$.  

\emph{Likelihood}. We suppose that the likelihood $\psvar\mapsto \pmap{\psvar}$ is a measurable map from the parameter space $\pspace$ to the Polish space $\mathcal{M}^+(\sspace)$ of finite Borel measures on the sample space $\sspace$, and that, for $\prior$-a.s. many $\psvar$, one has that $\pmap{\psvar}$ is an element of the Polish space $\mathcal{P}(\sspace)$ of Borel probability measures on~$\sspace$. Since any element $\nu$ of $\mathcal{M}^+(\sspace)$ is determined by its action on the clopens, we abbreviate $\nu([\sigma])$ as $\nu(\sigma)$.

\emph{Product space and joint distribution}. The prior and the likelihood induce a probability measure $\prodm$, called the \emph{joint distribution}, on the product space $\prodspace=\pspace\times \sspace$ given by the formula $\prodm(A\times B) = \int_A \pmape{B}{\psvar} \; d\prior(\psvar)$ on rectangles. For a $\prodm$-integrable function $f:S\rightarrow \mathbb{R}$, one has the following formula for the integral: 
\begin{equation}\label{eqn:formula:int:mu}
\int_{S} f(\psvar,\ssvar) \; d\prodm(\psvar,\ssvar) = \int_{\pspace} \int_{\sspace} f(\psvar,\ssvar) \; d\pmap{\psvar}(\ssvar) \; d\prior(\psvar)
\end{equation} Finally, the joint distribution $\prodm$ induces a pushforward $\prodm_{\prss}$ on the sample space $\sspace$. Its pushforward on the parameter space $\pspace$ is just the prior $\prior$.

\emph{Posterior}. We define the posterior as a map $\sigma \mapsto \postws{\sigma}$ from the tree $T$ to the Polish space $\mathcal{M}^+(\pspace)$ of finite Borel measures on the parameter space $\pspace$ by
\begin{equation}\label{eqn:post}
 \poste{A}{\sigma} = \frac{\int_A \pmape{\sigma}{\psvar}\; d\prior(\psvar)}{\int_{\pspace} \pmape{\sigma}{\psvar}\; d\prior(\psvar)}
\end{equation}
when the denominator is non-zero, and we assume that the posterior is zero otherwise. For each $n\geq 0$, the posterior induces a map $\ssvar\mapsto \post{\ssvarn{n}}$ from the sample space $\sspace$ to $\mathcal{M}^+(\pspace)$. For $\prodm_{\prss}$-a.s. many $\ssvar$ from $\sspace$, one can show that $\post{\ssvarn{n}}$ is in $\mathcal{P}(\pspace)$, the space of probability measures on the parameter space (cf. Proposition~\ref{prop:whenpostprob2}).

\emph{Relation to discrete forms of Bayes' Theorem}. In the case where the parameter space $\pspace$ is countable and $\sigma=\ssvarn{n}$, if we momentarily abuse notation by identifying $\psvar$ with its singleton event $A=\{\psvar\}$, then (\ref{eqn:post}) reduces to the familiar discrete form $ \poste{\psvar}{\ssvarn{n}} = \frac{\pmapews{\ssvarn{n}}{\psvar} \cdot \prior(\psvar)}{\sum_{\psvar\in \pspace} \pmapews{\ssvarn{n}}{\psvar} \cdot \prior(\psvar)}$ of Bayes' Theorem.

\begin{figure}
\begin{center}
\begin{equation*}
\xymatrix{
 & & & & \mathcal{M}^+(\sspace)\\
(\pspace, \prior) \ar@/^2pc/@{->}[urrrr]_{\psvar\mapsto \pmapws{\psvar}} \ar@{->}[d]_{\psvar\mapsto \delta_{\psvar}}& & (\prodspace, \prodm) \ar@{->}[ll]_{\prps} \ar@{->}[rr]^{\prss} & & (\sspace, \prodm_{\prss}) \ar@{->}[u]_{\ssvar\mapsto \delta_{\ssvar}}  \ar@/^2pc/@{->}[dllll]_{\ssvar\mapsto \postws{\ssvarn{n}}\;\;} \\
\mathcal{M}^+(\pspace) & & & & 
}
\end{equation*}
\caption{The five spaces and maps between them.}
\label{fig:generalsetup}
\end{center}
\end{figure}
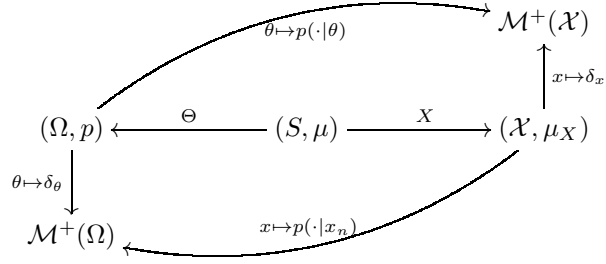

\emph{Diagram of the spaces}. The situation is then as depicted in Figure~\ref{fig:generalsetup}, where the $\sigma$-algebras on all five Polish spaces are the corresponding Borel $\sigma$-algebras, and where $\prodm_{\prss}$ refers again to the pushforward of $\prodm$ onto the sample space. We name the two projections $\prps:\prodspace\rightarrow \pspace$ and $\prss:\prodspace\rightarrow \sspace$. This diagram can help keep track of what is meant by the expression $p(\cdot \mid \cdot)$. Note that, in our notation, $p(\cdot \mid \cdot)$ does \emph{not} mean the ordinary conditional probability with respect to the prior.

Our goal is to study this setup within the framework of computable probability theory and effective analysis. We make the following assumptions:
\begin{defn}[Effective assumptions]\label{defn:assum}
\text{}
\begin{enumerate}[leftmargin=*]
    \item \label{assum2} The parameter space $\pspace$ is a computable Polish space, and the prior $\prior$ is a computable probability measure on $\pspace$.
    \item \label{assum1} The sample space $\sspace$ is the set of paths $[T]$ through a computable tree $T\subseteq \mathbb{N}^{<\mathbb{N}}$ which has no dead ends. 
   \item \label{assum2.5} The likelihood $\psvar\mapsto \pmape{\sigma}{\psvar}$ is $L_1(\prior)$ computable, uniformly in $\sigma$ from $T$; and $\pmap{\psvar}$ is in $\mathcal{P}(\sspace)$ for $\prior$-a.s. many $\psvar$ in $\pspace$.
   \end{enumerate}
\end{defn}

For computable Polish spaces and computable probability measures on them, see again the sources \cite{Gacs2005}, \cite{Hoyrup2009-pl}, \cite{Rut:2013}, \cite{Hoyrup2021-qj}, and \cite{Huttegger2024-zb}. Readers unfamiliar with computable Polish spaces can still read much of this paper, simply by specializing to paradigmatic cases, such as when the parameter space is Cantor space, Baire space, $[0,1]^n$ or $\mathbb{R}^n$.

Together, (\ref{assum2})-(\ref{assum2.5}) suffice to ensure that the joint distribution $\prodm$ is a computable probability measure on the product space $\prodspace$ (cf. Proposition~\ref{prop:prodmcomputable}). In (\ref{assum1}), having no dead ends provides a natural countable dense set\textemdash namely, the leftmost path above each $\sigma$ in $T$\textemdash and so ensures that the sample space $\sspace=[T]$ is a computable Polish space. This is also natural on the statistical interpretation, since finite samples can at least in principle always be continued. 

In (\ref{assum2.5}), if $r$ is a computable real $\geq 1$, then, by $L_r(\prior)$ computable, we just mean a computable point of the computable Polish space $L_r(\prior)$ of Borel measurable functions whose absolute value has $r$-th power which is integrable (the functions in $L_r(\prior)$ being identified when they are $\prior$-a.s equal). Its metric is given by the $L_r(\prior)$ norm $\|f\|_{L_r(\prior)}$\textemdash namely, the $r$-th root of $\int_{\pspace} \left|f\right|^r\; d\prior$\textemdash and its countable dense set is a certain set of rational-valued step functions with events whose $\prior$-measure can be computed.\footnote{The events are those from the algebra generated by a $\prior$-computable basis, see Definition~\ref{defn:nucomputablebasis}. For more on $L_r(\prior)$ as a computable Polish space, see \cite[\S{2.3}]{Huttegger2024-zb}.} 

Condition (\ref{assum2.5}) of Definition~\ref{defn:assum} is admittedly a strong condition. For instance, it implies that the $L_1(\prior)$ norm $\|\pmape{\sigma}{\cdot}\|_{L_1(\prior)}$ of the likelihood $\psvar\mapsto \pmape{\sigma}{\psvar}$ is computable for any $\sigma$, since it is the distance $\|\pmape{\sigma}{\cdot}-0\|_{L_1(\prior)}$, and since the distance between any two computable points in a computable Polish space is automatically computable. Further, condition~(\ref{assum2.5}) has deeper consequences, such as Proposition~\ref{prop:likelihoodareprob} below, which gives a sufficient condition for parameters $\theta$ to be such that $\pmap{\psvar}$ is in $\mathcal{P}(\sspace)$.

We now turn to recalling the components of the definition of Schnorr randomness. An \emph{$L_r(\prior)$ Schnorr test} is a function $g:\pspace\rightarrow [0,\infty]$ whose $L_r(\prior)$ norm $\|g\|_{L_r(\prior)}$ is computable and which is \emph{lower semi-computable} (abbreviated \emph{lsc}): that is, $g^{-1}(q,\infty]$ is c.e. open uniformly in rational $q\geq 0$. Since they are lsc functions, the $L_r(\prior)$ Schnorr tests are everywhere defined, and one can show that their $\prior$-a.s. equivalence class is an $L_r(\prior)$ computable point. 

Suppose that $Y$ is a computable Polish space and that $\nu$ is a computable probability measure on $Y$. Then one defines:\footnote{On Cantor space, see \cite{Downey2010aa} for a comprehensive treatment of randomness notions. Much of this theory generalizes easily from Cantor space to arbitrary computable Polish spaces.}
\begin{defn}\label{defn:random}
A point $y$ of $Y$ is \emph{Schnorr $\nu$-random}, abbreviated $\mathsf{SR}^{\nu}$, if for all computable sequences of c.e. opens $U_n$ with $\nu(U_n)\leq 2^{-n}$ uniformly computable, one has that $y$ is not in $\bigcap_n U_n$.
\end{defn}

\noindent The sequences $U_n$ are called \emph{sequential Schnorr $\nu$-tests}. An equivalent characterization of Schnorr randomness, due to Miyabe,\footnote{\cite[Theorem 3.5]{Miyabe2013-wy}.} is: $y$ in $\mathsf{SR}^{\nu}$ iff $f(y)<\infty$ for all $L_1(\nu)$ Schnorr tests~$f$. 
 
We need other characterizations of $\mathsf{SR}^{\nu}$:
\begin{restatable}{defn}{defnsigmanuzerotwo}\label{defn:defnsigmanuzerotwo}
A class $B$ is \emph{$\nu$-effectively $\Sigma^0_2$}, abbreviated $\Sigma^{0,\nu}_2$, if it has computable $\nu$-measure and it can be written as $B=\bigcup_i C_i$ where $C_i$ is a computable sequence of effectively closed sets $C_i$ which have uniformly computable measure $\nu(C_i)$. 
\end{restatable}
\noindent We have the following characterizations of $\mathsf{SR}^{\nu}$ in terms of $\Sigma^{0,\nu}_2$ sets:

\begin{restatable}{prop}{schnorrsigmazerotwonutest}\label{prop:schnorrsigmazerotwonutest}
A point $y$ is in $\mathsf{SR}^{\nu}$ iff for all computable sequences of $\Sigma^{0,\nu}_2$ sets $B_n$ with $\nu(B_n)\leq 2^{-n}$, one has that $y$ is not in $\bigcap_n B_n$.
\end{restatable}
\begin{restatable}{prop}{propschnorrsigmazeronutwo}\label{prop:schnorrsigma0nu2}
A point $y$ is in $\mathsf{SR}^{\nu}$ iff $y$ is in all $\Sigma^{0,\nu}_2$ classes of $\nu$-measure one.
\end{restatable}

\noindent For spaces like Cantor space, Proposition~\ref{prop:schnorrsigma0nu2} was noted by Schnorr in his book (\cite[68-69]{Schnorr1971-qf}), while Proposition~\ref{prop:schnorrsigmazerotwonutest} is new.

Since, in Definition~\ref{defn:assum}(\ref{assum2.5}), the  likelihood is viewed as an element of the space $L_1(\prior)$, we need to specify how to understand its pointwise behavior:
\begin{defn}[Convention on versions: differences of $L_1(\prior)$ Schnorr tests]\label{defn:conventionversions}
If $f$ is an $L_1(\prior)$ computable function, then we restrict attention to versions of it of the form $g-h$, where $g,h$ are two $L_1(\prior)$ Schnorr tests.
\end{defn}
\noindent By the Miyabe characterization of Schnorr randomness in terms of $L_1(\prior)$ Schnorr tests, these versions are always finite on the Schnorr randoms. By further results of Miyabe, all $L_1(\prior)$ computable functions have such versions, and any two such versions agree on the Schnorr randoms.\footnote{\cite[Theorem 4.3]{Miyabe2013-wy}. Further, these versions are known to agree, again by Miyabe \cite[7]{Miyabe2013-wy}, with the version $\lim_n f_n$ on the Schnorr randoms, where $f_n\rightarrow f$ fast in $L_1(\prior)$ and $f_n$ is a computable sequence from the countable dense set of $L_1(\prior)$. The latter version is preferred by Pathak, Rojas, and Simpson \cite[Definition 3.8, p. 314]{Pathak2014aa} and  Rute \cite[Definition 3.17, p. 16]{Rute2012aa}.} This convention has the virtue of allowing us to use many natural versions. But this convention needs to be paired with some convention on how to treat $g(\psvar)-h(\psvar)$ when both values are infinite. So as to avoid having to define notation on partial functions, we treat it as conventionally equal to zero. 

The paradigmatic example of a parameter space in nonparametric statistics is:
\begin{defn}[Infinite-dimensional simplex]\label{defn:sinfty} 
The infinite-dimensional simplex is 
$\mathbb{S}_{\infty}= \{\psvar\in [0,1]^{\mathbb{N}}: \sum_i \psvar(i)=1\}$.
\end{defn}
\noindent  This is especially natural when the sample space is Baire space $\mathbb{N}^{\mathbb{N}}$ and when the likelihood is the countable product of~$\psvar$, as a probability measure on the natural numbers. This is what happens in countable versions of stickbreaking and the Dirichlet process in nonparametric statistics (cf. \cite[\S{3.3}]{Ghosal2017-jk}, \cite{Sethuraman1994-mq}).

We then define a classical consistency notion and its effectivization:
\begin{defn}[Consistency notions, classical and effective]\label{defn:con:eff} 
\text{}
\begin{enumerate}[leftmargin=*]
    \item \label{defn:con:eff:20} The parameter $\psvar$ is \emph{classically consistent} if $\pmap{\psvar}$ is in $\mathcal{P}(\sspace)$ and, for all open $U$ in $\pspace$, one has that $\lim_n \poste{U}{\ssvarn{n}}=I_U(\psvar)$ for $\pmap{\psvar}$-a.s. many $\ssvar$ in $\sspace$.
    \item \label{defn:con:eff:2} The parameter $\psvar$ is \emph{computably consistent} if $\pmap{\psvar}$ is in $\mathcal{P}(\sspace)$ and, for all c.e. open $U$ in $\pspace$ with $\prior(U)$ computable, one has that $\lim_n \poste{U}{\ssvarn{n}}=I_U(\psvar)$ for $\pmap{\psvar}$-a.s. many $\ssvar$ in $\sspace$.
\end{enumerate}
\end{defn}
\noindent Since this definition concerns the pointwise behavior of the likelihood, it presupposes a version which accords with Definition~\ref{defn:conventionversions}. The classical notion in (\ref{defn:con:eff:20}) is close to that used in the traditional statistical literature, although this literature usually restricts it to open $U$ containing the parameter $\psvar$, as this can be equivalently expressed in terms of weak convergence of measures (\cite[123-124]{Ghosal2017-jk}).

Our first main theorem is the following effectivization of Doob's Consistency Theorem  (\cite{Doob1949-hg}, \cite[Theorem 6.9]{Ghosal2017-jk}, \cite[Theorem 7.78]{Schervish2012-cn}).
\begin{restatable}[Effective Doob Consistency Theorem and Schnorr randomness]{thm}{thmdoob}\label{thm:doob}
Suppose there is a sequence of uniformly computable continuous functions ${f_n:\sspace\rightarrow \pspace}$ such that both of the following happen:
\begin{enumerate}[leftmargin=*]
    \item \label{thm:doob1} For all $(\psvar,\ssvar)$ in a $\prodm$-measure one subset of $\prodspace$, we have that $\lim_n f_n(\ssvar)$ exists and is equal to $\psvar$.
    \item \label{thm:doob2} There is a $\prior$-computable basis on $\pspace$ such that, for all $n\geq 0$ and all elements $U$ of the $\prior$-computable basis, the event $f_n^{-1}(U)$ is uniformly both c.e. open and effectively closed.
\end{enumerate}
Then all pairs $(\psvar, \ssvar)$ of parameters and samples in $\mathsf{SR}^{\prodm}$ are such that, for all c.e. open $U$ in $\pspace$ with $\prior(U)$ computable, one has that $\lim_n \poste{U}{\ssvarn{n}}=I_U(\psvar)$.
Further, all parameters $\psvar$ in $\mathsf{SR}^{\prior}$ are computably consistent.
\end{restatable}

In this, recall that a function between two computable Polish spaces is \emph{computable continuous} if the inverse image of c.e. opens is uniformly c.e. open.   For the notion of a $\prior$-computable basis in (\ref{thm:doob2}), see Definition~\ref{defn:nucomputablebasis}. For the moment, we just note some simple examples. If $\pspace$ is Cantor space or Baire space, then the basic clopens are a $\prior$-computable basis for any computable prior $\prior$. If $\pspace$ is $\mathbb{R}^n$ or $[0,1]^n$ then the open boxes with rational-valued corners are a $\prior$-computable basis for any atomless computable prior $\prior$.

The paradigmatic example of  conditions~(\ref{thm:doob1})-(\ref{thm:doob2}) of Theorem~\ref{thm:doob} is:
\begin{restatable}[Paradigmatic example: limiting relative frequencies]{prop}{proppara}\label{prop:para}
Let the parameter space $\pspace$ be the unit interval $[0,1]$, and let the sample space $\sspace$ be Cantor space $2^{\mathbb{N}}$. Let the prior $\prior$ be any computable probability measure on $\pspace$. Let $C$ be an effectively closed subset of $\pspace$ of $\prior$-measure one.
Let the likelihood $\pmap{\psvar}$, for $\psvar$ in $C$, be the measure on Cantor space which is the countable product of the measure on $\{0,1\}$ with $1$ being assigned probability~$\psvar$. For $\psvar$ not in $C$, let it be the zero element of $\mathcal{M}^+(\sspace)$.
Let $f_n:\sspace\rightarrow \pspace$ by $f_n(\ssvar)=\frac{1}{n}\sum_{i<n} \ssvar(i)$, so that $\lim_n f_n$ is the limiting relative frequency of $1$'s, if the limit exists.
Then conditions~(\ref{thm:doob1})-(\ref{thm:doob2}) of Theorem~\ref{thm:doob} are satisfied. 
\end{restatable}

We can also effectivize at least some cases of consistency via identifiability:

\begin{prop}[Effective identifiability]\label{thm:effectiveidentifiability}

Suppose that $\pspace$ is computably homeomorphic to a computable Polish subspace of $\mathbb{S}_{\infty}$ via the map $h:\pspace\rightarrow \mathbb{S}_{\infty}$.

Suppose that the sample space is Baire space $\mathbb{N}^{\mathbb{N}}$.

Suppose that the likelihood $\psvar\mapsto \pmap{\psvar}$ from $\pspace$ to $\mathcal{P}(\sspace)$ is given by taking $\pmap{\psvar}$ to be the countable product of $h(\psvar)$. 

Then all $\psvar$ in $\mathsf{SR}^{\prior}$ are computably consistent.
\end{prop}

See \S\ref{sec:simplex} for the definition of computable Polish subspace. Like with the classical proofs of consistency via identifiability (cf. \cite[Proposition 6.10]{Ghosal2017-jk}), our effectivization ultimately goes back through limiting relative frequencies.

While the proof of Theorem~\ref{thm:doob} uses lots of facts about Schnorr randomness, it is apriori not clear whether a weaker randomness notion suffices. For instance, there are many randomness notions properly larger than Schnorr randomness: how does one know that Theorem~\ref{thm:doob} does not hold of them?\footnote{It is known that some effective consistency-like notions in weaker settings, often called density notions, do not entail Schnorr randomness, since some of them fail the Strong Law of Large Numbers (cf. introduction to \cite{Bienvenu2014-cs}).} However, we can show that Theorem~\ref{thm:doob} does not hold of them, at least when we allow the sample space, the likelihood, and the functions $f_n:\sspace\rightarrow \pspace$ to vary. To this end, we first define:
\begin{defn}\label{defn:apttriple}
Suppose that $\prior$ is a computable probability measure on a computable Polish space $\pspace$. 

A triple of sample space $\sspace$, likelihood $\psvar\mapsto \pmap{\psvar}$, and uniformly computable continuous functions $f_n:\sspace\rightarrow \pspace$ is called an \emph{apt triple} if both of the following conditions hold:
\begin{enumerate}[label=(\alph*), ref=\alph*, leftmargin=*]
\item\label{thm:reversal:2a} The sample space $\sspace$ and likelihood $\psvar\mapsto \pmape{\sigma}{\psvar}$ satisfies conditions (\ref{assum1})-(\ref{assum2.5}) of Definition~\ref{defn:assum}.
\item\label{thm:reversal:2c} The uniformly computable continuous functions $f_n:\sspace\rightarrow \pspace$ satisfy conditions~(\ref{thm:doob1})-(\ref{thm:doob2}) of Theorem~\ref{thm:doob}.
\end{enumerate}
\end{defn}
Then we prove the following:
\begin{restatable}{thm}{thmreverse}\label{thm:reversal}
Suppose that $\prior$ is a computable probability measure on a computable Polish space $\pspace$. Suppose that there is at least one apt triple with sample space $\mathcal{X}$, likelihood $\psvar\mapsto \pmap{\psvar}$, and uniformly computable continuous functions $f_n:\sspace\rightarrow \pspace$. Then the following are equivalent for a parameter $\psvar_0$ in $\pspace$:
\begin{enumerate}[leftmargin=*]
    \item\label{thm:reversal:1} The parameter $\psvar_0$ is in $\mathsf{SR}^{\prior}$
    \item\label{thm:reversal:2} For all apt triples with sample space $\sspacealt$ and likelihood $\psvar\mapsto \pmapalt{\psvar}$ and uniformly computable continuous functions $g_n:\sspacealt\rightarrow \pspace$, one has that the parameter $\psvar_0$ is computably consistent relative to prior $\prior$, sample space $\sspacealt$, and likelihood $\psvar\mapsto \pmapalt{\psvar}$.
\end{enumerate}
\end{restatable}

This theorem applies in particular to the paradigmatic example of a computable prior on the unit interval (cf. Proposition~\ref{prop:para}). Hence this theorem provides many examples of parameters that are not computably consistent. However, as the proof in \S\ref{sec:reversal} makes clear, it does so by making $\pmapalt{\psvar_0}$ the zero element of $\mathcal{M}^+(\sspacealt)$, and it is not known to us whether there are other ways of causing the inconsistency of a non-random $\psvar_0$. Further, as suggested by the statement of (\ref{thm:reversal:2}), and as borne out by its proof in~\S\ref{sec:reversal}, Theorem~\ref{thm:reversal} produces inconsistency by greatly varying the sample space and likelihood. 

By contrast, Freedman produced examples of classical inconsistency by fixing a familiar parameter space and sample space and likelihood and by considering different priors on it. We effectivize Freedman's Inconsistency Theorem (\cite[Theorem 6.12]{Ghosal2017-jk}) as follows:

\begin{restatable}{thm}{thmfreedman}\label{thm:freedman} (Effective Freedman Inconsistency).

Let the parameter space $\pspace$ be the infinite dimensional simplex $\mathbb{S}_{\infty}$. Let the sample space $\sspace$ be Baire space $\mathbb{N}^{\mathbb{N}}$. Let the likelihood $\psvar\mapsto \pmap{\psvar}$ from $\pspace$ to Baire space be given by taking $\pmap{\psvar}$ to be the countable product of $\psvar$, as a probability measure on the natural numbers.

Then the set of points $( \prior,\psvar)$ in $\mathcal{P}(\pspace)\times \pspace$ which are not classically consistent is effectively comeager.
\end{restatable}

In this, effectively comeager means: it contains the countable intersection of a computable sequence of c.e. open sets, each of which is dense.\footnote{Hence, in Cantor space, a point is in all effectively comeager sets iff it is weakly 1-generic.}

By a combination of Theorem~\ref{thm:doob} and Proposition~\ref{prop:para}, we can further prove:

\begin{restatable}{cor}{corfreedman}\label{cor:freedman} 
Let $\pspace$ and $\sspace$ and $\psvar\mapsto \pmap{\psvar}$ be as in Theorem~\ref{thm:freedman}. In any non-empty c.e. open of $\mathcal{P}(\pspace)\times \pspace$ we can find a computable point $(\prior, \psvar_0)$ such that 
\begin{enumerate}[leftmargin=*]
    \item\label{cor:freedman:1} The parameter $\psvar_0$ is not  computably consistent.
    \item\label{cor:freedman:2} All parameters $\psvar$ in $\mathsf{SR}^{\prior}$ are computably consistent.
\end{enumerate}
\end{restatable}

This corollary provides a computability-theoretic solution to Diaconis and Freedman's problem of ``know[ing] for which [parameters] $\psvar$ the rule [Bayes' rule] is consistent'' (\cite[4]{Diaconis1986-us}). On the one hand, (\ref{cor:freedman:2}) says that a canonical kind of randomness satisfies an effective form of consistency. On the other hand, (\ref{cor:freedman:1}) says that the paradigmatic non-random parameters, namely the computable parameters, do not necessarily have this property. Hence a general solution to Diaconis and Freedman's problem, which is not given by classical analytic methods, is given by effective methods. It would be well worth sharpening  (\ref{cor:freedman:1}) to understand the extent to which other non-randoms do not suffice for consistency, but we leave this to future work.

The topic of this paper is closely related to that of Takahashi \cite{Takahashi2023-rl}. He was concerned with Martin-L\"of randomness in the case where the parameter space and the sample space are both Cantor space.\footnote{As Takahashi notes, this study was initiated by Vovk and V'Yugin \cite{Vovk1993-xw}. Another antecedent is Kjos-Hanssen \cite{Hanssen2010-bn}, who studies Martin-L\"of randomness in the setting of Proposition~\ref{prop:para}, albeit without a prior.} By contrast, we have focused on Schnorr randomness in the case where the parameter space is an arbitrary computable Polish space and the sample space is a closed subset of Baire space which carries the natural structure of a computable Polish space. Further, whereas we build the joint distribution from the prior and likelihood, Takahashi starts from the assumption of a computable probability measure $\prodm$ on $2^{\mathbb{N}}\times 2^{\mathbb{N}}$, defines the prior as the pushforward on the parameter space, and defines the likelihood $\pmape{\sigma}{\psvar}=\lim_n \prodm(\pspace\times [\sigma] \mid [\psvar_n]\times \sspace)$. This equation is true on the Schnorr randoms in our setting (cf. Proposition~\ref{prop:relationtotakahasi}), and so our results apply to Takahashi's setting. Since the Schnorr randoms are a superset of the Martin-L\"of randoms, our Theorem~\ref{thm:doob} thus improves Takahashi's results, in that we have identified more parameters at which consistency is achieved, albeit at the cost of more constructive constraints on the map from the sample space to the parameter space.\footnote{These more constructive constraints induce another contrast which bears brief mentioning. One can show in our setup that if $\theta$ is in $\mathsf{SR}^{\prior}$, then $\pmap{\psvar}$ is computable from any fast Cauchy sequence for $\theta$ (this follows from the Self-Location Lemma 3.3 of \cite{Huttegger2024-zb}). However, this need not happen with Takahashi's preferred version of likelihood (cf. \cite[\S{2.2}]{Takahashi2013-nt}, mentioned in \cite[Remark 3.3(3)]{Takahashi2023-rl}).}

This paper is independent of our prior paper \cite{Huttegger2024-zb} on L\'evy's Upward Theorem. The basic reason is that the map $f:\sspace\rightarrow \pspace$, given by $f=\lim_n f_n$ where $f_n:\sspace\rightarrow \pspace$ satisfies conditions~(\ref{thm:doob1})-(\ref{thm:doob2}) of Theorem~\ref{thm:doob}, can itself be quite complicated (cf. again the example in Proposition~\ref{prop:para}), and thus complicates any direct application of the results in \cite{Huttegger2024-zb}. In more detail: the classical proof of Doob's Consistency Theorem goes through applying L\'evy's Upward Theorem to the indicator function $I_{f^{-1}(U)}$ in $\sspace$ with respect to the measure $\prodm_{\prss}$ and the filtration generated by length $n$-strings, where $U$ is open in $\pspace$ (cf. the proof of Theorem~\ref{thm:doob} at close of \S\ref{sec:proofdoob}). Since $f$ is not necessarily computable continuous, the function $I_{f^{-1}(U)}$ is not necessarily lsc when $U$ is c.e. open. Hence any application of an effective version of L\'evy's Theorem will go through some version of $I_{f^{-1}(U)}$, which may behave differently than $I_{f^{-1}(U)}$ itself does on the relevant random points.\footnote{Even if one were to find a way around this issue of versions and find a way to directly apply effective versions of L\'evy's Upward Theorem, the bulk of the proof of Theorem~\ref{thm:doob} in \S\ref{sec:proofdoob} revolves around the proof of Proposition~\ref{prop:thiswashard}, which one would still have to do; and it depends on everything in \S\ref{sec:productspace}; and the reversal in \S\ref{sec:reversal} does not appear to come out of the characterizations of Schnorr randomness via effective versions of L\'evy's Upward Theorem.} Because of this impediment, we proceed by developing an alternative approach through $\Sigma^{0,\prior}_2$ classes and Propositions~\ref{prop:schnorrsigmazerotwonutest}, \ref{prop:schnorrsigma0nu2}, which in the proof of Theorem~\ref{thm:doob} we apply to the lsc functions $\poste{U}{\ssvarn{n}}$ (cf. Proposition~\ref{prop:posteriorintegral}, and the proof of Theorem \ref{thm:doob} in \S\ref{sec:proofdoob}). The virtue of this approach is that, when $U$ is c.e. open, the posteriors $\poste{U}{\ssvarn{n}}$, as quotients of integrals, are less complicated to define pointwise than the function $I_{f^{-1}(U)}$ is.

This paper is organized as follows. In \S\ref{sec:sigmazerotwo} we prove Propositions~\ref{prop:schnorrsigmazerotwonutest}-\ref{prop:schnorrsigma0nu2} and related facts about Schnorr randomness.  In section \S\ref{sec:productspace} we verify all the effective properties of the classical statistical setup. In \S\ref{sec:proofdoob} we prove Theorem~\ref{thm:doob}. In \S\ref{sec:paradigmatic} we prove Proposition~\ref{prop:para}. In \S\ref{sec:reversal} we prove Theorem~\ref{thm:reversal}. In \S\ref{sec:simplex} we take note of the computable Polish space structure on the infinite dimensional simplex and prove Proposition~\ref{thm:effectiveidentifiability}.  In \S\ref{sec:suffincon} we prove some elementary results in preparation for Effective Freedman Inconsistency. Finally, in  \S\ref{sec:efincon} we prove Theorem~\ref{thm:freedman} and Corollary~\ref{cor:freedman}.

%% file: 02-sigma02-schnorr.tex
\section{\texorpdfstring{$\Sigma^0_2$}{Sigma zero two} and Schnorr randomness}\label{sec:sigmazerotwo}

In this section, we suppose that $Y$ is a computable Polish space and that $\nu$ is a computable probability measure on $Y$. The latter means that $\nu(U)$ is uniformly left-c.e. as $U$ ranges over c.e. opens $U$. Equivalently, it means that $\nu$ is a computable point of the computable Polish space $\mathcal{P}(Y)$. 

We presuppose familiarity with algorithmic randomness as developed on arbitrary computable Polish spaces.\footnote{\cite{Gacs2005}, \cite{Hoyrup2009-pl}, \cite{Rut:2013}, \cite{Hoyrup2021-qj}, \cite{Huttegger2024-zb}. As regards notation, we follow the latter.} The only definition from the general theory which bears mentioning, since it can vary slightly between authors, is the following:\footnote{This is from \cite[Definition 1.1(10)]{Huttegger2024-zb}. For related definitions, see \cite[Section {2.3}]{Gacs2005}, \cite[Section {1.4}, 8]{Rut:2013}, \cite[Section {5.2}]{Hoyrup2009-pl}.}
\begin{defn}\label{defn:nucomputablebasis}
A \emph{$\nu$-computable basis} is a computable sequence of c.e. opens $U_0, U_1, \ldots$ such that any c.e. open can be written as an effective union of them and such that $\nu(U_i)$ is uniformly computable and they are paired with a computable sequence of effectively closed sets $C_0, C_1, \ldots$ such that $C_i\supseteq U_i$ and $\nu(C_i)=\nu(U_i)$.
\end{defn}

The notion of $\nu$-comptuable basis interacts well with another fundamental randomness notion:
\begin{defn}
A point $y$ of $Y$ is \emph{Kurtz $\nu$-random}, abbreviated $\mathsf{KR}^{\nu}$, if $y$ is in all c.e. open $U$ with $\nu(U)=1$.
\end{defn}
Note that if $U_i, C_i$ are as in the definition of $\nu$-computable basis, then $U_i, C_i$ necessarily agree on $\mathsf{KR}^{\nu}$.

By a result of Hoyrup and R\'ojas which goes through the Effective Baire Category Theorem, such a basis always exists, and indeed can be taken to be among the open balls with centers from the countable dense set with the closed balls being the associated effectively closed sets.\footnote{\cite[Lemma 5.1.1]{Hoyrup2009-pl}.}

The following proposition collects together some elementary facts about $\Sigma^{0,\nu}_2$ classes (cf. Definition~\ref{defn:defnsigmanuzerotwo}):
\begin{prop}\label{prop:sigma02nu}
(General results on $\Sigma^{0,\nu}_2$ classes).
\begin{enumerate}[leftmargin=*]
\item\label{prop:sigma02nu:0}  If $C$ is effectively closed with $\nu(C)$ computable, then $C$ is $\Sigma^{0,\nu}_2$.
\item\label{prop:sigma02nu:1}  If $U$ is c.e. open with $\nu(U)$ computable, then $U$ is $\Sigma^{0,\nu}_2$.
\item\label{prop:sigma02nu:1.5} The $\Sigma^{0,\nu}_2$ classes are closed under finite intersection; they are also closed under finite union and effective union provided the union has computable $\nu$-measure.
\item\label{prop:sigma02nu:2} If $V$ is a c.e. open set with $\nu(V)$ computable, then for any rational $\epsilon>0$ we can compute an index for an effectively closed $D\subseteq V$ such that $\nu(D)$ is computable and $\nu(V\setminus D)<\epsilon$.    
\item\label{prop:sigma02nu:3} If $C$ is an effectively closed set with $\nu(C)$ computable, then for any rational $\epsilon>0$ we can compute an index for a c.e. open $U\supseteq C$ such that $\nu(U)$ is computable and $\nu(U\setminus C)<\epsilon$.
\end{enumerate}
\end{prop}
\begin{proof}
For (\ref{prop:sigma02nu:0}), this by definition since one can just take $C_n=C$ for all $n\geq 0$.

For (\ref{prop:sigma02nu:1}), let us identify the countable dense set of $Y$ with natural numbers and let us use variables $i,j$ for this countable dense set, and let $d$ be the metric with respect to which $Y$ is a computable Polish space. Let $r_n>0$ be a computable sequence of reals such that $B(i,r_n)$ is a $\nu$-computable basis with $B[i,r_n]$ the associated effectively closed ball with the same $\nu$-measure.\footnote{See \cite[Corollary 5.2.1 p. 844]{Hoyrup2009-pl}, \cite[Theorem 2.2.1.2 p. 60]{Hoyrup2008-oj}. For more discussion and references, see \cite{Huttegger2024-zb}, immediately before Proposition 2.8.} It suffices to consider the case of $U$ an open ball $B(j,\epsilon)$ where $\epsilon>0$ is rational. In particular it suffices to show that we can compute an index for a c.e. set $I\subseteq \mathbb{N}\times \mathbb{N}$ such that  $B(j,\epsilon)=\bigcup_{(i,n)\in I} B[i,r_n]$. We define $I=\{(i,n)\in \mathbb{N}\times \mathbb{N}:  d(i,j)+r_n<\epsilon\}$. One can then check that one has the identity $B(j,\epsilon)=\bigcup_{(i,n)\in I} B[i,r_n]$.

\begin{proofdetail}
In more detail, to show that $B(j,\epsilon)=\bigcup_{(i,n)\in I} B[i,r_n]$, we show the two inclusions:
\begin{itemize}[leftmargin=*]
\item For the right-to-left, suppose that $(i,n)$ is in $I$ and $x$ is in the closed ball $B[i,r_n]$. Then $d(x,j)\leq d(x,i)+d(i,j)<r_n+\epsilon-r_n = \epsilon$, and so $x$ is in $B(j,\epsilon)$.
\item For the left-to-right direction, suppose $x$ is in $B(j,\epsilon)$. By the density of the $r_n$ sequence in $\mathbb{R}^{>0}$, choose $n\geq 0$ such that $r_n<\frac{1}{2}(\epsilon-d(x,j))$. Choose $i$ from the countable dense set in the open ball $B(x,r_n)$. Then $(i,n)$ is in $I$: for, we have $d(i,j)+r_n\leq d(i,x)+d(x,j)+r_n<2r_n+d(x,j)<\epsilon$.
\end{itemize} 
\end{proofdetail}

In the remainder of the proofs, we often appeal to the fact that if c.e. opens $U_0, U_1, \ldots$ have uniformly computable $\nu$-measure, then any event in the algebra generated by them also has uniformly computable $\nu$-measure(cf. \cite[Proposition 2.8(2)]{Huttegger2024-zb}).

For (\ref{prop:sigma02nu:1.5}), it suffices to show it for a single intersection, and in this case just consider the intersection $C_n\cap C_n^{\prime}$ where $C_n$ witnesses that $B$ is $\Sigma^{0,\nu}_2$ and $C_n^{\prime}$ witnesses that $B^{\prime}$ is $\Sigma^{0,\nu}_2$.

\begin{proofdetail}
In more detail, suppose that $B,B^{\prime}$ are $\Sigma^{0,\nu}_2$ classes with respective sequences $C_n, C_n^{\prime}$ of effectively closed sets. Without loss of generality, both of these sequences are increasing. Then $B\cap B^{\prime}=\bigcup_n (C_n\cap C_n^{\prime})$. Each of $C_n\cap C_n^{\prime}$ is effectively closed with computable $\nu$-measure. Given $k\geq 0$, compute $n_k\geq 0$ such that $\nu(B\setminus C_{n_k})<2^{-(k+1)}$ and $\nu(B^{\prime}\setminus C_{n_k}^{\prime})<2^{-(k+1)}$. Then one has that $\nu(B\cap B^{\prime})-\nu(C_{n_k}\cap C_{n_k}^{\prime})=\nu((B\cap B^{\prime})\setminus (C_{n_k}\cap C_{n_k}^{\prime}))\leq \nu(B\setminus C_{n_k})+\nu(B^{\prime}\setminus C_{n_k}^{\prime})<2^{-k}$. Hence $\nu(B\cap B^{\prime})$ is also computable.

To finish (\ref{prop:sigma02nu:1.5}), suppose that $B_i$ is a computable sequence of $\Sigma^{0,\nu}_2$ sets with respective sequences $C_{i,j}$ of effectively closed sets; and suppose that $B=\bigcup_i B_i$ has $\nu$-computable measure. Then simply note that $B=\bigcup_{i,j} C_{i,j}$.
\end{proofdetail}

For (\ref{prop:sigma02nu:2}), since $V$ is $\Sigma^{0,\nu}_2$, let $V=\bigcup_n D_n$, where $D_n$ is effectively closed with $\nu(D_n)$ uniformly computable. Without loss of generality $D_n$ is increasing. Hence, using computability of $\nu(V)$, we can compute $n$ such that $\nu(V\setminus D_n)<\epsilon$.

For (\ref{prop:sigma02nu:3}) simply apply (\ref{prop:sigma02nu:2}) to the relative complement. 
\begin{proofdetail}
In more detail, let $V:=Y\setminus C$. By (\ref{prop:sigma02nu:2}) we can compute an index for an effectively closed $D\subseteq V$ such that $\nu(D)$ is computable and $\nu(V\setminus D)<\epsilon$. Let $U:=Y\setminus D$. Then $U\supseteq C$ and $\nu(U\setminus C)=\nu(V\setminus D)<\epsilon$.
\end{proofdetail}
\end{proof}

The following is a technical but important strengthening of (\ref{prop:sigma02nu:3}):
\begin{prop}\label{prop:sigma02nu:4}
(Effectively covering a $\Sigma^{0,\nu}_2$ class).

If $B$ is $\Sigma^{0,\nu}_2$, then for any rational $\epsilon>0$ we can compute an index for a c.e. open $U\supseteq B$ such that $\nu(U)$ is computable and $\nu(U\setminus B)<\epsilon$.
\end{prop}
\begin{proof}
Let $B=\bigcup_n C_n$ where $C_n$ is uniformly effectively closed and $\nu(C_n)$ is computable; and where in addition $\nu(B)$ is computable. Without loss of generality $C_n$ is increasing. Define uniformly computable reals $0\leq \eta_n\leq 1$ by $\eta_0=\nu(C_0)$ and $\eta_n=\nu(C_n\setminus C_{n-1})$ for $n>0$. Since $\sum_n \eta_n = \nu(B)$, we have $\lim_n \eta_n=0$. Further $\nu(B\setminus C_m)=\sum_{n>m} \eta_n$. Let $\delta_n$ be a uniformly computable sequence of rationals with $0\leq \delta_n-\eta_n<2^{-(n+1)}$. Since $\delta_n<2^{-(n+1)}+\eta_n$, also $\lim_n \delta_n=0$. Further:
\begin{equation}\label{eqn:technprop:eqn:2}
\sum_{n>m} \delta_n=\sum_{n>m} \eta_n+\sum_{n>m} (\delta_n-\eta_n)<\nu(B\setminus C_m)+2^{-(m+1)}
\end{equation}
By Proposition~\ref{prop:sigma02nu}(\ref{prop:sigma02nu:3}), uniformly compute c.e. open $U_n\supseteq C_n$ with $\nu(U_n)$ computable and $\nu(U_n\setminus C_n)<\frac{\epsilon}{2}\cdot \delta_n$.  Let $U=\bigcup_n U_n$. This gives the required estimate
\begin{equation*}
\nu(U\setminus B)\leq \sum_n \nu(U_n\setminus C_n)< \sum_n \frac{\epsilon}{2}\cdot \delta_n \leq \sum_n \frac{\epsilon}{2}\cdot \eta_n + \sum_n \frac{\epsilon}{2} \cdot (\delta_n-\eta_n)\leq \epsilon
\end{equation*}
It remains to show that $\nu(U)$ is computable. Since $\nu(B)$ is computable and $\nu(U)=\nu(B)+\nu(U\setminus B)$, it suffices to show that $\nu(U\setminus B)$ is computable. For $n>0$: 
\begin{equation}\label{eqn:technprop:eqn:3}
\nu(U_n\setminus C_{n-1})\leq \nu(U_n\setminus C_n)+\nu(C_n\setminus C_{n-1})\leq \epsilon\cdot \delta_n+ \delta_n = (1+\epsilon)\cdot \delta_n
\end{equation}
For $m>0$, note that one has the inclusions 
\begin{equation}\label{eqn:technprop:aninclusion}
(\bigcup_{n\geq m} U_n)\setminus C_m\subseteq (U_m\setminus C_m)\cup \bigcup_{n>m} (U_n\setminus U_{n-1})\subseteq (U_m\setminus C_m) \cup \bigcup_{n>m} (U_n\setminus C_{n-1})
\end{equation}
where the last inclusion follows from $C_{n-1}\subseteq U_{n-1}$. Hence for all $m> 0$ one has that $\nu(U\setminus C_m)-\nu((\bigcup_{n< m} U_n)\setminus C_m)$ is
\begin{align*}
\leq \;&   \nu((\bigcup_{n\geq m} U_n)\setminus C_m)  \leq \nu(U_m\setminus C_m)+\sum_{n>m} \nu(U_n\setminus C_{n-1}) \\  \leq \;&  \nu(U_m\setminus C_m)+\sum_{n>m} (1+\epsilon)\cdot \delta_n  < \epsilon\cdot \delta_m+(1+\epsilon)\cdot \nu(B\setminus C_m) +(1+\epsilon)\cdot 2^{-(m+1)}
\end{align*}
where in the first line we apply (\ref{eqn:technprop:aninclusion}) and where in the last line we apply (\ref{eqn:technprop:eqn:3}) and then (\ref{eqn:technprop:eqn:2}). 
Then for all $m> 0$ we have that $\left|\nu(U\setminus B)-\nu((\bigcup_{n<m} U_n)\setminus C_m)\right|$ is:
\begin{align*}
& \leq \left|\nu(U\setminus B)-\nu(U\setminus C_m)\right|+\left| \nu(U\setminus C_m)-\nu((\bigcup_{n< m} U_n)\setminus C_m)\right| \\
\leq& \nu(B\setminus C_m)+\epsilon\cdot \delta_m+(1+\epsilon)\cdot \nu(B\setminus C_m) +(1+\epsilon)\cdot 2^{-(m+1)}
\end{align*}
Since the quantity in this last line is uniformly computable and goes to zero as $m$ goes to infinity, we can compute a subsequence which witnesses that $\nu(U\setminus B)$ is a computable real. We appeal here to the fact that $(\bigcup_{n<m} U_n)\setminus C_m$ has computable $\nu$-measure, since it is in the algebra generated by the uniformly computable sequence $U_n, Y\setminus C_n$ of c.e. opens with $\nu$-computable measure (cf. \cite[Proposition 2.8(2)]{Huttegger2024-zb}).
\end{proof}

From this, Proposition~\ref{prop:schnorrsigmazerotwonutest} easily follows. Proposition~\ref{prop:schnorrsigma0nu2} follows by going back and forth in the natural way between sequential Schnorr $\nu$-tests $U_n$ and the associated $\Sigma^{0,\nu}_2$-class $\bigcup_n (Y\setminus U_n)$ of $\nu$-measure one.

We can then derive the following characterization in terms of lsc functions:
\begin{prop}\label{prop:wr2propv2}
A point $y$ is in $\mathsf{SR}^{\nu}$ iff for every computable sequence of lsc functions $f_n:Y\rightarrow [0,\infty]$, with $\sup_{n\geq m} f_n$ an $L_1(\nu)$ Schnorr test uniformly in $m\geq 0$, and for every effectively closed $C$ with $\nu(C)$ computable and $f_n\rightarrow 0$ $\nu$-a.s. on $C$, one has that $y$ is in $C$ implies $\lim_n f_n(y)=0$.
\end{prop}
\begin{proof}
For the forward direction, suppose that $f_n$ and $C$ are as specified; we want to show that $\lim_n f_n=0$ on $\mathsf{SR}^{\nu}\cap C$. Suppose that $\epsilon>0$ is rational. For each $m\geq 0$, uniformly compute a real $\epsilon_{m}$ in the interval $(\frac{\epsilon}{2}, \epsilon)$ such that the c.e. open $\{y: \sup_{n\geq m} f_n(y)>\epsilon_{m}\}$ has uniformly computable $\nu$-measure.\footnote{We can do this by a result of Miyabe \cite{Miyabe2013-wy}; for this specific form, see \cite[Proposition 2.21, Lemma 2.22]{Huttegger2024-zb}.\label{footnote:miyabeish}} Then its relative complement $D_{m}: = \{y: \sup_{n\geq m} f_n(y)\leq \epsilon_{m}\}$ is effectively closed and has uniformly computable $\nu$-measure as well. Finally, using a $\nu$-computable basis, let $Y\setminus C=\bigcup_m U_m$, where $U_m$ is a computable sequence of c.e. opens from the $\nu$-computable basis, and let $C_m$ be effectively closed superset of $U_m$ with the same $\nu$-measure. Consider the class $V_{\epsilon} = (\bigcup_{m} (D_{m}\cap C)) \cup (\bigcup_m C_m)$, which is an effective union of effectively closed classes each with uniformly computable measure. Since by hypothesis $f_n\rightarrow 0$ $\nu$-a.s. on $C$, there are $\nu(C)$-many $y$ in $C$ such that there is $m\geq 0$ such that for all $n\geq m$ one has $f_n(y)<\frac{\epsilon}{2}$, and so $\sup_{n\geq m} f_n(y)\leq \frac{\epsilon}{2}<\epsilon_{m}$ and so this $y$ would be in $D_{m}$. Hence $\nu(V_{\epsilon})=1$, so that $V_{\epsilon}$ is a $\Sigma^{0,\nu}_2$ class of $\nu$-measure one.  Let $y$ in $\mathsf{SR}^{\nu}\cap C$. Then by Proposition~\ref{prop:schnorrsigma0nu2}, $y$ is in $V_{\epsilon}$. But since $y$ is in $\mathsf{KR}^{\nu}\cap C$, we cannot have that $y$ is in $\bigcup_m C_m$. (For, suppose that $y$ is $C_m$ for some $m\geq 0$. Since $C_m$ is an effectively closed superset of c.e. open $U_m$ with the same $\nu$-measure, it follows that $C_m\setminus U_m$ is effectively closed and $\nu(C_m\setminus U_m)=0$. Then since $y$ is in $\mathsf{KR}^{\nu}$, we have that $y$ is not in $C_m\setminus U_m$. Hence $y$ is in $U_m$ as well. But $U_m$ is by definition a subset of $Y \setminus C$). Thus rather we must have that $y$ is in $D_{m}\cap C$ for some $m\geq 0$, and so $\sup_{n\geq m} f_n(y)\leq \epsilon_{m}<\epsilon$. 

The backward direction is easy since if $U_n$ is a decreasing sequential Schnorr $\nu$-test and we set $f_n=I_{U_n}$, then $\sup_{n\geq m} f_n =I_{\bigcup_{n\geq m} U_n}$ is an $L_1(\nu)$ Schnorr test.
\end{proof}

Now we turn briefly to $L_1(\nu)$ computability (cf. \cite[\S{2.3}]{Huttegger2024-zb}):
\begin{prop}\label{prop:indicator} ($L_1(\nu)$ computability and indicator functions).
\begin{enumerate}[leftmargin=*]
    \item\label{prop:indicator:1} If $U$ is c.e. open with $\nu(U)$ computable, then $I_U$ is $L_1(\nu)$ computable.
    \item\label{prop:indicator:2} If $C$ is effectively closed with $\nu(C)$ computable, then $I_C$ is $L_1(\nu)$ computable.
    \item\label{prop:indicator:3} If $B$ is $\Sigma^{0,\nu}_2$ then $I_B$ is $L_1(\nu)$ computable.
\end{enumerate}
\end{prop}
\begin{proof}
For (\ref{prop:indicator:1}), suppose that $U=\bigcup_n V_n$, where $V_n$ comes from a $\nu$-computable basis. Then $\|I_U-I_{\bigcup_{m<n} V_m}\|_1 =\nu(U)-\nu(\bigcup_{m<n} V_m)$ is computable and goes to zero. Then we can compute sequence $n_{\ell}$ such that $\|I_U-I_{\bigcup_{m<n_{\ell}} V_m}\|_1\leq 2^{-\ell}$. Hence $I_U$ is $L_1(\nu)$ computable since $I_{\bigcup_{m<n_{\ell}} V_m}$ comes from the countable dense set.  The proof of (\ref{prop:indicator:2}) follows by taking complements; the proof of (\ref{prop:indicator:3}) is similar to that of~(\ref{prop:indicator:1}).

\begin{proofdetail}
In more detail:

 For (\ref{prop:indicator:2}), let $U=Y\setminus C$ so that $\nu(U)$ is also computable. By  (\ref{prop:indicator:1}) one has that $I_U$ is $L_1(\nu)$ computable. Since $I_C=1-I_U$, we have that $I_C$ is also $L_1(\nu)$ computable.

 For (\ref{prop:indicator:3}), suppose that $B=\bigcup_n C_n$, where $C_n$ is effectively closed with $\nu(B), \nu(C_n)$ uniformly computable. Without loss of generality, the $C_n$ sequence is increasing. Then $\|I_B-I_{C_n}\|_1 =\nu(B)-\nu(C_n)$, which is computable and goes to zero. Then we can compute a subsequence $C_{n_m}$ such that $\|I_B-I_{C_{n_m}}\|_1\leq 2^{-m}$, and then by (\ref{prop:indicator:2}) we have that $I_B$ is $L_1(\nu)$ computable.
    
\end{proofdetail}

\end{proof}

One also has the following, whose proof is very simple:
\begin{prop}\label{prop:productsL1nu}
($L_1(\nu)$-computability and products of bounded functions).

If $f,g$ are $L_1(\nu)$ computable and $\left|f\right|, \left|g\right|\leq c$ then $fg$ is $L_1(\nu)$-computable.
\end{prop}

The following result is one of the few pointwise results which we need:
\begin{prop}\label{prop:miyabeish}
Suppose $g,h$ are two $L_1(\nu)$ Schnorr tests. Suppose $f=g-h$ and suppose $f\geq 0$ $\nu$-a.s. 

Suppose that $a_n<b_n$ is a computable sequence of non-negative reals. Then: 
\begin{enumerate}[leftmargin=*]
    \item\label{prop:miyabeish:1} There is a computable sequence of computable reals $r_n$ in $(a_n, b_n)$ and a computable sequence of $\Sigma^{0,\nu}_2$ classes $B_n$ such that $B_n=f^{-1}(r_n,\infty]$ on~$\mathsf{SR}^{\nu}$.
    \item\label{prop:miyabeish:2} One has $g\geq h$ on~$\mathsf{SR}^{\nu}$.
    \item\label{prop:miyabeish:3} There is a computable sequence of computable reals $r_n$ in $(a_n, b_n)$ and a computable sequence of $\Sigma^{0,\nu}_2$ classes $B_n$ such that $B_n=f^{-1}[0,r_n)$ on~$\mathsf{SR}^{\nu}$.
\end{enumerate}
Moreover, the sequences $r_n$ in (\ref{prop:miyabeish:1}), (\ref{prop:miyabeish:3}) can be chosen to be the same. 

Finally, the result is uniform, in that if $g_i,h_i$ is a computable sequence of uniformly computable $L_1(\nu)$ Schnorr tests and we set $f_i=g_i-h_i$ and suppose $f_i\geq 0$ $\nu$-a.s., then we can find one and the same sequence $r_n$ which works for each~$i$.
\end{prop}

\begin{proof}
For (\ref{prop:miyabeish:1}), the three $\nu$-a.s. non-negative $L_1(\nu)$ computable functions $f,g,h$ are such that their pushforwards $f\#\nu, g\#\nu, h\#\nu$ are computable probability measures on $\mathbb{R}^{\geq 0}$. Using the Effective Baire Category Theorem, we can compute a uniformly computable sequence of reals $r_i$, dense in $\mathbb{R}^{>0}$, such that the events $f^{-1}(r_i,\infty]$ and  $g^{-1}(r_i,\infty]$ and $h^{-1}(r_i,\infty]$ have uniformly computable $\nu$-measure.\footnote{See footnote~\ref{footnote:miyabeish}. As one can see from the references there, one can further organize that $f^{-1}(r_i,\infty]$ has the same $\nu$-measure as $f^{-1}[r_i,\infty]$; this additional information gets used in (\ref{prop:miyabeish:3}).} 
By doing this again for the non-negative $L_1(\nu)$ computable functions $f+r_i,g+r_i,h+r_i$, we can compute a uniformly computable sequence of reals $s_k$, dense in $\mathbb{R}^{>0}$, such that the events
$(f+r_i)^{-1}(s_k,\infty]$ and  $(g+r_i)^{-1}(s_k,\infty]$ and $(h+r_i)^{-1}(s_k,\infty]$ have computable $\nu$-measure, uniformly in $i,k\geq 0$. Define c.e. index set $I=\{(j,k): s_k<r_j\}$. For $i\geq 0$ and $(j,k)$ in $I$, define $\Sigma^{0,\nu}_2$ class $B_{i,j,k}=g^{-1}(r_j,\infty]\cap (h+r_i)^{-1}[0,s_k]$, and then define $B_i:=\bigcup_{(j,k)\in I} B_{i,j,k}$. Then on $\mathsf{SR}^{\nu}$ we have the identity $f^{-1}(r_i,\infty]=B_i$.
\begin{proofdetail}
\begin{itemize}[leftmargin=*]
    \item Suppose $y$ in $\mathsf{SR}^{\nu}$ and $y$ is in $f^{-1}(r_i,\infty]$. Then $f(y)>r_i$. Then $g(y)-h(y)>r_i$. Then $g(y)>h(y)+r_i$. By  the density of the $r_j,s_k$ sequence on $\mathbb{R}^{>0}$, one has that there is $j,k\geq 0$ with $g(y)>r_j>s_k>h(y)+r_i$. Then $(j,k)$ is in $I$ and $g(y)>r_j$ and $s_k>h(y)+r_i$. Then also $s_k\geq h(y)+r_i$. Hence $y$ is in $B_{i,j,k}$ and so in $B_i$.
    \item Suppose that $y$ is in $\mathsf{SR}^{\nu}$ and $y$ is in $B_i$. Choose $(j,k)$ in $I$ with $y$ in $B_{i,j,k}$. Then $g(y)>r_j>s_k$ and $h(y)+r_i\leq s_k$. Then $g(y)>r_j>s_k\geq h(y)+r_i$. Then $g(y)-h(y)>r_i$. Then $f(y)>r_i$.
\end{itemize}    
\end{proofdetail}
Since $f^{-1}(r_i,\infty]$ has $\nu$-computable measure, so does $B_i$, and hence it too is a $\Sigma^{0,\nu}_2$ class by Proposition~\ref{prop:sigma02nu}(\ref{prop:sigma02nu:1.5}). Then we can finish by using the density of the $r_i$ sequence on $\mathbb{R}^{>0}$ to compute a subsequence $a_n<r_{i_n}<b_n$. 

For (\ref{prop:miyabeish:2}), fix $s>0$ to be rational. Let $b_n=s$ and let $a_n$ converge up to $s$. Let $g_n=g+s, h_n=h+s, f_n=f+s$. Apply part~(\ref{prop:miyabeish:1}) to $g_n, h_n, f_n$  and $(a_n, b_n)$ to get a computable sequence $r_n$ and $B_n$. Then $f\geq 0$ $\nu$-a.s. implies $f_n\geq s$ $\nu$-a.s. which implies that $f_n^{-1}(r_n, \infty]$ has $\nu$-measure one, which implies that $B_n$ has $\nu$-measure one. Then each $B_n$ has $\nu$-measure one, and so $\mathsf{SR}^{\nu}$ is a subset of each $B_n$ by Proposition~\ref{prop:schnorrsigma0nu2}. Then $f_n\geq r_n\geq a_n$ on $\mathsf{SR}^{\nu}$ for all $n\geq 0$, and so $f+s\geq s$ on on $\mathsf{SR}^{\nu}$ and hence $f\geq 0$ on $\mathsf{SR}^{\nu}$.

For (\ref{prop:miyabeish:3}), the argument is just like (\ref{prop:miyabeish:1}) but with the inequalities reversed. 

\begin{proofdetail}
In more detail, let $r_i, s_k$ be just as in the proof of (\ref{prop:miyabeish:1}). Define c.e. index set $I^{\prime}=\{(j,k): s_k>r_j\}$. For $i\geq 0$ and $(j,k)$ in $I^{\prime}$, define $\Sigma^{0,\nu}_2$ class $B_{i,j,k}^{\prime}=g^{-1}[0,r_j]\cap (h+r_i)^{-1}(s_k,\infty]$, and then define $B_i^{\prime}:=\bigcup_{(j,k)\in I^{\prime}} B_{i,j,k}^{\prime}$. Then on $\mathsf{SR}^{\nu}$ we have the identity $f^{-1}[0,r_i)=B_i^{\prime}$:
\begin{itemize}[leftmargin=*]
    \item Suppose $y$ in $\mathsf{SR}^{\nu}$ and $y$ is in $f^{-1}[0,r_i)$. Then $f(y)<r_i$. Then $g(y)-h(y)<r_i$. Then $g(y)<h(y)+r_i$. By  the density of the $r_j,s_k$ sequence on $\mathbb{R}^{>0}$, one has that there is $j,k\geq 0$ with $g(y)<r_j<s_k<h(y)+r_i$. Then $(j,k)$ is in $I^{\prime}$ and $g(y)<r_j$ and $s_k<h(y)+r_i$. Then also $g(y)\leq r_j$. Hence $y$ is in $B_{i,j,k}^{\prime}$ and so in $B_i^{\prime}$.
    \item Suppose that $y$ is in $\mathsf{SR}^{\nu}$ and $y$ is in $B_i^{\prime}$. Choose $(j,k)$ in $I^{\prime}$ with $y$ in $B_{i,j,k}^{\prime}$. Then $g(y)\leq r_j<s_k$ and $h(y)+r_i> s_k$. Then $g(y)\leq r_j<s_k<h(y)+r_i$. Then $g(y)-h(y)<r_i$. Then $f(y)<r_i$.
\end{itemize}
 Since $f^{-1}[0,r_i)$ has $\nu$-computable measure, so does $B_i^{\prime}$, and hence it too is a $\Sigma^{0,\nu}_2$ class by Proposition~\ref{prop:sigma02nu}(\ref{prop:sigma02nu:1.5}). Then we can finish by using the density of the $r_i$ sequence on $\mathbb{R}^{>0}$ to compute a subsequence $a_n<r_{i_n}<b_n$.   
\end{proofdetail}
\end{proof}

%% file: 04-productspace.tex
\section{Computable Polish space structure on the product}\label{sec:productspace}

We begin by showing:

\begin{prop}\label{prop:prodmcomputable} (The joint distribution is a computable probability measure).

$\prodm$ is a computable probability measure on $\prodspace$.
\end{prop}
\begin{proof}
We must show that for c.e. open $W\subseteq \prodspace$ we have that $\prodm(W)$ is left-c.e. We can write $W=\bigcup_i (U_i\times [\sigma_i])$, where $U_i\subseteq \pspace$ is a computable sequence from a $\prior$-computable basis with effectively closed $C_i\supseteq U_i$ with the same $\prior$-measure, and $\sigma_i$ is a computable sequence from $T$. It suffices to show that for all $n$, we have that $W_n=\bigcup_{i<n} (U_i\times [\sigma_i])$ has uniformly left-c.e. $\prodm$-measure. For this, in turn, it suffices to show the same of $W_{n,i}:=\big(U_i\times [\sigma_i]\big)\setminus \bigcup_{j<i} \big( U_j\times [\sigma_j] \big)$ for each $i<n$.

Fix $i<n$ and for $j<i$, define computable set $I_{i,j}\subseteq T$ so that $[\sigma_i]\setminus [\sigma_j]=\bigcup_{\eta \in I_{i,j}} [\eta]$. Since $I_{i,j}$ is computable let $I_{i,j,s}$ be a computable non-decreasing sequence such that $I_{i,j}=\bigcup_s I_{i,j,s}$ and $J_{i,j,s}:=I_{i,j,s+1}\setminus I_{i,j,s}$ has at most one element.

For each $\tau$ in $2^{<\mathbb{N}}$ with length $i>0$ and each $\varsigma$ in $\mathbb{N}^{<\mathbb{N}}$ of  length $i$, define the following c.e. open $U_{\tau}\subseteq \pspace$ and c.e. open $V_{\tau,\varsigma}\subseteq \sspace$, wherein $j$ ranges over $<i$:
\begin{equation*}
U_{\tau}  = \big(\bigcap_{\tau(j)=1} U_i\setminus C_j\big) \cap \big(\bigcap_{\tau(j)=0} U_i\cap U_j\big), \hspace{2mm} V_{\tau,\varsigma}  = \big(\bigcap_{\tau(j)=1} [\sigma_i]\big) \cap \big(\bigcap_{\tau(j)=0, \eta\in J_{i,j,\varsigma(j)}} [\eta]\big)
\end{equation*}
Let $I_i$ be the computable index set $\{(\tau, \varsigma)\in 2^i\times \mathbb{N}^i: \forall \; j<i \; (\tau(j)=1\rightarrow \varsigma(j)=0)\}$. Then on $\mathsf{KR}^{\prior}\times \sspace$ one has $W_{n,i}=\bigsqcup_{(\tau,\varsigma)\in I_i} (U_{\tau}\times V_{\tau,\varsigma})$, where the disjoint union is disjoint only on $\mathsf{KR}^{\prior}\times \sspace$. Then $\prodm(W_{n,i})=\sum_{(\tau,\varsigma)\in I_i} \prodm(U_{\tau}\times V_{\tau,\varsigma})$.

Hence, we have reduced the problem to showing that $\prodm(U\times [\sigma])$ is left-c.e. when $U\subseteq \pspace$ is c.e. open. Let $U=\bigcup_i U_i$, where $U_i$ is an increasing sequence of c.e. opens with uniformly $\prior$-computable measure. By Definition~\ref{defn:assum}(\ref{assum2.5}), one has that $\psvar\mapsto \pmape{\sigma}{\psvar}$ is uniformly $L_1(\prior)$ computable, and hence so is $\psvar\mapsto I_{U_i}(\psvar)\cdot \pmape{\sigma}{\psvar}$ by Proposition \ref{prop:indicator}(\ref{prop:indicator:1}) and Proposition~\ref{prop:productsL1nu}. Then $\prodm(U\times [\sigma])>q$ iff there is $i\geq 0$ with $\prodm(U_i\times [\sigma])>q$, which by definition happens iff $\int_{U_i} \pmape{\sigma}{\psvar} \; d\prior(\psvar)>q$. Since $\int_{U_i} \pmape{\sigma}{\psvar} \; d\prior(\psvar)$ is a computable real uniformly in $i\geq 0$, we are done and indeed $\prodm(U\times [\sigma])$ is left-c.e.
\end{proof}

  As the last paragraph makes clear, $\prodm$ is computable if $\int_{U_i} \pmape{\sigma}{\psvar} \; d\prior(\psvar)$ is merely left-c.e. rather than computable. We do not know of an interesting condition, for which this holds, which is broader than the $L_1(\prior)$ computability of the likelihood. But $L_1(\prior)$ computability does get us the following natural $\prodm$-computable basis:

\begin{prop}\label{prop:thebasisonprod}
(The $\prodm$-computable basis).

If $U\subseteq \pspace$ is c.e. open with $\prior(U)$ computable, then $\prodm(U\times [\sigma])$ is computable, uniformly in $\sigma$ from $T$.

This implies that $\{U_i\times [\sigma]:  i\in I, \sigma\in T\}$ is a $\prodm$-computable basis for any $\prior$-computable basis $\{U_i: i\in I\}$.

Further, if $C_i\supseteq U_i$ is the associated effectively closed superset with $C_i,U_i$ of same $\prior$-measure, then $C_i\times [\sigma]\supseteq U_i\times [\sigma]$ is effectively closed superset with $C_i\times [\sigma],U_i\times [\sigma]$ of same $\prodm$-measure.
\end{prop}
\begin{proof}
By the same considerations as at the end of the last proof.
\end{proof}

Since projection maps are computable continuous, we have:

\begin{rmk}
The pushforward $\prodm_{\prss}$ is a computable probability measure on the sample space $\sspace$. The $\prodm_{\prss}$-computable basis is just the basic clopens $[\sigma]$ for $\sigma$ in $T$. 
\end{rmk}

\begin{prop}\label{prop:elefact}
$\mathsf{KR}^{\prodm}\subseteq \mathsf{KR}^p\times \sspace$ and $\mathsf{KR}^{\prodm}\subseteq \pspace\times \mathsf{KR}^{\prodm_{\prss}}$, and likewise for Schnorr.
\end{prop}
Note that since $\mathsf{KR}^{\prodm}$ has $\prodm$-measure one, we have that $\mathsf{KR}^p\times \sspace$ has $\prodm$-measure one. Hence in proving Theorem~\ref{thm:doob}, we can assume that the $\prodm$-measure one set in question in condition (\ref{thm:doob1}) of Theorem~\ref{thm:doob} is a subset of  $\mathsf{KR}^p\times \sspace$.

We now discuss when the likelihoods are probability measures. This is a pointwise result and so we use the conventions on versions (cf. Definition~\ref{defn:conventionversions}). It can be seen as an  effective variant of classical results about regularity of conditional probability measures, which traditionally go through being a Borel subset of a Polish space (cf. \cite[Theorem B.32]{Schervish2012-cn}).

\begin{prop}\label{prop:likelihoodareprob} (The likelihoods and probability measures).

For all $\psvar_0$ in $\mathsf{SR}^{\prior}$ one has that $\pmap{\psvar_0}$ is in $\mathcal{P}(\sspace)$.
\end{prop}
\begin{proof}
We suppose that there is a uniformly computable sequence of $L_1(\prior)$ Schnorr tests $\gmape{\sigma}{\cdot}, \hmape{\sigma}{\cdot}$ such that $\pmape{\sigma}{\cdot}=\gmape{\sigma}{\cdot}-\hmape{\sigma}{\cdot}$ pointwise. By Proposition~\ref{prop:miyabeish}(\ref{prop:miyabeish:2}) we have $\gmape{\sigma}{\cdot}\geq \hmape{\sigma}{\cdot}$ on $\mathsf{SR}^{\prior}$.

Where $\emptyset$ is the root of the tree $T$, we show that $\pmape{\emptyset}{\psvar}=1$ for $\psvar$ in $\mathsf{SR}^{\prior}$:
\begin{itemize}[leftmargin=*]
    \item Suppose $0<s<1$ is rational. By Proposition~\ref{prop:miyabeish}, choose a computable $r$ in $(s,1)$, and a $\Sigma^{0,\prior}_2$ set $B$ which is equal to $\{\psvar: \pmape{\emptyset}{\psvar}>r\}$ on $\mathsf{SR}^{\prior}$. Then $\pmape{\emptyset}{\psvar}> r>s$ for $\psvar$ in $\mathsf{SR}^{\prior}$ by Proposition~\ref{prop:schnorrsigma0nu2}. Since this holds for all rational $s$ in $(0,1)$ we have $\pmape{\emptyset}{\psvar}\geq 1$ for $\psvar$ in $\mathsf{SR}^{\prior}$.
\item Suppose $s>1$ is rational. By Proposition~\ref{prop:miyabeish}, choose a computable $r$ in $(1,s)$ and a  $\Sigma^{0,\prior}_2$ set $B^{\prime}$ which is equal to $\{\psvar: \pmape{\emptyset}{\psvar}>r\}$ on $\mathsf{SR}^{\prior}$. But then $B^{\prime}$ is $\prior$-null, and no element of $\mathsf{KR}^{\prior}$ can be in a $\Sigma^{0,\prior}_2$ $\prior$-null set. Since this holds for all rational $s>1$, we have $\pmape{\emptyset}{\psvar}\leq 1$ for $\psvar$ in $\mathsf{SR}^{\prior}$.
\end{itemize}

Suppose that $\sigma$ in $T$ is fixed; enumerate its immediate successors in $T$ as $\tau_0, \tau_1, \ldots$. We may suppose by induction that we know $0\leq \pmape{\sigma}{\psvar}\leq 1$ on $\psvar$ in $\mathsf{SR}^{\prior}$. We want to show that $\pmape{\sigma}{\psvar}=\sum_i \pmape{\tau_i}{\psvar}$ for all $\psvar$ in $\mathsf{SR}^{\prior}$:
\begin{itemize}[leftmargin=*]
    \item First we show that for all $\psvar$ in $\mathsf{SR}^{\prior}$, we have $\pmape{\sigma}{\psvar}\leq \sum_i \pmape{\tau_i}{\psvar}$. By Proposition~\ref{prop:miyabeish}, choose dense computable $r_n$ sequence in $(0,1)$ and a computable sequence of $\Sigma^{0,\prior}_2$ sets $B_{n,j}$ which are equal to $\{\psvar: \pmape{\sigma}{\psvar}> r_n> \sum_{i\leq j} \pmape{\tau_i}{\psvar}\}$ on $\mathsf{SR}^{\prior}$. Fix $n\geq 0$ and set $B_n:=\bigcap_j B_{n,j}$. Since $B_n$ is $\prior$-null by our standing effective assumptions (in particular by the second conjunct of Definition~\ref{defn:assum}(\ref{assum2.5})), we can compute a subsequence $B_{n,j_m}$ with $\nu(B_{n,j_m})\leq 2^{-m}$ for all $m\geq 0$. By Proposition~\ref{prop:schnorrsigmazerotwonutest}, no element of $\mathsf{SR}^{\prior}$ is in $B_n$. Since this holds for all $n\geq 0$ and since $r_n$ is dense in $(0,1)$, we are done.
    \item Second we show that for all $\psvar$ in $\mathsf{SR}^{\prior}$, we have $\pmape{\sigma}{\psvar}\geq \sum_i \pmape{\tau_i}{\psvar}$. Fix $j>0$. By Proposition~\ref{prop:miyabeish}, choose dense computable $r_n$ sequence in $(0,1)$ and a computable sequence of $\Sigma^{0,\prior}_2$ sets $B_n$ which are equal to $\{\psvar: \pmape{\sigma}{\psvar}< r_n< \sum_{i\leq j} \pmape{\tau_i}{\psvar}\}$ on $\mathsf{SR}^{\prior}$. But $B:=\bigcup_n B_n$ is $\prior$-null by our standing effective assumptions (in particular by the second conjunct of Definition~\ref{defn:assum}(\ref{assum2.5})), and hence by Proposition~\ref{prop:sigma02nu}(\ref{prop:sigma02nu:1.5}), $B$ is a $\Sigma^{0,\prior}_2$ $\prior$-null class.  No element of $\mathsf{KR}^{\prior}$ can be in a $\Sigma^{0,\prior}_2$ $\prior$-null set. Since this holds for all $j>0$, we have that $\pmape{\sigma}{\psvar}\geq \sum_i \pmape{\tau_i}{\psvar}$ for all $\psvar$ in $\mathsf{SR}^{\prior}$.
\end{itemize}
\end{proof}

There is a similar but much easier result for posteriors, whose proof we omit:
\begin{prop}\label{prop:whenpostprob2} (The posterior and probability measures).

One has that $\post{\ssvarn{n}}$ is in $\mathcal{P}(\pspace)$ for all $\ssvar$ in $\mathsf{KR}^{\prodm_{\prss}}$ and all $n\geq 0$.
\end{prop}

\begin{proofdetail}
\begin{proof}
 Let $I_n=\{\sigma\in T: \left|\sigma\right|=n \; \& \; \int_{\pspace} \pmape{\sigma}{\psvar} \; d\prior(\psvar)=0\}$. Note that $I_n=\{\sigma\in T: \left|\sigma\right|=n \; \& \;\prodm(\pspace\times [\sigma])=0\}$. Suppose that $\psvar$ in $\mathsf{KR}^{\prodm_{\prss}}$. If $\post{\ssvarn{n}}$ is not in $\mathcal{P}(\pspace)$, then $\ssvar$ in $[\sigma]$ for some $\sigma$ in $I_n$ and so $\prodm_{\prss}([\sigma])=0$. But elements of $\mathsf{KR}^{\prodm_{\prss}}$ cannot be in $\prodm_{\prss}$-null effectively closed sets. 
 \end{proof}
\end{proofdetail}

We note that the posterior is an integral test:
\begin{prop}\label{prop:posteriorintegral} (The posterior and integral tests).

If $U\subseteq \pspace$ is c.e. open, then $x\mapsto \poste{U}{\ssvarn{n}}$ is lsc and $\leq 1$, uniformly in $n\geq 0$.
\end{prop}
\begin{proof}
For rational $q>0$, one has $\poste{U}{\ssvarn{n}}>q$ iff there is $\sigma$ in $T$ of length $n$ such that $\ssvar$ is in $[\sigma]$ and there is rational $r>0$ with $\prodm(U\times [\sigma])>r$ and $0<\prodm(\pspace\times [\sigma])\leq \frac{r}{q}$. This is a c.e. open condition in variable $\ssvar$, since $U\times [\sigma]$ is c.e. open in $\prodspace$ and since $\pspace\times [\sigma]$ is both c.e. open and effectively closed in $\prodspace$.  
\end{proof}

Recall the notation from e.g. the Fubini-Tonelli Theorems: if $B\subseteq \prodspace$ and $\psvar$ in $\pspace$, then the $\psvar$-section $B^{\psvar}$ of $B$ is defined as $B^{\psvar}:=\{\ssvar\in \sspace : (\psvar, \ssvar)\in B\}$. We have:
\begin{prop}\label{prop:improve}
(Parameter space sections of the likelihood, evaluated at c.e. opens with computable measure, are $L_1(\prior)$ computable).
\begin{enumerate}[leftmargin=*]
\item \label{prop:improve:1} The map $\psvar\mapsto \pmape{U^{\psvar}}{\psvar}$ is $L_1(\prior)$ computable, for c.e. open $U\subseteq \prodspace$ with $\prodm(U)$ computable.
\item \label{prop:improve:2} Any version of it as a difference of $L_1(\prior)$ Schnorr tests agrees on points $\psvar$ in $\mathsf{SR}^{\prior}$ with the value which the element $\pmap{\psvar}$ of $\mathcal{P}(\sspace)$ assigns to the event $U^{\psvar}$.
\end{enumerate}
\end{prop}
In (\ref{prop:improve:2}), we are referring to the probability measure $\pmap{\psvar}$ from Proposition~\ref{prop:likelihoodareprob}, which as established by that proposition is a probability measure when $\psvar$ is in $\mathsf{SR}^{\prior}$.
\begin{proof}
We first prove that (\ref{prop:improve:1})-(\ref{prop:improve:2}) hold for c.e. opens $V$ in $S$ which are finite unions $\bigcup_{i\leq n} (W_i\times [\sigma_i])$ of elements from the $\prodm$-computable basis. For each $i\leq n$, choose $C_i\supseteq W_i$ effectively closed with the same $\prior$-measure as $W_i$. For each $J\subseteq [0,n]$ let $J^{\prime}=\{i\in J: \forall \; j\in J\; \sigma_j\nprec \sigma_i\}$ and 
define the following events where $i$ in the subscript ranges over $\leq n$:
\begin{equation*}
A_J=\bigcap_{i\in J} W_i\cap \bigcap_{i\notin J} \pspace\setminus W_i, \hspace{5mm} B_J=\bigcap_{i\in J} W_i\cap \bigcap_{i\notin J} \pspace\setminus C_i
\end{equation*}
Then $\pmape{V^{\psvar}}{\psvar}=\sum_{J\subseteq [0,n]} \sum_{i\in J^{\prime}} I_{A_J}(\psvar) \cdot \pmape{\sigma_i}{\psvar}$. By Proposition \ref{prop:indicator} and Proposition~\ref{prop:productsL1nu} this is a finite sum of $L_1(\prior)$ computable functions, since $A_J$ comes from the algebra generated by the $\prior$-computable basis and since $\psvar\mapsto \pmape{\sigma}{\psvar}$ is $L_1(\prior)$ computable uniformly in $\sigma$ from $T$. This takes care of (\ref{prop:improve:1}). For (\ref{prop:improve:2}), it suffices to find one such version, since all such versions agree on the Schnorr randoms. Suppose that $\pmape{\sigma}{\cdot}$ uses the version which is the difference of $\gmape{\sigma}{\cdot}$ and $ \hmape{\sigma}{\cdot}$, where these are $L_1(\prior)$ Schnorr tests, uniformly in $\sigma$ from $T$. Then to finish (\ref{prop:improve:2}), we note that $\pmape{V^{\psvar}}{\psvar}$ has a version $\fmape{V}{\psvar}$ which is the difference of 
\begin{equation}\label{eqn:gmape}
\gmape{V}{\psvar}:=\sum_{J\subseteq [0,n]} \sum_{i\in J^{\prime}} I_{B_J}(\psvar) \cdot \gmape{\sigma_i}{\psvar}, \hspace{5mm} \hmape{V}{\psvar}:=\sum_{J\subseteq [0,n]} \sum_{i\in J^{\prime}} I_{B_J}(\psvar) \cdot \hmape{\sigma_i}{\psvar}
\end{equation}

Second, we treat the general case. Let $U=\bigcup_m U_m$, where $U_m$ comes from the $\prodm$-computable basis. Let $V_{\ell}=\bigcup_{m\leq \ell} U_{m}$. For each $n\geq 0$, compute $\ell_n\geq 0$ such that $\prodm(U\setminus V_{\ell_n})<2^{-n}$. One has that $\| \pmape{U^{\cdot}}{\cdot} - \pmape{V_{\ell_n}^{\cdot}}{\cdot} \|_{L_1(\prior)}$ is equal to
\begin{equation*}
 \int_{\pspace} \int_{\sspace} I_U(\psvar,\ssvar)-I_{V_{\ell_n}}(\psvar,\ssvar) \; d\pmap{\psvar}(\ssvar) \; d\prior(\psvar) = \prodm(U\setminus V_{\ell_n})< 2^{-n} 
\end{equation*}
By (\ref{prop:improve:1}) applied to $V_{\ell_n}$, we have that $\psvar\mapsto \pmape{V_{\ell_n}}{\psvar}$ is $L_1(\prior)$ computable; and we just showed that these converge fast to $\psvar\mapsto \pmape{U}{\psvar}$ in $L_1(\prior)$; and so we are done with (\ref{prop:improve:1}). For (\ref{prop:improve:2}), let $\psvar\mapsto \pmape{U^{\psvar}}{\psvar}$ have a version $f$ which is a difference of $L_1(\prior)$ Schnorr tests $g,h$. By (\ref{prop:improve:2}) applied to $V_i$, we have that $\fmape{V_i}{\psvar}$ is increasing and its limit exists in $[0,1]$, where $\fmape{V_i}{\psvar}$ is the difference of $\gmape{V_i}{\psvar}, \hmape{V_i}{\psvar}$ from (\ref{eqn:gmape}).  It suffices to show the identity $f(\psvar)=\lim_i \fmape{V_i}{\psvar}$ for $\psvar$ in $\mathsf{SR}^{\prior}$. The proof of this identity, via the two inqualities, follows as in the proof of Proposition~\ref{prop:likelihoodareprob}.

\begin{proofdetail}
In more detail:
\begin{itemize}[leftmargin=*]
    \item We show that for all $\psvar$ in $\mathsf{SR}^{\prior}$, we have $f(\psvar)\leq \lim_i \fmape{V_n}{\psvar}$. By Proposition~\ref{prop:miyabeish}, choose dense computable $r_n$ sequence in $(0,1)$ and a computable sequence of $\Sigma^{0,\prior}_2$ sets $B_{n,j}$ which are equal to $\{\psvar: f(\psvar)> r_n>  \fmape{V_j}{\psvar}\}$ on $\mathsf{SR}^{\prior}$. Fix $n\geq 0$ and set $B_n:=\bigcap_j B_{n,j}$. Since $B_n$ is $\prior$-null, we can compute a subsequence $B_{n,j_m}$ with $\nu(B_{n,j_m})\leq 2^{-m}$ for all $m\geq 0$. By Proposition~\ref{prop:schnorrsigmazerotwonutest}, we have that no element of $\mathsf{SR}^{\prior}$ is in $B_n$. Since this holds for all $n\geq 0$ and since $r_n$ is dense in $(0,1)$, we are done.
    \item Second we show that for all $\psvar$ in $\mathsf{SR}^{\prior}$, we have $f(\psvar)\geq \lim_i \fmape{V_i}{\psvar}$. Fix $j>0$. By Proposition~\ref{prop:miyabeish}, choose dense computable $r_n$ sequence in $(0,1)$ and a computable sequence of $\Sigma^{0,\prior}_2$ sets $B_n$ which are equal to $\{\psvar: \fmape{V}{\psvar}< r_n<  f(\psvar)\}$ on $\mathsf{SR}^{\prior}$. But $B:=\bigcup_n B_n$ is $\prior$-null, and hence by Proposition~\ref{prop:sigma02nu}(\ref{prop:sigma02nu:1.5}), $B$ is a $\Sigma^{0,\prior}_2$ $\prior$-null class.  No element of $\mathsf{KR}^{\prior}$ can be in a $\Sigma^{0,\prior}_2$ $\prior$-null set. Since this holds for all $j>0$, we have that $f(\psvar)\geq \lim_i \fmape{V_i}{\psvar}$ for all $\psvar$ in $\mathsf{SR}^{\prior}$.
\end{itemize}
\end{proofdetail}

\end{proof}

We can use this to prove the next proposition. See Takahashi \cite[Theorem 3.1-3.2]{Takahashi2023-rl} for the analogue in his setting, and for precedents in Vovk and V'Yugin. As Takahashi notes, these results are reminiscent of van Lambalgen's Theorem. So as to avoid further notation, we have not stated them in terms of oracle computation. Rute \cite[Theorem 8.2]{Rute2018-nd} proves a similar result for Schnorr randomness when, in our terminology and setting, the likelihood map is computable continuous from $\pspace$ to $\mathcal{M}^+(\sspace)$.
\begin{prop}\label{prop:important} (Relation between randomness on prior and joint distribution).

If $\psvar$ in $\mathsf{SR}^{\prior}$, then $(\psvar,\ssvar)$ is in $\mathsf{SR}^{\prodm}$ for $\pmap{\psvar}$-a.s. many $\ssvar$ in $\sspace$. 

\end{prop}
\begin{proof}
Fix $\psvar$ in  $\mathsf{SR}^{\prior}$. Suppose $(\psvar,\ssvar)$ is not in $\mathsf{SR}^{\prodm}$. Choose a sequential Schnorr $\prodm$-test $U_n$ with $(\psvar,\ssvar)$ in $\bigcap_n U_n$. By Proposition~\ref{prop:improve} the map $\psvar^{\prime}\mapsto \pmape{U_n^{\psvar^{\prime}}}{\psvar^{\prime}}$ is $L_1(\prior)$ computable, uniformly in $n\geq 0$. By Proposition~\ref{prop:miyabeish}, there is a computable sequence of computable reals $\eta_n$ in the interval $(2^{-n}, 2^{-n+1})$ and a computable sequence of $\Sigma^{0,\prior}_2$ classes $B_n$ such that $B_n=\{\psvar^{\prime}: \pmape{U_{2n}^{\psvar^{\prime}}}{\psvar^{\prime}}>\eta_n\}$ on $\mathsf{SR}^{\prior}$. 
Then:
\begin{equation*}
2^{-2n}\geq \prodm(U_{2n}) = \int_{\pspace} \pmape{U_{2n}^{\psvar^{\prime}}}{\psvar^{\prime}} \; dp(\psvar^{\prime}) \geq \int_{B_n} \pmape{U_{2n}^{\psvar^{\prime}}}{\psvar^{\prime}} \; dp(\psvar^{\prime}) \geq 2^{-n} \prior(B_n)
\end{equation*}
By Proposition~\ref{prop:schnorrsigmazerotwonutest} and by $\psvar$ being in  $\mathsf{SR}^{\prior}$, there is $n_0> 0$ such that for all $n\geq n_0$ one has $\pmape{U_{2n}^{\psvar}}{\psvar} \leq 2^{-n+1}$. Then by Proposition~\ref{prop:improve}(\ref{prop:improve:2}), the sample $\ssvar$ is in the  $\pmap{\psvar}$-null set $\bigcap_{n\geq n_0} U_{2n}^{\psvar}$. With $\psvar$ fixed, there are only countably many such null sets for each Schnorr $\prodm$-test $U_n$. And since there are only countably many Schnorr $\prodm$-tests, there are only $\pmap{\psvar}$-null many $\ssvar$ such that $(\psvar,\ssvar)$ is not in $\mathsf{SR}^{\prodm}$.
\end{proof}

Finally, as mentioned in \S\ref{sec:intro}, in the case where the parameter space is Cantor space, one has that our version of the likelihood (Definition~\ref{defn:conventionversions}) agrees with Takahashi's preferred version of the likelihood (\cite[Definition 2.2]{Takahashi2023-rl}):
\begin{prop}\label{prop:relationtotakahasi}
If $\pspace=2^{\mathbb{N}}$ and $\psvar$ in $\mathsf{SR}^{\prior}$, one has $\lim_n \prodm(\pspace\times [\sigma] \mid [\psvar_n]\times \sspace)=\pmape{\sigma}{\psvar}$ for all $\sigma$ in $T$.
\end{prop}
\begin{proof}
This follows from the identity $\prodm(\pspace\times [\sigma] \mid [\psvar_n]\times \sspace)=\frac{1}{\prior([\psvar_n])} \int_{[\psvar_n]} \pmape{\sigma}{\psvar^{\prime}}\; d\prior(\psvar^{\prime})$ and the effective version of L\'evy's Upward Theorem (\cite{Huttegger2024-zb}; see discussion therein of \cite{Rute2012aa}, \cite{Pathak2014aa}). 
\end{proof}

%% file: 05-proof-doob.tex
\section{Proof of Theorem~\ref{thm:doob}}\label{sec:proofdoob}

We first note the following elementary result:

\begin{prop}\label{prop:switcheroo}
Suppose $f,f_n:\sspace\rightarrow \pspace$ are Borel measurable. Suppose that $f:\sspace\rightarrow \pspace$ is such that $f(\ssvar)=\psvar$ for all $(\psvar,\ssvar)$ in a $\prodm$-measure one set. Then:
\begin{enumerate}[leftmargin=*]
    \item\label{prop:switcheroo:1} For all Borel $A\subseteq \pspace$ and $B\subseteq \sspace$, one has $A\times B=\pspace\times (f^{-1}(A)\cap B)$ $\prodm$-a.s.
    \item\label{prop:switcheroo:2} For all Borel $A\subseteq \pspace$ one has $\poste{A}{\sigma} = \prodm_{\prss}(f^{-1}(A)\mid \sigma)$ if $\prodm_{\prss}(\sigma)>0$.
    \item\label{prop:switcheroo:3} If $\lim_n (f_n\circ \prss)=(f\circ \prss)$ $\prodm$-a.s., then $\lim_n f_n=f$ $\prodm_{\prss}$-a.s.
\end{enumerate}
\end{prop}
 \noindent In (\ref{prop:switcheroo:3}), the function $\prss:\prodspace\rightarrow \sspace$ is again the projection onto the sample space (cf. diagram in \S\ref{sec:intro}).  Further, in~(\ref{prop:switcheroo:3}), if the $\prodm$-measure one event in the antecedent is Borel then the $\prodm$-measure one event in the consequent can be chosen to be Borel using the proof of the universal measurability of analytic sets (\cite[Theorem 21.10]{Kechris1995-hr}).

The following is the key proposition used in the proof of Theorem~\ref{thm:doob}:
\begin{prop}\label{prop:thiswashard}
Suppose that the prior $p$ is computable. 

Suppose there is a sequence of uniformly computable continuous functions $f_n:\sspace\rightarrow \pspace$ satisfying conditions~(\ref{thm:doob1})-(\ref{thm:doob2}) of Theorem~\ref{thm:doob}. Let $f=\lim_n f_n$ when it exists; and zero otherwise. 

Let $U\subseteq \pspace$ be c.e. open with $\prior(U)$ computable.

Then $I_{f^{-1}(U)}$ is a computable point of $L_1(\prodm_{\prss})$. 

Further, for all rational $\epsilon>0$, one can compute an index for an element $V\subseteq U$ of the $\prior$-computable basis and an $n\geq 0$ such that $\|I_{f^{-1}(U)}-I_{f_n^{-1}(V)}\|_{L_1(\prodm_{\prss})}<\epsilon$, which is the same as $\prodm_{\prss}(f^{-1}(U)\triangle f^{-1}_n(V))<\epsilon$.
\end{prop}
\begin{proof}
First suppose that $V\subseteq \pspace$ is an element of the $\prior$-computable basis, with effectively closed $C\supseteq V$ such that $V,C$ have same $\prior$ measure. We first show that  $I_{f^{-1}(V)}$ is a computable point of $L_1(\prodm_{\prss})$.

Let $D\subseteq \pspace\times \sspace$ be a $\prodm$-measure one subset of $\mathsf{KR}^{\prior}\times \sspace$ such that for $(\psvar, \ssvar)$ in $D$ we have $\lim_n f_n(\ssvar)$ exists and is equal to $\psvar$ (cf. remarks immediately after Proposition~\ref{prop:elefact}).

Let us show that $\lim_n (I_{f^{-1}_n(V)}\circ \prss)=I_{f^{-1}(V)}\circ \prss$ $\prodm$-a.s. It suffices to show that it happens on $D$. Suppose that $(\psvar, \ssvar)$ is in $D$. 
\begin{itemize}[leftmargin=*]
\item First suppose that $I_{f^{-1}(V)}(\ssvar)=1$. Then $f(\ssvar)$ in $V$. Since $V$ is open and we are on $D$, eventually $f_n(\ssvar)$ in $V$, which implies that eventually $I_{f^{-1}_n(V)}(\ssvar)=1$.
\item Second suppose that $I_{f^{-1}(V)}(\ssvar)=0$. Then $f(\ssvar)$ is in $\pspace\setminus V$. Since we are on $D\subseteq \mathsf{KR}^{\prior}\times \sspace$, we have that $\psvar=f(\ssvar)$ is in the c.e. open $\pspace \setminus C$. Suppose that there were infinitely many $n\geq 0$ with $I_{f_n^{-1}(V)}(\ssvar)=1$, so that there were infinitely many $n\geq 0$ with $f_n(\ssvar)$ in $V$. Since we are on $D$, we then have that $f(\ssvar)$ is in the topological closure $\overline{V}$ of $V$, and hence $f(\ssvar)$ in $C$, since $C$ is an effectively closed superset of $V$; a contradiction. Hence, rather there are only finitely many $n\geq 0$ with $I_{f_n^{-1}(V)}(\ssvar)=1$, and hence $\lim_n f_n(\ssvar)=0$, which is what we wanted to show.
\end{itemize}
Since $\lim_n (I_{f^{-1}_n(V)}\circ \prss)=I_{f^{-1}(V)}\circ \prss$ $\prodm$-a.s., by Proposition~\ref{prop:switcheroo}(\ref{prop:switcheroo:3}) we have that $\lim_n I_{f^{-1}_n(V)}=I_{f^{-1}(V)}$ $\prodm_{\prss}$-a.s. By the Dominated Convergence Theorem, $I_{f^{-1}_n(V)}\rightarrow I_{f^{-1}(V)}$ in $L_1(\prodm_{\prss})$. 

We argue that $I_{f^{-1}_n(V)}$ is uniformly a computable point of $L_1(\prodm_{\prss})$. Since $f_n^{-1}(V)$ is c.e. open we have that $I_{f^{-1}_n(V)}$ is lsc. It further has computable $\prodm_{\prss}$ expectation. For: by our hypothesis on the $\prior$-computable basis, we have that $f_n^{-1}(V)$ is both c.e. open and effectively closed, so that $\mathbb{E}_{\prodm_{\prss}} I_{f_n^{-1}(V)} = \prodm_{\prss}(f_n^{-1}(V))$ is both left-c.e. and right-c.e. and hence computable. Hence $I_{f^{-1}_n(V)}$ is an $L_1(\prodm_{\prss})$ Schnorr test and so a computable point of $L_1(\prodm_{\prss})$.

Further, the $L_1(\prodm_{\prss})$ distance between $I_{f^{-1}_n(V)}, I_{f^{-1}(V)}$ is uniformly computable:
\begin{align}
& \|I_{f^{-1}_n(V)}- I_{f^{-1}(V)}\|_{L_1(\prodm_{\prss})}  =\prodm_{\prss}(f^{-1}(V)\triangle f_n^{-1}(V)) =\prodm\big(\pspace \times (f^{-1}(V)\triangle f_n^{-1}(V))\big)\notag \\
& = \prodm\big((\pspace \times f^{-1}(V))\triangle (\pspace\times f_n^{-1}(V)\big) = \prodm\big((V \times \sspace)\triangle (\pspace\times f_n^{-1}(V))\big) \label{eqn:finally}
\end{align}
To get the last identity, we use Proposition~\ref{prop:switcheroo}(\ref{prop:switcheroo:1}). Further, $\prodm(V\times \sspace)=\prior(V)$, which is computable by hypothesis. And $\prodm(\pspace\times f_n^{-1}(V))=\prodm_{\prss}(f_n^{-1}(V))$, which is computable by the previous paragraph. Hence, since (\ref{eqn:finally}) is the measure of a Boolean combination of c.e. open events in $S$ with computable $\prodm$-measure, it too has $\prodm$-computable measure, and uniformly so.

Hence indeed $I_{f^{-1}(V)}$ is a computable point of $L_1(\prodm_{\prss})$, when $V\subseteq \pspace$ is an element of the $\prior$-computable basis.

Now suppose that $U$ is a c.e. open with $\prior(U)$ computable. Then choose a computable sequence of elements $V_n$ from the $\prior$-computable basis such that $U=\bigcup_n V_n$.  By the previous argument, one has that $I_{f^{-1}(V_n)}$ is uniformly a computable point of $L_1(\prodm_{\prss})$.  One has that:
\begin{align}
& \|I_{f^{-1}(V_n)}- I_{f^{-1}(U)}\|_{L_1(\prodm_{\prss})}  =\prodm_{\prss}(f^{-1}(V_n)\triangle f^{-1}(U))  =\prodm\big(\pspace \times (f^{-1}(U)\triangle f^{-1}(V_n))\big) \notag \\
& = \prodm\big((\pspace \times f^{-1}(U))\triangle (\pspace\times f^{-1}(V_n)\big) = \prodm\big((U \times \sspace)\triangle (V_n\times \sspace)\big)  = \prior(U\setminus V_n)
\label{eqn:finally2}
\end{align}
The second to last identity follows again from Proposition~\ref{prop:switcheroo}(\ref{prop:switcheroo:1}). Hence the $L_1(\prodm_{\prss})$ distance between $I_{f^{-1}(V_n)}, I_{f^{-1}(U)}$ is uniformly computable and goes to zero. Hence we can compute a subsequence such that $I_{f^{-1}(V_n)}\rightarrow I_{f^{-1}(U)}$ fast, so that $I_{f^{-1}(U)}$ is likewise a computable point of $L_1(\prodm_{\prss})$.
\end{proof}

Finally, we need one classical estimate from Doob's Maximal Inequality (\cite[Theorem 9.4]{Gut2013-ou}).
\begin{prop}\label{prop:estimatedoobxmax}
(Classical estimate from Doob's Maximal Inequality).

 Let $\ssfilt{n}$ be the $\sigma$-algebra of events on $\sspace$ generated by the basic clopens $[\sigma]$, where $\sigma$ ranges over length $n$ strings from $T$. 

Suppose that $B\subseteq \sspace$ is a Borel event. Then, as $n\geq 0$, the sequence of functions $M_n(\ssvar):=\prodm_{\prss}(B\mid \ssvarn{n})$ is a classical martingale in $L_2(\prodm_{\prss})$ with respect to $\ssfilt{n}$. Further, for all $m\geq 0$, one has $\|\sup_{n\geq m} \prodm_{\prss}(B\mid \cdot_n)\|_{L_2(\prodm_{\prss})}\leq 2\cdot \sqrt{\prodm_{\prss}(B)}$.
\end{prop}
\begin{proof}
As a version of the conditional expectation, we use: $\mathbb{E}_{\prodm_{\prss}}[f\mid \ssfilt{n}](\ssvar):=\frac{1}{\prodm_{\prss}([\ssvarn{n}])} \int_{[\ssvarn{n}]} f(\ssvar^{\prime}) \; d\prodm_{\prss}(\ssvar^{\prime})$ when $\prodm_{\prss}([\ssvarn{n}])>0$, and zero otherwise.

Then $M_n(\ssvar)=\prodm_{\prss}(B\mid \ssvarn{n})$ is a classical martingale since it is the conditional expectation of the random variable $I_B$.

For the estimate, fix $m\geq 0$. By Doob's Maximal Inequality followed by Conditional Jensen one has:
\begin{align*}
\|\sup_{m\leq n\leq k} \prodm_{\prss}(B\mid \cdot_n)\|_{L_2(\prodm_{\prss})}  \leq \|\sup_{n\leq k} \prodm_{\prss}(B\mid \cdot_n)\|_{L_2(\prodm_{\prss})} \\
\leq 2\cdot \| \prodm_{\prss}(B\mid \cdot_k)\|_{L_2(\prodm_{\prss})}\leq 2\cdot \|I_{B}\|_{L_2(\prodm_{\prss})}= 2\cdot \sqrt{\prodm_{\prss}(B)}
\end{align*}
Then we are done by applying the Monotone Convergence Theorem on right-hand side to take $k\rightarrow \infty$.

\end{proof}

Now we prove Theorem~\ref{thm:doob}:
\begin{proof}

By Propositions~\ref{prop:likelihoodareprob}, \ref{prop:important} it suffices to show the part about pairs. Let $U\subseteq \pspace$ c.e. open with $\prior(U)$ computable. 

We show $\sup_{n\geq m} \poste{U}{\ssvarn{n}}$ is a Schnorr $L_2(\prodm_{\prss})$-test, uniformly in $m\geq 0$. By Proposition~\ref{prop:posteriorintegral} and $0\leq \poste{U}{\ssvarn{n}}\leq 1$, one has that $\sup_{n\geq m} \poste{U}{\ssvarn{n}}$ is lsc and $\leq 1$, uniformly in $m\geq 0$. Hence, it remains to show that $\sup_{n\geq m} \poste{U}{\ssvarn{n}}$ has computable $L_2(\prodm_{\prss})$-norm, uniformly in $m\geq 0$. To this end, it suffices to show that given rational $\epsilon>0$ we can compute an index for a computable point of $L_2(\prodm_{\prss})$ within $\epsilon$ of $\sup_{n\geq m} \poste{U}{\ssvarn{n}}$. 

Let $\epsilon>0$ be rational. By Proposition~\ref{prop:thiswashard}, compute an index for an element $V\subseteq U$ of the $\prior$-computable basis and compute $i\geq 0$ such that  $\prodm_{\prss}(f^{-1}(U)\triangle f^{-1}_i(V))<\frac{\epsilon^2}{8}$. Since $f_i^{-1}(V)$ is c.e. open, we can write $f_i^{-1}(V)=\bigsqcup_k [\sigma_k]$, for a computable sequence of strings $\sigma_k$ in $T$. Since $f_i^{-1}(V)$ is also effectively closed, we have $\prodm_{\prss}(f_i^{-1}(V))$ is computable. Hence, we can compute $j$ such that $\prodm_{\prss}(f_i^{-1}(V)\setminus \bigsqcup_{k\leq j} [\sigma_k])<\frac{\epsilon^2}{8}$. Putting these together, we have $\prodm_{\prss}(f^{-1}(U)\triangle \bigsqcup_{k\leq j} [\sigma_k])<\frac{\epsilon^2}{4}$. Compute $\ell$ which is strictly greater than the maximum length of $\sigma_k$ for $k\leq j$. Then one has the following, where the norm is with respect to $L_2(\prodm_{\prss})$ and where the first inequality uses Proposition~\ref{prop:switcheroo}(\ref{prop:switcheroo:2}):
\begin{align*}
& \|\sup_{n\geq m} \poste{U}{\ssvarn{n}} - \sup_{m\leq n \leq \ell} \prodm_{\prss}(\bigsqcup_{k\leq j} [\sigma_k]\mid \ssvarn{n})\|_2 =\|\sup_{n\geq m} \poste{U}{\ssvarn{n}} - \sup_{n\geq m} \prodm_{\prss}(\bigsqcup_{k\leq j} [\sigma_k]\mid \ssvarn{n})\|_2 \\
& \leq \|\sup_{n\geq m} \prodm_{\prss}(f^{-1}(U)\triangle \bigsqcup_{k\leq j} [\sigma_k]\mid \ssvarn{n})\|_2  \leq 2\cdot \sqrt{\prodm_{\prss}(f^{-1}(U)\triangle \bigsqcup_{k\leq j} [\sigma_k])} <\epsilon
\end{align*} 
The last two inequalities use Proposition~\ref{prop:estimatedoobxmax} and our earlier estimate. Since the function $x\mapsto \sup_{m\leq n \leq \ell} \prodm_{\prss}(\bigsqcup_{k\leq j} [\sigma_k]\mid \ssvarn{n})$ is a computable element of $L_2(\prodm_{\prss})$, we are done. Namely, we have shown that the function $x\mapsto \sup_{n\geq m} \poste{U}{\ssvarn{n}}$ is indeed a Schnorr $L_2(\prodm_{\prss})$-test, uniformly in $m\geq 0$.

By change of variables, the function $(\psvar,\ssvar)\mapsto \sup_{n\geq m} \poste{U}{\ssvarn{n}}$ is a Schnorr $L_2(\prodm)$-test.

By the classical version of L\'evy's Upward Theorem, one has that $
\prodm_{\prss}(f^{-1}(U)\mid \ssvarn{n})\rightarrow I_{f^{-1}(U)}(\ssvar)$ for $\prodm_{\prss}$-a.s. many $\ssvar$ from $\sspace$. Hence by Proposition~\ref{prop:switcheroo}(\ref{prop:switcheroo:2}) we have $\poste{U}{\ssvarn{n}}\rightarrow I_{U\times \sspace}(\psvar, \ssvar)$ for $\prodm$-a.s. many $(\psvar,\ssvar)$ from $\prodspace$.

Then putting the two previous paragraphs together with Proposition~\ref{prop:wr2propv2}, we have: for all $(\psvar,\ssvar)$ in $\mathsf{SR}^{\prodm}\setminus (U\times \sspace)$ that $\poste{U}{\ssvarn{n}} \rightarrow 0$. 

Now it remains to show for all $(\psvar,\ssvar)$ in $\mathsf{SR}^{\prodm}\cap (U\times \sspace)$ that $\poste{U}{\ssvarn{n}} \rightarrow 1$.

Let $U=\bigcup_m U_m$, where $U_m$ is a computable sequence of c.e. opens from the $\prior$-computable basis. Choose computable sequence $C_m\supseteq U_m$ of effectively closed sets such that $U_m, C_m$ have same $\prior$-measure.

It suffices to show that for each $m\geq 0$ and each $(\psvar,\ssvar)$ in $\mathsf{SR}^{\prodm}\cap (U_m\times \sspace)$ that $\poste{U_m}{\ssvarn{n}} \rightarrow 1$. For, given $(\psvar,\ssvar)$ in $\mathsf{KR}^{\prodm}$ one has that $\post{\ssvarn{n}}$ is in $\mathcal{P}(\pspace)$ by Propositions~\ref{prop:elefact},\ref{prop:whenpostprob2} and so this and $U_m\subseteq U$ also makes it that for each $m\geq 0$ and $(\psvar,\ssvar)$ in $\mathsf{SR}^{\prodm}\cap (U_m\times \sspace)$ one has $\poste{U}{\ssvarn{n}} \rightarrow 1$.

Fix $m\geq 0$ for the remainder of the proof. Let $V_m=\pspace\setminus C_m$, which is c.e. open with computable $\prior$-measure. Then we can run the argument above to conclude that for all $(\psvar,\ssvar)$ in $\mathsf{SR}^{\prodm}\setminus (V_m\times \sspace)$ one has that $\poste{V_m}{\ssvarn{n}}\rightarrow 0$. Let $(\psvar,\ssvar)$ in $\mathsf{SR}^{\prodm}\cap (U_m\times \sspace)$. By Proposition~\ref{prop:elefact}, we have $\mathsf{SR}^{\prodm}\subseteq \mathsf{KR}^{\prior}\times \sspace$, and so $(\psvar,\ssvar)$ is in $\mathsf{SR}^{\prodm}\setminus (V_m\times \sspace)$. Hence $\poste{V_m}{\ssvarn{n}}\rightarrow 0$. 
Due to  $(\psvar,\ssvar)$ in $\mathsf{KR}^{\prodm}$, one has $\postws{\ssvarn{n}}$ is in $\mathcal{P}(\pspace)$ by Propositions~\ref{prop:elefact},~\ref{prop:whenpostprob2}, so that $\poste{U_m}{\ssvarn{n}}=\poste{C_m}{\ssvarn{n}} \rightarrow 1$. Note the identity follows from the fact that for $(\psvar,\ssvar)$ in $\mathsf{KR}^{\prodm}$ we have $\poste{\mathsf{KR}^{\prior}}{\ssvarn{n}}=1$, as one can see from the the formula for the posterior in (\ref{eqn:post}).

\end{proof}

%% file: 06-paradigmatic.tex
\section{Limiting relative frequencies}\label{sec:paradigmatic}

In this section, we prove Proposition~\ref{prop:para}, largely following traditional proofs of the Strong Law of Large Numbers for Bernoulli trials (e.g. \cite[43-44]{McKean2014-zj}). Since we are proving Proposition~\ref{prop:para}, everything in this section (the parameter space, sample space, etc.) are as in the statement of the proposition.

Before we begin, since this proposition provides a sufficient condition for the application of Theorem~\ref{thm:doob}, let us say a brief word about this application. Let $C\subseteq \pspace$ of $\prior$-measure one be the set mentioned in the statement of Proposition~\ref{prop:para}.

First, note that since $\prior(C)=1$, we have that $\psvar$ is in $C$ whenever $(\psvar, \ssvar)$ is in $\mathsf{KR}^{\prodm}$, by Proposition~\ref{prop:elefact}.

Second, for a binary string $\sigma$ of length $n$ which possesses $i$ ones, the likelihood in Proposition~\ref{prop:para} is defined as $\pmape{\sigma}{\psvar}=\psvar^i \cdot (1-\psvar)^{n-i} \cdot I_{C}(\psvar)$. Then we can use the version in the sense of Definition~\ref{defn:conventionversions} given by $\gmape{\sigma}{\psvar} = \psvar^i \cdot (1-\psvar)^{n-i}$ and $\hmape{\sigma}{\psvar} = \psvar^i \cdot (1-\psvar)^{n-i}\cdot I_{\pspace\setminus C}(\psvar)$. Further, by the previous paragraph, for $\psvar$ in the $\prodm$-measure one set $\mathsf{KR}^{\prodm}$, we have $\gmape{\sigma}{\psvar}-\hmape{\sigma}{\psvar} = \psvar^i \cdot (1-\psvar)^{n-i}$.

Now we turn to the proof of Proposition~\ref{prop:para}. Where $\prps:\prodspace\rightarrow \pspace$ is the projection onto the first component (cf. diagram in \S\ref{sec:intro}), we define:
\begin{defn}
$\alpha= \mathbb{E}_{\prodm} \prps=\int_{\pspace} \psvar \; d\prior(\psvar)$ and $\beta=\mathbb{E}_{\prodm} \prps^2=\int_{\pspace} \psvar^2 \; d\prior(\psvar)$.
\end{defn}
\noindent Then $\alpha\geq \beta$ are computable reals since $\prior$ is computable and the identity map from $\pspace$ to $\mathbb{R}$ and the square map are computable continuous. We can then calculate the expectations of the following random variables in terms of $\alpha,\beta$, where $g_i(\psvar,\ssvar)=\ssvar(i)$ and where $i,j$ distinct:
\begin{equation*}
\mathbb{E}_{\mu} g_i=\alpha, \hspace{2mm} \mathbb{E}_{\mu} g_i g_j = \beta, \hspace{2mm} \mathbb{E}_{\mu} \prps g_i =\beta, \hspace{2mm} \mathbb{E}_{\mu} (g_i-\prps)^2 = \alpha-\beta, \hspace{2mm} \mathbb{E}_{\mu} (g_i-\prps)(g_j-\prps) = 0
\end{equation*}

Then we show:
\begin{prop}\label{prop:fnthebasis} 

There is a $\prior$-computable basis on $\pspace$ such that for all $n\geq 0$ and all elements $U$ of the $\prior$-computable basis, the event $f_n^{-1}(U)$ is uniformly both c.e. open and effectively closed.

Further, for $\epsilon>0$ and $n\geq 0$ one has $\prodm(\left|f_n\circ \prss -\prps\right|>\epsilon)\leq \epsilon^{-2} \cdot n^{-1} \cdot (\alpha-\beta)$
\end{prop}
\begin{proof}
Using the Effective Baire Category Theorem (\cite[Theorem 5.1.2]{Hoyrup2009-pl}), we can choose a $\prior$-computable basis consisting of intervals $(p,q)$ and $[0,p)$ and $(q,1]$ and $[0,1]$ such that $0\leq p<q\leq 1$ is a uniformly computable sequence of reals such that these $p,q$ are not equal to any element of the countable set $\{\frac{1}{n}\sum_{i<n} \sigma(i) : \sigma\in 2^n, n>0\}$. This assumption makes both $I_{n,p,q} =\{\sigma\in 2^n: p<\frac{1}{n}\sum_{i<n} \sigma(i)<q\}$ and its relative complement $J_{n,p,q}  =2^n\setminus I_{n,p,q}$ uniformly c.e. in $n,p,q$. Further, one has $f_n^{-1}(p,q)=\bigsqcup_{\sigma\in I_{n,p,q}} [\sigma]$ and $\sspace \setminus f_n^{-1}(p,q)=\bigsqcup_{\sigma\in J_{n,p,q}} [\sigma]$. Hence, $f_n^{-1}(p,q)$ is both c.e. open and effectively closed. The cases of $[0,p)$ and $(q,1]$ are similar. For $\epsilon>0$ and $n\geq 0$ one has:
\begin{align*}
& \prodm(\left|f_n\circ X-\prps\right|>\epsilon)  =  \prodm(\left|\sum_{i<n} (\ssvar(i)-\psvar)\right|>n\cdot \epsilon) = \prodm(\left|\sum_{i<n} (\ssvar(i)-\psvar)\right|^2>n^2\cdot \epsilon^2)\\ & \leq  \frac{1}{\epsilon^2\cdot n^2} \mathbb{E}_{\prodm}\big(\sum_{i<n} (g_i-\prps)\big)^2 =  \frac{1}{\epsilon^2\cdot n^2} \sum_{i<n} \mathbb{E}_{\prodm} (g_i-\prps)^2 = \frac{1}{\epsilon^2\cdot n} (\alpha-\beta)
\end{align*}

\end{proof}

\begin{prop}\label{prop:limrelfreqs}
For $(\psvar,\ssvar)$ in $\mathsf{SR}^{\prodm}$ one has $\lim_n f_n(\ssvar)= \psvar$.
\end{prop}
\begin{proof}
Since $f_n, \prps$ are computable continuous, so is $\left|f_n\circ \prss -\prps\right|$. By Proposition~\ref{prop:miyabeish}, choose a computable sequence of reals $\eta_n$ for $n\geq 1$ such that $\eta_n$ is in the interval $((n+1)^{-\frac{1}{3}}, n^{-\frac{1}{3}})$ and such that the event $\left|f_{n^2}\circ \prss -\prps\right|>\eta_n$ has computable $\prodm$-measure. Then let $U_n$ be this event, and set $g=\sum_{n\geq 1} I_{U_n}$.  Then one has $\prodm(U_n)\leq \prodm\big(\left|f_{n^2}\circ \prss -\prps\right|>(n+1)^{-\frac{1}{3}}\big)\leq (n+1)^{\frac{2}{3}}\cdot n^{-2} \cdot (\alpha-\beta) \leq \sqrt[3]{4} \cdot n^{-\frac{4}{3}} \cdot (\alpha-\beta)$. Then one has $\mathbb{E}_{\prodm} g \leq  \sum_{n=1}^{\infty} \sqrt[3]{4} \cdot n^{-\frac{4}{3}} \cdot (\alpha-\beta)$, and since the latter is finite and computable so is $\mathbb{E}_{\prodm} g$ by the Comparison Test.  Hence $g$ is an $L_1(\prodm)$ Schnorr test. 

Suppose that $(\psvar,\ssvar)$ is in $\mathsf{SR}^{\prodm}$. Choose $n_0> 0$ such that for all $n\geq n_0$ one has $(\psvar,\ssvar)$ is not in $U_n$. Then for all $n\geq n_0$ we have $\left|\sum_{i<n^2} (\ssvar(i)-\psvar)\right|\leq n^{\frac{5}{3}}$. 

Let $m> n_0^2$. Choose least $n> n_0$ with $n^2<m\leq (n+1)^2$. Then one has 
\begin{equation*}
\left|\sum_{i<m} (\ssvar(i)-\psvar)\right|  \leq  \left|\sum_{i<n^2} (\ssvar(i)-\psvar)\right|+(n+1)^2-n^2  \leq n^{\frac{5}{3}}+2n+1 \leq m^{\frac{5}{6}}+2\sqrt{m}+1
\end{equation*}
Hence $\left|(\frac{1}{m} \sum_{i<m} \ssvar(i))-\psvar\right|=\left|\frac{1}{m} \sum_{i<m} (\ssvar(i)-\psvar)\right|\leq m^{-\frac{1}{6}}+2\cdot m^{-\frac{1}{2}}+m^{-1}$.
\end{proof}

%% file: 07-reversal.tex
\section{Proof of Theorem~\ref{thm:reversal}}\label{sec:reversal}

\begin{proof}
The direction from (\ref{thm:reversal:1}) to (\ref{thm:reversal:2}) follows directly from Theorem~\ref{thm:doob}.

We now show that ``not (\ref{thm:reversal:1})'' implies ``not (\ref{thm:reversal:2}).'' Suppose that $\psvar_0$ is in $\bigcap_n U_n$, where $U_n$ is a sequential Schnorr $\prior$-test. Without loss of generality $U_0=\pspace$ and $U_n$ is decreasing. Note that since $U_0=\pspace$, we have $\sum_n \prior(U_n\setminus U_{n+1})=1$ and hence $\sum_n \prodm((U_n\setminus U_{n+1})\times \sspace)=1$.

By the hypothesis of the Theorem, choose an apt triple $\sspace, \psvar\mapsto\pmap{\psvar}, f_n:\sspace\rightarrow \pspace$ (cf. Definition~\ref{defn:apttriple}). We will construct another apt triple $\sspacealt, \psvar\mapsto\pmapalt{\psvar}, g_n:\sspacealt\rightarrow \pspace$ such that $\psvar_0$ is not computably consistent relative to $\prior, \sspacealt, \psvar\mapsto\pmapalt{\psvar}$. 

Let $\sspacealt=[T^{\infty}]$, where $T^{\infty}=\{\emptyset\}\cup \{(m)^{\frown}\sigma: m\geq 0, \sigma\in T\}$, where $\emptyset$ denotes the root (the unique length zero string). That is, $T^{\infty}$ copies $T$ countably many times after the root $\emptyset$. Clearly since $T$ is a computable subtree of $\mathbb{N}^{<\mathbb{N}}$ with no dead ends, so is $T^{\infty}$.

Define likelihood as follows, for $\tau$ in $T^{\infty}$:
\begin{equation}\label{eqn:pmapalt}
\pmapalte{\tau}{\psvar} = \begin{cases}
1      & \text{if $\tau$ is the root and $\psvar$ is not in $\bigcap_n U_n$}, \\
0      & \text{if $\tau$ is the root and $\psvar$ is in $\bigcap_n U_n$}, \\
\pmape{\sigma}{\psvar}\cdot I_{U_m\setminus U_{m+1}}(\psvar)      & \text{if $\tau=(m)^{\frown}\sigma$}.
\end{cases}
\end{equation}
The idea is: if $\psvar$ is in $U_m\setminus U_{m+1}$, then $\pmapalt{\psvar}$ puts all the probability on the $m$-th copy of $T$ in $T^{\infty}$, where it copies the original likelihood; whereas if $\psvar$ is in the $\prior$-null set $\bigcap_n U_n$, then $\pmapalt{\psvar}$ is the zero element of $\mathcal{M}^+(\sspacealt)$. 

Regarding the effective properties of the likelihood function (cf. Definition~\ref{defn:assum}(\ref{assum2.5}), note that for $\prior$-a.s. many $\psvar$ in $\pspace$, one has that $\psvar$ is not in $\bigcap_n U_n$ and hence is in some $U_m\setminus U_{m+1}$, and thus:
\begin{itemize}[leftmargin=*] 
\item The likelihood function $\psvar\mapsto \pmapalte{\tau}{\psvar}$ is $L_1(\prior)$ computable, uniformly in $\tau$ in $T^{\infty}$ by Proposition~\ref{prop:indicator}(\ref{prop:indicator:3}) and Proposition~\ref{prop:productsL1nu}. 
\item $\pmapalt{\psvar}$ is in $\mathcal{P}(\sspacealt)$ for $\prior$-a.s. many $\psvar$ in $\pspace$ since $\pmap{\psvar}$ is in $\mathcal{P}(\sspace)$ for $\prior$-a.s. many $\psvar$ in $\pspace$.
\end{itemize}

Suppose that, per Definition~\ref{defn:conventionversions},  $\pmape{\sigma}{\psvar}$ is understood pointwise as its version $\gmape{\sigma}{\psvar}-\hmape{\sigma}{\psvar}$, where these are $L_1(\prior)$ Schnorr tests, uniformly in $\sigma$ from $T$. We define 
a pair $\Gmape{\tau}{\psvar}, \Hmape{\tau}{\psvar}$ of $L_1(\prior)$ Schnorr tests, uniformly in $\tau$ in $T^{(\infty)}$, such that $\Gmape{\tau}{\psvar}-\Hmape{\tau}{ \psvar}$ is a version of $\pmapalte{\tau}{\psvar}$:
\begin{itemize}[leftmargin=*]
\item If $\tau$ is the root, define $\Gmape{\tau}{\psvar} = \sum_m I_{U_{m}}(\psvar)$ and $\Hmape{\tau}{\psvar} = \sum_m I_{U_{m+1}}(\psvar)$.
\item If $\tau = (m)^{\frown}\sigma$, define
$\Gmape{(m)^{\frown}\sigma}{\psvar}  =\gmape{\sigma}{\psvar}\cdot I_{U_m}(\psvar)+\hmape{\sigma}{\psvar}\cdot I_{U_{m+1}}(\psvar)$ and 
$\Hmape{(m)^{\frown}\sigma}{\psvar}  =\gmape{\sigma}{\psvar}\cdot I_{U_{m+1}}(\psvar)+\hmape{\sigma}{\psvar}\cdot I_{U_m}(\psvar)$.
\end{itemize}
Then $\Gmape{\tau}{\psvar}, \Hmape{\tau}{ \psvar}$ are finite for $\psvar$ in $\mathsf{SR}^{\prior}$, and on these points we have that $\Gmape{\tau}{\psvar}-\Hmape{\tau}{\psvar}$ agrees with (\ref{eqn:pmapalt}), when $\pmape{\sigma}{\psvar}$ is understood pointwise as its version $\gmape{\sigma}{\psvar}-\hmape{\sigma}{\psvar}$. When $\psvar$ is in $\bigcap_n U_n$, by our convention that $\infty-\infty=0$, we have that  $\Gmape{\tau}{\psvar}-\Hmape{\tau}{\psvar}=0$ for all $\tau$. Hence when $\psvar$ is in $\bigcap_n U_n$, this version of $\pmapalt{\psvar}$ is the zero element of $\mathcal{M}^+(\sspacealt)$, and hence not an element of $\mathcal{P}(\sspacealt)$.\footnote{If one wanted to make $\infty-\infty$ undefined, one would have to introduce notation for partial functions, and then emend the definition of computable consistency to require that the likelihood is defined and a probability measure on the sample space. Instead of ``when $\psvar$ is in $\bigcap_n U_n$, this version of $\pmapalt{\psvar}$ is the zero element of $\mathcal{M}^+(\sspacealt)$,'' at this point the proof would read ``when $\psvar$ is in $\bigcap_n U_n$, the likelihood $\pmapalt{\psvar}$ is not defined.''} By the first conjunct of Definition~\ref{defn:con:eff}(\ref{defn:con:eff:2}), any $\psvar$ in $\bigcap_n U_n$ is not computably consistent.

Define uniformly computable continuous $g_n:\sspacealt\rightarrow \pspace$ by $g_n((m)^{\frown} x) = f_n(\ssvar)$. It remains to show that $\sspacealt, \psvar\mapsto \pmapalt{\psvar}, g_n$ satisfies conditions~(\ref{thm:doob1})-(\ref{thm:doob2}) of Theorem~\ref{thm:doob}.

To this end, define uniformly computable continuous $\iota_m:\prodspace\rightarrow \pspace\times \sspacealt$ by $\iota_m(\psvar, \ssvar)=(\psvar, (m)^{\frown} \ssvar)$. Further, let $\prodmalt$ be the joint distribution on the product space $\pspace\times \sspacealt$ formed from the prior $\prior$ and the new likelihood $\psvar\mapsto \pmapalte{\tau}{\psvar}$. Then for each Borel event $B\subseteq \pspace\times \sspacealt$, we have
\begin{equation}\label{eqn:prodmalt}
\prodmalt(B) = \sum_m \prodm(\iota_m^{-1}(B)\cap \big((U_m\setminus U_{m+1})\times \sspace\big))
\end{equation}
This follows by a simple induction on $B$, which we omit for reasons of space.

\begin{proofdetail}

Here are the details.

Suppose that $B$ is open and hence a a countable union of elements from the $\prodmalt$-computable basis. Hence we can write $B=\bigcup_i (V_i\times [\sigma_i])$, where $V_i$ comes from a $\prior$-computable basis and $\sigma_i$ is in $T^{\infty}$; by adding more to the list $\sigma_i$ we may assume that $\left|\sigma_i\right|>1$. First we show that for all $m\geq 0$ and all $\psvar$ in $U_m\setminus U_{m+1}$ we have 
\begin{equation}\label{eqn:prodmalt2}
\int_{\sspacealt \cap [m]} I_B(\psvar,y) \; d\pmapalt{\psvar}(y)  = \int_{\sspace} I_{\iota_m^{-1}(B)}(\psvar, \ssvar) \; d\pmap{\psvar}(\ssvar)
\end{equation}

To see this, first define index set $I_m=\{i: \sigma_i(0)=m\}$ and for each $i$ in $I_m$, let $\sigma_i=(m)^{\frown} \tau_i$. Let $J_m =\{i\in I_m: \psvar\in V_i\}$ (which depends on $\psvar$, which is fixed for this part of the proof). Let $K_m=\{i\in J_m: \forall \; j\in J_m \; \tau_j\nprec \tau_i\}$, so that both $\{\tau_i: i\in K_m\}$ and $\{\sigma_i: i\in K_m\}$ are prefix-free and both $\bigsqcup_{i\in K_m} [\tau_i] = \bigcup_{i\in J_m} [\tau_i]$ and $\bigsqcup_{i\in K_m} [\sigma_i] = \bigcup_{i\in J_m} [\sigma_i]$. Then we have the  identity $\bigsqcup_{i\in K_m} [\tau_i]=(\iota_m^{-1}(B))^{\psvar}$ on $\sspace$.
 To see the identity we argue for both inclusions:
\begin{itemize}[leftmargin=*]
\item Suppose that $\ssvar$ is in $[\tau_i]$ for some $i$ in $K_m$. Then $\iota_m(\psvar, \ssvar)=(\psvar, (m)^{\frown}\ssvar)$ is in $V_i\times [\sigma_i] \subseteq B$ and so $(\psvar, \ssvar)$ is in $\iota_m^{-1}(B)$ and so $\ssvar$ is in $(\iota_m^{-1}(B))^{\psvar}$.
\item Suppose that $\ssvar$ is in $(\iota_m^{-1}(B))^{\psvar}$. Then $\iota_m(\psvar, \ssvar)$ in $B$. Then $(\psvar, (m)^{\frown}\ssvar)$ is in some $V_i\times [\sigma_i]$. Then $i$ is in $I_m$ and $i$ is in $J_m$. If $i$ is in $K_m$, then $\ssvar$ is in $[\tau_i]$ and we are done. If $i$ is not in $K_m$, then there is $j$ in $J_m$ with $\tau_j\prec \tau_i$; and by taking $j$ in $J_m$ such that $\tau_j$ is $\prec$-minimal with this property, we have that this $j$ is also in $K_m$. Since $\ssvar$ is in $[\tau_i]$, it is also in $[\tau_j]$, and we are done. 
\end{itemize}
This finishes the verification of $\bigsqcup_{i\in K_m} [\tau_i]=(\iota_m^{-1}(B))^{\psvar}$ on $\sspace$. Then we have the following, where the third identity comes from $\psvar$ being an element of $U_m\setminus U_{m+1}$:
\begin{align*}
& \int_{\sspace} I_{\iota_m^{-1}(B)}(\psvar, \ssvar) \; d\pmap{\psvar}(\ssvar) = \pmape{(\iota_m^{-1}(B))^{\psvar}}{\psvar} =\sum_{i\in K_m} \pmape{\tau_i}{\psvar} \\
& = \sum_{i\in K_m} \pmapalte{\sigma_i}{\psvar} =\int_{\sspacealt} I_{\bigcup_{i\in J_m} [\sigma_i]}(\ssvaralt) \; d\pmapalt{\psvar}(\ssvaralt)=\int_{\sspacealt \cap [m]} I_B(\psvar,y) \; d\pmapalt{\psvar}(y)
\end{align*}
Hence indeed (\ref{eqn:prodmalt2}) holds for all $m\geq 0$ and all $\psvar$ in $U_m\setminus U_{m+1}$. 

Then we have:
\begin{align*}
& \prodmalt(B) = \int_{\pspace} \int_{\sspacealt} I_B(\psvar,\ssvar) \; d\pmapalt{\psvar}(\ssvar) \; d\prior(\psvar) =\sum_m \int_{U_m\setminus U_{m+1}} \int_{\sspacealt\cap [m]} I_B(\psvar,\ssvaralt) \; d\pmapalt{\psvar}(\ssvaralt) \; d\prior(\psvar) \\
& = \sum_m  \int_{U_m\setminus U_{m+1}}  \int_{\sspace} I_{\iota_m^{-1}(B)}(\psvar, \ssvar) \; d\pmap{\psvar}(\ssvar) \; d\prior(\psvar) = \sum_m \prodm(\iota_m^{-1}(B)\cap \big((U_m\setminus U_{m+1})\times \sspace\big))
\end{align*}
This finishes the argument for (\ref{eqn:prodmalt}) in the case where $B$ is open.

Suppose that (\ref{eqn:prodmalt}) holds for $B$; then we show it holds for $(\pspace\times \sspacealt)\setminus B$:
\begin{align*}
& \prodmalt((\pspace\times \sspacealt)\setminus B) = 1-\prodmalt(B)\\
&= \sum_m \prodm((U_m\setminus U_{m+1})\times \sspace)-\sum_m \prodm(\iota_m^{-1}(B)\cap \big((U_m\setminus U_{m+1})\times \sspace\big)) \\
& = \sum_m \prodm( \big((U_m\setminus U_{m+1})\times \sspace\big)\setminus \iota_m^{-1}(B)) = \sum_m \prodm(\iota_m^{-1}((\pspace\times \sspacealt) \setminus B)\cap \big((U_m\setminus U_{m+1})\times \sspace\big))
\end{align*}

 Suppose that (\ref{eqn:prodmalt}) holds for each of $B_0, B_1, \ldots$ where these are pairwise disjoint; then we show it holds for $B=\bigcup_i B_i$. Since the $B_i$ are pairwise disjoint, so are the $\iota_m^{-1}(B_i)$. One has the following, where the exchange of the summations is permitted since we are working with non-negative terms:
\begin{align*}
& \prodmalt(B) = \sum_i \prodmalt(B_i) =\sum_i \sum_m \prodm(\iota_m^{-1}(B_i)\cap \big((U_m\setminus U_{m+1})\times \sspace\big))\\
& = \sum_m \sum_i \prodm(\iota_m^{-1}(B_i)\cap \big((U_m\setminus U_{m+1})\times \sspace\big))  = \sum_m \prodm(\bigsqcup_i \iota_m^{-1}(B_i)\cap \big((U_m\setminus U_{m+1})\times \sspace\big)) \\
& = \sum_m \prodm(\iota_m^{-1}(B)\cap \big((U_m\setminus U_{m+1})\times \sspace\big))
\end{align*}

 Suppose that (\ref{eqn:prodmalt}) holds for each of $B_0, B_1, \ldots$ where these are decreasing; then we show it holds for $B=\bigcap_i B_i$. Since the $B_i$ sequences is decreasing, so are the $\iota_m^{-1}(B_i)$. One has the following, where the exchange of limit and summation is justified by the Dominated Convergence Theorem relative to the counting measure:
\begin{align*}
& \prodmalt(B) =\lim_i \prodmalt(B_i)=\lim_i \sum_m \prodm(\iota_m^{-1}(B_i)\cap \big((U_m\setminus U_{m+1})\times \sspace\big)) \\
& =\sum_m\lim_i \prodm(\iota_m^{-1}(B_i)\cap \big((U_m\setminus U_{m+1})\times \sspace\big)) \\
& =\sum_m \prodm(\iota_m^{-1}(B)\cap \big((U_m\setminus U_{m+1})\times \sspace\big))
\end{align*}

This finishes the verification of (\ref{eqn:prodmalt}).
\end{proofdetail}

Finally, here is the verification that $\sspacealt, \psvar\mapsto \pmapalt{\psvar}, g_n$ satisfies conditions~(\ref{thm:doob1})-(\ref{thm:doob2}) of Theorem~\ref{thm:doob}:
\begin{itemize}[leftmargin=*]
\item Suppose that $A$ is a $\prodm$-measure one subset of $\mathsf{KR}^{\prior}\times \sspace$ on which $\lim_n f_n(\ssvar)$ exists and is equal to $\psvar$. Let $B=\bigsqcup_m \iota_m(A)$, so that $\iota_m^{-1}(B)=A$ for all $m\geq 0$ and so that $B\subseteq \mathsf{KR}^{\prior}\times \sspacealt$.
Then by (\ref{eqn:prodmalt}) one has
\begin{equation*}
 \prodmalt(B) = \sum_m \prodm(A\cap \big((U_m\setminus U_{m+1})\times \sspace\big))  = \sum_m \prodm((U_m\setminus U_{m+1})\times \sspace) =1 
\end{equation*}
Finally note that for all $(\psvar, (m)^{\frown} \ssvar)$ in $B$ one has that $(\psvar, \ssvar)$ is in $A$, and hence by condition~(\ref{thm:doob1}) of Theorem~\ref{thm:doob}, we have that the limit $\lim_n g_n((m)^{\frown} \ssvar))=\lim_n f_n(\ssvar)$ exists and is equal to $\psvar$.
\item Suppose that $U\subseteq \pspace$ is from a $\prior$-computable basis. Then, by condition~(\ref{thm:doob2}) of Theorem~\ref{thm:doob} we can write $f_n^{-1}(U)=\bigsqcup_i [\sigma_i]$ and $\sspace\setminus f_n^{-1}(U)=\bigsqcup_i [\tau_i]$ for computable sequences of $\tau_i, \sigma_i$ from $T$. Then one has $g_n^{-1}(U)=\bigsqcup_m \bigsqcup_i [(m)^{\frown} \sigma_i]$ and $\sspacealt\setminus g_n^{-1}(U)=\bigsqcup_m \bigsqcup_i [(m)^{\frown} \tau_i]$.
\end{itemize}
\end{proof}

%% file: 08.5-simplex.tex
\section{The infinite simplex, subspaces, and identifiability}\label{sec:simplex}

The infinite dimensional simplex $\mathbb{S}_{\infty}$ (Definition~\ref{defn:sinfty}) is a subspace of the Hilbert cube~$[0,1]^{\mathbb{N}}$. But since $\mathbb{S}_{\infty}$ is not closed in $[0,1]^{\mathbb{N}}$, some care has to be devoted to describing the computable Polish space structure on $\mathbb{S}_{\infty}$.

We define the natural effectivization of the subspace relation:
\begin{defn}\label{defn:subspace}
We say $Y_0$ is a \emph{computable Polish subspace} of $Y$ if
\begin{enumerate}[leftmargin=*]
    \item\label{defn:subspace:0} $Y,Y_0$ are computable Polish spaces and $Y_0\subseteq Y$
    \item\label{defn:subspace:1} There is a computable procedure which when given an index for a c.e. open $U$ in $Y$ returns an index for a c.e. open $V$ in $Y_0$ such that $V=U\cap Y_0$.
    \item\label{defn:subspace:2} There is a computable procedure which when given an index for a c.e. open $V$ in $Y_0$ returns an index for a c.e. open $U$ in $Y$ such that $V=U\cap Y_0$.
\end{enumerate}
\end{defn}

\begin{rmk}\label{rmk:howused}
To illustrate how this concept is used, note that when $Y_0$ is a computable Polish subspace of $Y$, we have:
\begin{enumerate}[leftmargin=*]
    \item\label{rmk:howused:1} If $f:Y\rightarrow Z$ is a computable continuous function, then  $f\upharpoonright Y_0:Y_0\rightarrow Z$ is computable continuous.
    \item\label{rmk:howused:2} If $U_0, U_1, \ldots$ is a computable basis in $Y$, then $U_0\cap Y_0, U_1\cap Y_0, \ldots$ is a comptuable basis in $Y_0$.
\end{enumerate}
In the last point, a \emph{computable basis} in a computable Polish space is just a uniformly computable sequence of c.e. opens such that every c.e. open can be written as an effective union of sets from the sequence.\footnote{As its name suggests, a $\nu$-computable basis from Definition~\ref{defn:nucomputablebasis} is a computable basis.}
\end{rmk}

We build up to proving:
\begin{prop}\label{prop:sinfty}
$\mathbb{S}_{\infty}$ is a computable Polish subspace of $[0,1]^{\mathbb{N}}$. 
\end{prop}

But the simplest example of Definition~\ref{defn:subspace} is:
\begin{ex}\label{rmk:closedcompsubspace}
If $Y_0$ is effectively closed in $Y$ and inherits the metric from $Y$ and has as a countable dense set a uniformly computable sequence of points from $Y$, then $Y_0$ is a computable Polish subspace of $Y$.
\end{ex}
However, $\mathbb{S}_{\infty}$ is properly $\Pi^0_2$ in $[0,1]^{\mathbb{N}}$, and so one needs an effective version of the classical result that $\Pi^0_2$ subspaces of Polish spaces are Polish (\cite[17]{Kechris1995-hr}):
\begin{prop}\label{prop:pi02eff}
Suppose that $U_n$ is a uniformly computable sequence of c.e. opens in a computable Polish space $Y$, with $C_n=Y\setminus U_n$ the corresponding effectively closed sets, which we further suppose to be non-empty. 

Let $d$ be the metric with respect to which $Y$ is a computable Polish space. Suppose that $d$ is bounded above by some rational constant $c>1$. Suppose further that the map $y\mapsto d(y, C_n):=\inf_{z\in C_n} d(y,z)$ is computable continuous, uniformly in $n\geq 0$.

Let $Y_0=\bigcap_n U_n$, and suppose that $Y_0$ contains a computable subset $D_0$ of the countable dense set $D$ of $Y$, such that $D_0$ is dense in $Y_0$.

Then $Y_0$ is a computable Polish subspace of $Y$ space with metric:
\begin{equation}\label{eqn:d0}
d_0(x,y)=d(x,y)+\sum_n \min\big(2^{-(n+1)}, \left|\frac{1}{d(x,C_n)}-\frac{1}{d(y,C_n)}\right|\big)
\end{equation}
\end{prop}
\begin{proof}

We identify the countable dense set $D$ of $Y$ with natural numbers and use variables $i,j$ for them in this proof.

Since classically $d_0$ provides a complete compatible metric for $Y_0$ (cf. \cite[17]{Kechris1995-hr}), to show that $Y_0$ is a computable Polish space it remains to show that if $i,j$ come from $D_0$, then $d_0(i,j)$ is uniformly computable. To see this, note that since $i,j$ are computable points of $Y_0\subseteq U_n$, we have that $d(i,C_n), d(j,C_n)>0$ are uniformly computable. We can search for, and find, a non-zero rational, uniformly in $i,j,n$, which is below $d(i,C_n), d(j,C_n)$. Using this, we can compute $\min\big(2^{-(n+1)}, \left|\frac{1}{d(i,C_n)}-\frac{1}{d(j,C_n)}\right|\big)$. Then by the Comparison Test, $d_0(i,j)$ is computable, and uniformly so. Hence indeed $Y_0$ is a computable Polish space. 

Now we show that $Y_0$ is a computable Polish subspace of $Y$.

We focus on condition~(\ref{defn:subspace:2}) of Definition~\ref{defn:subspace} since the other one is similar but easier. It suffices to consider the case where $V$ is an open ball $B_{d_0}(j,\epsilon):=\{y_0\in Y_0: d_0(j,y_0)<\epsilon\}$ in $Y_0$, where $j$ in $D_0$ and $\epsilon>0$ is rational; we must show that $B_{d_0}(j, \epsilon)=U\cap Y_0$ for some c.e. open $U$ in $Y$. We proceed as follows:
\begin{itemize}[leftmargin=*]
\item For each rational $q$ in $(0,\frac{1}{c})$ compute $n_q\geq 0$ such that $2^{-(n_q+1)}<\frac{q}{3}$.
\item For each $(i,q)$ in $D_0\times (\mathbb{Q}\cap (0,\frac{1}{c}))$, compute real numbers 
\begin{equation*}
\delta_{i,q,n}:=\frac{1}{d(i,C_n)}+ \frac{q}{3}\cdot 2^{-(n+1)}, \hspace{10mm} \eta_{i,q,n}:=\frac{1}{d(i,C_n)}- \frac{q}{3}\cdot 2^{-(n+1)}
\end{equation*}
Since the metric $d$ is bounded $<c$ and since $q$ is from $(0,\frac{1}{c})$, note that we have that $\eta_{i,q,n}>0$; and trivially $\delta_{i,q,n}>0$.

\begin{proofdetail}
In more detail:

Since $d$ is bounded $<c$, we have $d(i,C_n)<c$, and thus $0<q<\frac{1}{c}<\frac{1}{d(i,C_n)}$ and thus $0<\frac{q}{3}\cdot 2^{-(n+1)}<q<\frac{1}{d(i,C_n)}$ and thus $\eta_{i,q,n}>0$.
\end{proofdetail}

\item Define the c.e. index set $I=\{(i,q)\in D_0\times (\mathbb{Q}\cap (0,\frac{1}{c})): d_0(i,j)+\frac{q}{3}<\epsilon\}$.
\item For each $(i,q)$ in $I$ let $U_{i,q,n}$ be the c.e. open 
\begin{equation}
U_{i,q,n}=(d(\cdot, C_n))^{-1}(-\infty,\frac{1}{\eta_{i,q,n}})\cap (d(\cdot, C_n))^{-1}(\frac{1}{\delta_{i,q,n}},\infty)
\end{equation}
\end{itemize}
Then we claim that $B_{d_0}(j,\epsilon)=Y_0\cap \bigcup_{(i,q)\in I} B_d(i,\frac{q}{3})\cap \bigcap_{n\leq n(q)} U_{i,q,n}$:
\begin{itemize}[leftmargin=*]
\item First suppose that $x$ in $Y_0\cap B_d(i,\frac{q}{3})\cap \bigcap_{n\leq n(q)} U_{i,q,n}$ for some $(i,q)$ in $I$. Then for $n\leq n(q)$, membership in $U_{i,q,n}$ implies that $\left|\frac{1}{d(i,C_n)}-\frac{1}{d(x,C_n)}\right|<\frac{q}{3}\cdot 2^{-(n+1)}$. Together with membership in $ B_d(i,\frac{q}{3})$ and the definition of $n_q$ we then have that $d_0(x,i)<q$. Then $d_0(x,j)\leq d_0(x,i)+d_0(i,j)<q+d_0(i,j)<\epsilon$.
\item Second suppose that $x$ is in $B_{d_0}(j,\epsilon)$. Choose non-zero rational $q$ such that both $q<\frac{3}{2}(\epsilon-d_0(x,j))$ and $q<\frac{1}{c}$. Since for each $n\leq n_q$, the function $y\mapsto \frac{1}{d(y,C_n)}$ is continuous on $U_n$ and $x$ is in $U_n$, we can find $i$ in $D_0$ such that $x$ is in $B_{d_0}(i,\frac{q}{3})$ and $x$ is in $U_{i,q,n}$ for all $n\leq n_q$. Further $(i,q)$ is in $I$: for we have $d_0(i,j)+\frac{q}{3}\leq d_0(i,x)+d_0(x,j)+\frac{q}{3}<\frac{2}{3}q+d_0(x,j)<\epsilon$.
\end{itemize}

\begin{proofdetail}

To show condition~(\ref{defn:subspace:1}) of Definition~\ref{defn:subspace}, it suffices to consider the case where $U$ is a open ball $B_d(j,\epsilon)$ in $Y$, where $\epsilon>0$ is rational; we must show that $B_d(j, \epsilon)\cap Y_0$ is c.e. open in $Y_0$. Define the c.e. index set $I=\{(i,q)\in D_0\times \mathbb{Q}^{>0}: d(i,j)+q<\epsilon\}$. Then we claim that $B_d(j,\epsilon)\cap Y_0=\bigcup_{(i,q)\in I} B_{d_0}(i,q)$:
\begin{itemize}[leftmargin=*]
    \item First suppose that $x$ is in $B_{d_0}(i,q)$, where $(i,q)$ is in $I$. Then $x$ is in $B_d(i,q)$ as well since $d\leq d_0$ everywhere. Then $d(x,j)\leq d(x,i)+d(i,j)<q+d(i,j)<\epsilon$.
    \item Second suppose that $x$ is in $B_d(j,\epsilon)\cap Y_0$. Choose rational $0<q<\frac{1}{2}(\epsilon-d(x,j))$. Choose a point $i$ in $D_0$ in the open ball $B_{d_0}(x,q)$. Then $x$ is in the open ball $B_{d_0}(i,q)$ and hence in the open ball $B_d(i,q)$ as well since $d\leq d_0$ everywhere. Further $(i,q)$ is in $I$: for we have $d(i,j)+q\leq d(i,x)+d(x,j)+q<2q+d(x,j)<\epsilon$.
\end{itemize}

\end{proofdetail}

\end{proof}

In order to apply this proposition to $\mathbb{S}_{\infty}$ and $[0,1]^{\mathbb{N}}$, we use:
\begin{defn}
(Computably Compact)

 A computable Polish space $Y$ is \emph{computably compact} if there is a partial computable function which when given an index for a computable sequence of c.e. opens $U_0, U_1, \ldots$, if the sequence covers $Y$ then the computable function halts and returns an $n\geq 0$ such that $U_0, \ldots, U_n$ covers $Y$.
\end{defn}

(This definition is implied by that given in \cite[Definition 11.2.10(2)]{Brattka2021-mv}).

The finite product $[0,1]^n$ is computably compact, using a basis of rectangles with rational-valued corners, and this extends to $[0,1]^{\mathbb{N}}$. Likewise, Cantor space $2^{\mathbb{N}}$ is computably compact since one can tell when a finite set of basic clopens covers the space. An important consequence of computable compactness is the following (cf. \cite[Theorem 3.14]{Porter2017-il} and references therein).
\begin{prop}\label{prop:computemax}
If $Y$ is computably compact and $f:Y\rightarrow \mathbb{R}$ is computable continuous, then both $\sup_{y\in Y} f(y)$ and $\inf_{y\in Y} f(y)$ are computable real numbers.
\end{prop}

We now prove Proposition~\ref{prop:sinfty}:
\begin{proof}
The Hilbert cube $[0,1]^{\mathbb{N}}$ with its metric $d(x,y)=\sum_n 2^{-(n+1)}  \left|x(n)-y(n)\right|$ is a computable Polish space with the dense set being the points $\psvar$ in $[0,1]^{\mathbb{N}}$ which are zero after some point. One has that $\mathbb{S}_{\infty}=\bigcap_n U_n\cap \bigcap_n V_n$, where $U_n,V_n$ are the below c.e. opens with respective relative non-empty complements $C_n,D_n$:
\begin{align*}
& U_n = \{x\in [0,1]^{\mathbb{N}}: \sum_i x(i)>1-\frac{1}{n+1}\} & C_n = \{x\in [0,1]^{\mathbb{N}}: \sum_i x(i)\leq 1-\frac{1}{n+1}\} \\
& V_n = \{x\in [0,1]^{\mathbb{N}}: \sum_{i\leq n+1} x(i)<1+\frac{1}{n+1}\} & D_n = \{x\in [0,1]^{\mathbb{N}}: \sum_{i\leq n+1} x(i)\geq 1+\frac{1}{n+1}\}
\end{align*}
Further, $C_n, D_n$ are computable Polish spaces which are computably compact:
\begin{itemize}[leftmargin=*]
\item For $C_n$: it is effectively closed in $[0,1]^{\mathbb{N}}$ since it is $f^{-1}[0, 1-\frac{1}{n+1}]$, where $f:[0,1]^{\mathbb{N}}\rightarrow [0,\infty]$ is the lsc function $f(x)=\sum_i x(i)$. Further, the intersection of the countable dense set of $[0,1]^{\mathbb{N}}$ with $C_n$ is computable and dense in $C_n$. Finally, the computable compactness of $C_n$ follows from that $[0,1]^{\mathbb{N}}$, since $C_n$ is effectively closed in $[0,1]^{\mathbb{N}}$.
\item For $D_n$: it is effectively closed in $[0,1]^{\mathbb{N}}$ since it is $g_n^{-1}[1+\frac{1}{n+1},\infty)$, where $g_n:[0,1]^{\mathbb{N}}\rightarrow [0,\infty)$ is the computable continuous function $g_n(x)=\sum_{i\leq n+1} x(i)$. The other parts follows just as for $C_n$.
\end{itemize}

For each element $x$ of the countable dense set of $[0,1]^{\mathbb{N}}$ and each $n\geq 0$, one can consider the computable continuous map $f_{x,n}:C_n\rightarrow [0,\infty)$ given by $f_{x,n}(y)=d(x,y)$. By Proposition~\ref{prop:computemax}, one has that $\inf_{y\in C_n} d(x,y)=d(x,C_n)$ is a computable real, uniformly in a point $x$ of the countable dense set and in $n\geq 0$. Since $y\mapsto d(y,C_n)$ is 1-Lipschitz, one then has that $y\mapsto d(y,C_n)$ is computable continuous, uniformly in $n\geq 0$ (cf. \cite[Proposition 2.5]{Huttegger2024-zb}). The same is true of $y\mapsto d(y,D_n)$, and hence the hypotheses of Proposition~\ref{prop:pi02eff} are satisfied.
\end{proof}

From this it follows:
\begin{prop}\label{prop:sinftyrojection}
For each $i\geq 0$, the projection maps $\pi_i:\mathbb{S}_{\infty}\rightarrow [0,1]$ given by $\pi_i(\psvar)=\psvar(i)$ are computable continuous and computable open.
\end{prop}
In this computable open means: images of c.e. opens are uniformly c.e. open.
\begin{proof}
Computable continuity follows automatically from Remark~\ref{rmk:howused}(\ref{rmk:howused:1}). For computable open, suppose that $V\subseteq \mathbb{S}_{\infty}$ is c.e. open; we want to show that its image $\pi_i(V)$ is c.e. open in $[0,1]$, uniformly in an index for $V$. Since images commute with unions, by Remark~\ref{rmk:howused}(\ref{rmk:howused:2}), it suffices to show it for a c.e. open of the form $V=\{\psvar\in \mathbb{S}_{\infty}: \psvar\in\prod_{j<n} V_j\times \prod_{j\geq n} [0,1]\}$ where $V_j$ is $(p_j,q_j)$ or $[0,q_j)$ or $(p_j,1]$ or $[0,1]$ for rational $0<p_j<q_j<1$. There are then three cases.

Case 1: $i<n$ and $1\notin V_i$. Then define $U=\{\psvar\in \prod_{j<n} V_j: \sum_{j<n} \psvar(j)< 1\}$; and note $\pi_i(V)=\pi_i(U)$.
\begin{proofdetail}
\begin{itemize}[leftmargin=*]
    \item Suppose that $r$ in $\pi_i(U)$. Then $r=\psvar(i)$ where $\psvar$ in $\prod_{j<n} V_j$ and satisfies $\sum_{j<n} \psvar(j)< 1$. Extend $\psvar$ to an element $\overline{\psvar}$ in $V$ by setting $\overline{\psvar}(n)=1-\sum_{j<n} \psvar(j)$ and $\overline{\psvar}(j)=0$ for $j>n$. Then $\overline{\psvar}$ in $V$ and so $r$ in $\pi_i(V)$.
    \item Suppose $r$ in $\pi_i(V)$. Then $r=\psvar(i)$ where $\psvar$ in $V$. If $\sum_{j<n} \psvar(j)<1$, then we are done. Suppose not. Then $\sum_{j<n} \psvar(j)=1$. Then since $1\notin V_i$ we have $n>1$. Then define $\psvar^{\prime}$ by changing $\psvar$ by decreasing some digit $<n$ besides $i$ just a little while keeping in the open box $\prod_{j<n} V_j\times \prod_{j\geq n} [0,1]$; and then by restricting to the first $n$ bits. Then $\psvar^{\prime}$ in $U$ and so $r$ in $\pi_i(U)$.
\end{itemize}
\end{proofdetail}
Since $U$ is c.e. open in $[0,1]^n$,  $\pi_i(U)$ is c.e. open in $[0,1]$.

Case 2: $i<n$ and $1\in V_i$. Then we first check whether for all $j<n$ with $j\neq i$ we have $V_j=[0,q_j)$ for a rational $0<q_j<1$ or $V_j=[0,1]$. If no, $V$ is the same as what one gets by changing $V_i$ to $V_i\setminus \{1\}$, and we are done by the previous paragraph. If yes, then we have $\pi_i(V)=V_i$.

\begin{proofdetail}
\begin{itemize}[leftmargin=*]
    \item In the one direction, we always have $\pi_i(V)\subseteq V_i$, since every point in $V$ has $i$-th coordinate in $V_i$.
    \item In the other direction, consider an arbitrary point $r$ in $V_i$, and define $\psvar(i)=r$ and $\psvar(n)=1-r$ and $\psvar(j)=0$ for all other coordinates $j$. Then $\psvar$ in $V$ and $\pi_i(\psvar)=r$.
\end{itemize}    
\end{proofdetail}

Case 3: $i\geq n$. This follows Case 1-2 by setting $V_j=[0,1]$ for all $j\in [n,i]$.
\end{proof}

We show how to extend Proposition~\ref{prop:para} to $\mathbb{S}_{\infty}$:
\begin{prop}\label{prop:sinftyprob}
Suppose that the parameter space $\pspace$ is $\mathbb{S}_{\infty}$ and that the sample space $\sspace$ is Baire space $\mathbb{N}^{\mathbb{N}}$.

Suppose that $C$ is a $\prior$-measure one effectively closed subset of $\mathbb{S}_{\infty}$.

Suppose that the the likelihood $\psvar\mapsto \pmap{\psvar}$ is given by the following for $\sigma$ in $\mathbb{N}^{<\mathbb{N}}$:
 \begin{equation*}
\pmape{\sigma}{\psvar} = \begin{cases}
\prod_{i<\left|\sigma\right|} \theta(\sigma(i))      & \text{if $\theta$ in $C$}, \\
0      & \text{otherwise}.
\end{cases}
\end{equation*}

Then conditions~(\ref{thm:doob1})-(\ref{thm:doob2}) of Theorem~\ref{thm:doob} are satisfied. 
\end{prop}

Note that the likelihood $\psvar\mapsto \pmap{\psvar}$ is defined on $C$ by taking $\pmap{\psvar}$ to be the countable product of $\psvar$, as a probability measure on the natural numbers; and that off of $C$ it is the zero element of $\mathcal{M}^+(\sspace)$.

\begin{proof}
We use these four computable continuous maps, which commute:
\begin{multicols}{2}

\begin{itemize}[leftmargin=*]
    \item[] $f_n(\ssvar)(j)=\frac{1}{n}\sum_{i<n} I_{\{j\}}(\ssvar(i))$
    \item[]
    \item[] $\pi_j(\psvar)=\psvar(j)$
    \item[]
    \item[]  $c_j(\ssvar)(i)=I_{\{j\}}(\ssvar(i))$
    \item[]
    \item[]  $g_n(y)=\frac{1}{n}\sum_{i<n} y(i)$
 \end{itemize}
\xymatrix{  & & \mathbb{S}_{\infty} \ar[drr]^{\pi_j} & &   \\
           \mathbb{N}^{\mathbb{N}} \ar[drr]_{c_j}\ar[urr]^{f_n} & &   & & [0,1]\\ 
        & & 2^{\mathbb{N}} \ar[urr]_{g_n} & & \\ }
\end{multicols}

First we verify condition (\ref{thm:doob2}) of Theorem~\ref{thm:doob}.  For each $j\geq 0$, consider the pushforward probability measure $\prior_j:=\pi_j \# \prior$ on $[0,1]$. By  Proposition~\ref{prop:fnthebasis}, uniformly in $j\geq 0$ we can compute a $\prior_j$-computable basis on $[0,1]$ such that for all $n\geq 0$ and all elements $U$ of the $p_j$-computable basis on $[0,1]$ one has that $g_n^{-1}(U)$ is uniformly both c.e. open and effectively closed in $2^{\mathbb{N}}$. Since $c_j$ is computable continuous, this means that $(g_n\circ c_j)^{-1}(U)=(\pi_j \circ f_n)^{-1}(U)$ is uniformly both c.e. open and effectively closed in $\mathbb{N}^{\mathbb{N}}$. Then a $\prior$-computable basis on $\mathbb{S}_{\infty}$ is given by elements of the form $V=\bigcap_{j<m} \pi_j^{-1}(U_j)$, where $U_j$ for $j<m$ comes from the $p_j$-computable basis on $[0,1]$. For this $V$, we have that $f_n^{-1}(V)=\bigcap_{j<m} (\pi_j\circ f_n)^{-1}(U_j)$; thus $f_n^{-1}(V)$ is both c.e. open and effectively closed in $\sspace$.

Second we verify condition (\ref{thm:doob1}) of Theorem~\ref{thm:doob}. For each $j\geq 0$, consider the pushforward probability measure $\prodm_j:=(\pi_j\otimes c_j)\# \prodm$ on $[0,1]\times 2^{\mathbb{N}}$, and for each $j\geq 0$ and $\psvar$ in $\mathbb{S}_{\infty}$ consider the pushforward measure $\pmapi{j}{\psvar}:=c_j\# \pmap{\psvar}$ on Cantor space~$2^{\mathbb{N}}$. On the $\prior$-measure one set $C$, this is the same as the probability measure on Cantor space which gives heads probability $\psvar(j)$ on each flip. Using the Greek letter $\rho$ (rho), let $\rho(\cdot\hspace{-.5mm}\mid\hspace{-.5mm}\upsilon)$ be the probability on Cantor space parameterized by $\upsilon$ from $[0,1]$ which gives heads probability $\upsilon$ on each flip; then we have $\pmapi{j}{\psvar}=\rho(\cdot \hspace{-.5mm}\mid\hspace{-.5mm} \psvar(j))\cdot I_C(\psvar)$. Further one has the equation $\prodm_j((a,b)\times [\sigma]) = \int_a^b \rho([\sigma]\hspace{-.5mm}\mid\hspace{-.5mm} \upsilon) \; d \prior_j(\upsilon)$, which says that the pushforward $\prodm_j$ is the same as the joint distribution given by prior $\prior_j$ on $[0,1]$ and likelihood $\upsilon\mapsto \rho(\cdot\hspace{-.5mm}\mid\hspace{-.5mm}\upsilon)$. 
\begin{proofdetail}
To see this, first note that $\pi_j^{-1}(\pi_j(C))\supseteq C$ and so both have $\prior$-measure one, and hence the effectively closed set $\pi_j(C)$ has $\prior_j$ measure 1. Then we have the following:
\begin{align*}
& \prodm_j((a,b)\times [\sigma]) = \prodm(\pi_j^{-1}(a,b)\times c_j^{-1}([\sigma])) = \int_{\pi_j^{-1}(a,b)} \pmapews{c_j^{-1}([\sigma])}{\psvar} \; d\prior(\psvar) \\
& = \int_{\pi_j^{-1}(a,b)\cap \pi_j^{-1}(\pi_j(C))} \rho([\sigma] \hspace{-.5mm}\mid\hspace{-.5mm} \pi_j(\psvar)) \; d\prior(\psvar) = \int_{(a,b)\cap \pi_j(C)} \rho([\sigma]\hspace{-.5mm}\mid\hspace{-.5mm} \upsilon) \; d \prior_j(\upsilon) \\
& =\int_a^b \rho([\sigma]\hspace{-.5mm}\mid\hspace{-.5mm} \upsilon) \; d \prior_j(\upsilon)
\end{align*}
where the second-to-last identity follows from change of variables. 
    
\end{proofdetail}
We now finish the verification of condition (\ref{thm:doob1}) of Theorem~\ref{thm:doob}. It suffices to show that for all $(\psvar, \ssvar)$ in $\mathsf{SR}^{\prodm}$, one has that $\lim_n f_n(\ssvar) = \psvar$. To show this is to show that, for all $j\geq 0$, we have $\lim_n \pi_j(f_n(\ssvar)) = \pi_j(\psvar)$, which by the diagram is the same as $\lim_n g_n(c_j(\ssvar))=\psvar(j)$. Let $j\geq 0$. Since $(\psvar, \ssvar)$ is in $\mathsf{SR}^{\prodm}$ and since $\pi_j\otimes c_j:\prodspace\rightarrow [0,1]\times 2^{\mathbb{N}}$ is computable continuous, we have that $(\psvar(j), c_j(\ssvar))$ is in $\mathsf{SR}^{\prodm_j}$. By Proposition~\ref{prop:limrelfreqs} one has that $\lim_n g_n(c_j(\ssvar))=\psvar(j)$, which is what we wanted to show. 
\end{proof}

The following proposition is a preliminary to our result on identifiability:
\begin{prop}\label{prop:superspace}
Suppose that the parameter space $\pspace$ is an effectively closed subset of a computable Polish space $\pspacealt$, and that $\pspace$ is a computable Polish subspace of $\pspacealt$.

Define extended priors and likelihoods on the superspace $\pspacealt$ as follows:
\begin{enumerate}[leftmargin=*]
    \item Define prior $\prioralt$ on $\pspacealt$ by pushforward $\prioralt=\iota\# \prior$, where $\iota:\pspace\rightarrow \pspacealt$ is the identity.
    \item Define likelihood $\widetilde{\psvar}\mapsto \pmapalt{\widetilde{\psvar}}$ on $\pspacealt$ by $\pmapalt{\widetilde{\psvar}}=\pmap{\widetilde{\psvar}}\cdot I_{\pspace}(\widetilde{\psvar})$.
\end{enumerate}
Suppose that the conditions~(\ref{thm:doob1})-(\ref{thm:doob2}) of Theorem~\ref{thm:doob} hold with respect to $\pspacealt$, $\prioralt$, and $\widetilde{\psvar}\mapsto \pmapalt{\widetilde{\psvar}}$.

Then all $\psvar$ in $\mathsf{SR}^{\prior}$ are computably consistent relative to $\pspace$, $\prior$, and $\psvar\mapsto \pmap{\psvar}$.
\end{prop}
\begin{proof}
Note that $\iota^{-1}(A)=A\cap \pspace$ for all events $A\subseteq \pspacealt$. Further note that $\pspacealt\setminus \pspace$ is a $\prioralt$-null c.e. open in $\pspacealt$. If $\prodmalt$ denotes the new joint distribution on $\pspacealt\times \sspace$ with respect to the extended prior and likelihood, then for Borel events $A\subseteq \pspacealt$ and $B\subseteq \sspace$, one has $\prodmalt(A\times B)=\prodm((A\cap \pspace)\times B)$. Further for Borel events $A\subseteq \pspacealt$ we have $\postalte{A}{\ssvarn{n}} =\poste{A\cap \pspace}{\ssvarn{n}}$ when the denominator is non-zero.

\begin{proofdetail}

To see these points, note: $\prodm(\iota^{-1}(A)\times B)=\int_{\iota^{-1}(A)} \pmape{B}{\psvar} \; d\prior(\psvar) = \int_{A\cap \pspace} \pmapalte{B}{\psvar} \; d\prior(\psvar)=\int_{A} \pmapalte{B}{\psvar} \; d\prioralt(\psvar)=\prodmalt(A\times B)$. Then for Borel events $A\subseteq \pspacealt$ we have $\postalte{A}{\ssvarn{n}} = \frac{\prodmalt(A\times [\ssvarn{n}])}{\prodmalt(\pspacealt\times [\ssvarn{n}])}=\frac{\prodm(\iota^{-1}(A)\times [\ssvarn{n}])}{\prodm(\pspace\times [\ssvarn{n}])}=\poste{\iota^{-1}(A)}{\ssvarn{n}}=\poste{A\cap \pspace}{\ssvarn{n}}$ when the denominator is non-zero.

\end{proofdetail}

We show that for any version of the original likelihood $\pmape{\sigma}{\cdot}$ in the sense of Definition~\ref{defn:conventionversions}, there is a version of extended likelihood $\pmapalte{\sigma}{\cdot}$ in the sense of Definition~\ref{defn:conventionversions} which agree on $\mathsf{SR}^{\prioralt}$. For, suppose that for a version of $\psvar\mapsto \pmape{\sigma}{\psvar}$ we use the difference $\gmape{\sigma}{\psvar}-\hmape{\sigma}{\psvar}$, where $\gmape{\sigma}{\cdot}, \hmape{\sigma}{\cdot}:\pspace\rightarrow [0,\infty]$ are $L_1(\prior)$ Schnorr tests. Since $\pspacealt\setminus \pspace$ is c.e. open in $\pspacealt$, we have that $\gmape{\sigma}{\widetilde{\psvar}} + \hmape{\sigma}{\widetilde{\psvar}} \cdot I_{\pspacealt\setminus \pspace}(\widetilde{\psvar})$ and $\gmape{\sigma}{\widetilde{\psvar}}\cdot I_{\pspacealt\setminus \pspace}(\widetilde{\psvar}) + \hmape{\sigma}{\widetilde{\psvar}}$ are $L_1(\prioralt)$ Schnorr tests. Further, on $\mathsf{SR}^{\prioralt}$, their difference is equal to $\big(\gmape{\sigma}{\widetilde{\psvar}}-\hmape{\sigma}{\psvar}\big)\big(1-I_{\pspacealt\setminus \pspace}(\widetilde{\psvar})\big) =\big(\gmape{\sigma}{\widetilde{\psvar}}-\hmape{\sigma}{\widetilde{\psvar}}\big)(I_{\pspace}(\widetilde{\psvar}))$.

Suppose that $\psvar$ is in $\mathsf{SR}^{\prior}$. Since $\iota:\pspace\rightarrow \pspacealt$ is computable continuous, we have that $\psvar$ is in $\mathsf{SR}^{\prioralt}$. Then by Theorem~\ref{thm:doob} we have that $\psvar$ is computable consistent relative to $\pspacealt$, $\prioralt$, and $\widetilde{\psvar}\mapsto \pmapalt{\widetilde{\psvar}}$. Let $U\subseteq \pspace$ be c.e. open with $\prior(U)$ computable. Since $\pspace$ is a computable Polish subspace of $\pspacealt$ we have that $U=\pspace\cap V$, where $V$ is a c.e. open in $\pspacealt$. Then $\prioralt(V)=\prior(\pspace\cap V)=\prior(U)$ is computable. Then $\lim_n \postalte{V}{\ssvarn{n}}=I_{V}(\psvar)$ for $\postalt{\psvar}$-many $\ssvar$ in $\sspace$. Since $\psvar$ is in $\pspace$, we have $\lim_n \postalte{V}{\ssvarn{n}}=I_{V}(\psvar)$ for $\post{\psvar}$-many $\ssvar$ in $\sspace$. By the first paragraph, we have $\lim_n \poste{V\cap \pspace}{\ssvarn{n}}=I_{V\cap \pspace}(\psvar)$ for $\post{\psvar}$-many $\ssvar$ in $\sspace$. By $U=\pspace\cap V$, we have that $\lim_n \poste{U}{\ssvarn{n}}=I_{U}(\psvar)$ for $\post{\psvar}$-many $\ssvar$ in $\sspace$.

\end{proof}

Here is then the proof of Proposition~\ref{thm:effectiveidentifiability}:
\begin{proof}
By passing to the image, we may assume that $\pspace$ is an effectively closed subset of $\pspacealt:=\mathbb{S}_{\infty}$, and that $\pspace$ is a computable Polish subspace of $\pspacealt$. Define the extended prior and likelihood relative to $\pspacealt$ as in the statement of Proposition~\ref{prop:superspace}. Then by Proposition~\ref{prop:sinftyprob} the conditions~(\ref{thm:doob1})-(\ref{thm:doob2}) of Theorem~\ref{thm:doob} hold with respect to $\pspacealt$ and the extended prior and likelihood. Then we are done by Proposition~\ref{prop:superspace}.
\end{proof}

%% file: 08-general-con.tex
\section{Sufficient conditions for inconsistency}\label{sec:suffincon}

Up to this point in the paper, we have been assuming that the prior $\prior$ is a fixed {\it computable} probability measure on the parameter space $\pspace$, as in the effective assumptions  (Definition~\ref{defn:assum}(\ref{assum2})). To handle Freedman's Inconsistency Theorem, we need to shift to considering the prior to be an {\it arbitrary} probability measure on the parameter space $\pspace$, that is, an arbitrary element of $\mathcal{P}(\pspace)$.

We note an elementary proposition, which makes salient a double integral:

\begin{prop}\label{prop:DCTapp}
 If $(\prior, \psvar)$ is classically consistent, then for all closed $C$ not containing $\psvar$ one has:
\begin{equation}
\label{prop:DCTapp:3} \lim_n \int_{\sspace} \poste{C}{\ssvarn{n}} \; d\pmap{\psvar}(\ssvar)=0, \hspace{5mm}
  \lim_n \int_{\sspace} \int_{\pspace} I_C(\psvar^{\prime}) \; d\post{\ssvarn{n}}(\psvar^{\prime}) \; d\pmap{\psvar}(\ssvar)=0
\end{equation}
The same holds true if $\prior$ is computable and $(\prior,\psvar)$ is computably consistent.
\end{prop}
\begin{proof}
One uses the hypothesis and the ``open set'' form of Portmanteau to get that $\post{\ssvarn{n}}\rightarrow \delta_{\psvar}$ in $\mathcal{P}(\pspace)$ for $\pmap{\psvar}$-a.s. many $\ssvar$ in $\sspace$. Then one uses the ``closed set'' form of Portmanteau and the Dominated Converence Theorem to get the conclusion.

The first step (the ``open set'' step) of the proof also works under the effective hypothesis since the c.e. opens with computable $\prior$-measure form a topological basis.
\end{proof}

\begin{proofdetail}
Here is a more detailed version of the proof:
\begin{proof}
We can focus on the single integral since the double integral follows by writing $\poste{C}{\ssvarn{n}} = \int_{\pspace} I_C(\psvar^{\prime}) \; d\post{\ssvarn{n}}(\psvar^{\prime})$.

Suppose that $V_i$ is a countable basis for $\pspace$ (resp. a $\prior$-computable basis). By classical (resp. computable) consistency, one has that $\lim_n \poste{V_i}{\ssvarn{n}}=I_{V_i}(\psvar)$ on a $\pmap{\psvar}$-measure one set $A_i$ of points $\ssvar$ in $\sspace$. Then $A=\bigcap_i A_i$ also has $\pmap{\psvar}$-measure one.

Note that if $x$ in $A$, then $\poste{\cdot}{\ssvarn{n}}$ is in $\mathcal{P}(\sspace)$ for all $n\geq 0$. Suppose for reductio that there was $n_0\geq 0$ with $\poste{\cdot}{\ssvarn{n_0}}$ being the zero element of $\mathcal{M}^+(\pspace)$. Then by definition of the posterior in (\ref{eqn:post}), one has that $\int_{\pspace} \pmape{\ssvarn{n_0}}{\psvar}\; d\prior(\psvar) =0$. But since for $n\geq n_0$ one has $[\ssvarn{n}]\subseteq [\ssvarn{n_0}]$ one has that $\int_{\pspace} \pmape{\ssvarn{n}}{\psvar}\; d\prior(\psvar) =0$ as well. But this contradicts that $\lim_n \poste{V_i}{\ssvarn{n}}=1$ for any $V_i$ containing $\psvar$.

Now we claim that for any $\ssvar$ in $A$ and any open $V$ we have $\liminf_n \poste{V}{\ssvarn{n}}\geq \delta_{\psvar}(V)$. This is trivial if $\psvar$ is not in $V$. If $\psvar$ in $V$, then choose $i$ such that $\psvar$ in $V_i$ and $V_i\subseteq V$. Then since $\ssvar$ in $A_i$ we have $\liminf_n \poste{V}{\ssvarn{n}}\geq \liminf_n \poste{V_i}{\ssvarn{n}}=1=\delta_{\psvar}(V)$. 

By the Portmanteau Theorem we then have that for all $\ssvar$ in $A$ that $\post{\ssvarn{n}}\rightarrow \delta_{\psvar}$ in $\mathcal{P}(\pspace)$. 

Let $C$ be a closed set which does not contain $\psvar$. Then by Portmanteau again, we have for any $\ssvar$ in $A$ that $\limsup_n \poste{C}{\ssvarn{n}}\leq \delta_{\psvar}(C)=0$. Then for any $\ssvar$ in $A$ one has $\lim_n \poste{C}{\ssvarn{n}}=0$.

Then apply the Dominated Convergence Theorem with respect to $\pmap{\psvar}$ to obtain $\lim_n \int_{\sspace} \poste{C}{\ssvarn{n}} \; d\pmap{\psvar}(\ssvar)= \int_{\sspace} \lim_n\poste{C}{\ssvarn{n}} \; d\pmap{\psvar}(\ssvar) = \int_{\sspace}0 \; d\pmap{\psvar}(\ssvar)=0$.

\end{proof}
\end{proofdetail}

The following proposition provides a sufficient condition for inconsistency:
\begin{prop}\label{prop:thelprop}
Suppose that $\pspace$ has at least two elements.  Suppose that $(\prior, \psvar)$ is classically consistent. Then there is a c.e. open $V$ such that 
\begin{equation*}
\limsup_n \int_{\sspace} \int_{\pspace} I_{V}(\psvar^{\prime}) \; d\post{\ssvarn{n}}(\psvar^{\prime}) \; d\pmap{\psvar}(\ssvar) <1
\end{equation*}
The same holds if $\prior$ is computable and $(\prior,\psvar)$ is computably consistent.
\end{prop}
\begin{proof}
Since $\pspace$ has at least two elements, choose disjoint c.e. opens $U,V$ such that $\psvar$ is in $U$. Let $C=\pspace\setminus U$ be the relative complement, so that $V\subseteq C$. Since  $(\prior, \psvar)$ is classically consistent, by Proposition~\ref{prop:DCTapp} we have $\lim_n \int_{\sspace} \poste{C}{\ssvarn{n}} \; d\pmap{\psvar}(\ssvar)=0$. Since $V\subseteq C$ we have $\lim_n \int_{\sspace} \poste{V}{\ssvarn{n}} \; d\pmap{\psvar}(\ssvar)=0$. From this the conclusion follows easily.
\end{proof}

Now we discuss how to ensure that this double integral is lsc:
\begin{prop}\label{prop:freeprep}
(Integrals lsc as function of prior, when likelihood comp. cont.).

Suppose the likelihood $\psvar\mapsto \pmap{\psvar}$ maps from $\pspace$ into $\mathcal{P}(\sspace)$. Suppose further the likelihood $\psvar\mapsto \pmape{\sigma}{\psvar}$ is computable continuous, uniformly in $\sigma$ from $T$. Then:

\begin{enumerate}[leftmargin=*]
    \item \label{prop:freeprep:1} The map $(\prior,\ssvar)\mapsto \poste{U}{\ssvarn{n}}$ is lsc from $\mathcal{P}(\pspace)\times \sspace$ to $[0,1]$, uniformly in c.e. open $U\subseteq \pspace$ and $n\geq 0$.
    \item \label{prop:freeprep:2} The map $(\prior,\ssvar)\mapsto \int_{\pspace} f(\psvar^{\prime}) \; d\post{\ssvarn{n}}(\psvar^{\prime})$ is lsc from $\mathcal{P}(\pspace)\times \sspace$ to $[0,\infty]$, uniformly in $n\geq 0$ and lsc $f:\pspace\rightarrow [0,\infty]$.
    \item \label{prop:freeprep:3} The map $(\prior, \psvar)\mapsto \int_{\sspace} g(p,\ssvar^{\prime}) \; d\pmap{\psvar}(\ssvar^{\prime})$ is lsc from  $\mathcal{P}(\pspace)\times \pspace$ to $[0,\infty]$, uniformly in lsc $g:\mathcal{P}(\pspace)\times \sspace\rightarrow [0,\infty]$.
    \item \label{prop:freeprep:4} The map $(\prior, \psvar)\mapsto \int_{\sspace} \int_{\pspace} f(\psvar^{\prime}) \; d\post{\ssvarn{n}}(\psvar^{\prime}) \;d\pmap{\psvar}(\ssvar)$ is lsc from  $\mathcal{P}(\pspace)\times \pspace$ to $[0,\infty]$, uniformly in $n\geq 0$ and lsc $f:\pspace\rightarrow [0,\infty]$.
\end{enumerate}
\end{prop}
Note that (\ref{prop:freeprep:1}) improves Proposition~\ref{prop:posteriorintegral}, since the latter just concerned a fixed computable prior.

\begin{proof}
For (\ref{prop:freeprep:1}), the hypothesis of the proposition implies that uniformly in c.e. open $U\subseteq \pspace$ and $\sigma$ in $T$, one has that $\psvar \mapsto I_U(\psvar) \cdot \pmape{\sigma}{\psvar}$ is lsc from $\pspace$ to $[0,1]$. Then uniformly in c.e. open $U\subseteq \pspace$ and $\sigma$ in $T$ one has that $\prior\mapsto \int_U  \pmape{\sigma}{\psvar} \; d\prior(\psvar)$ is lsc from $\mathcal{P}(\pspace)$ to $[0,1]$, by Hoyrup-Rojas \cite[Proposition 4.3.1]{Hoyrup2009-pl}.

Further,  the hypothesis of the proposition, implies that $\psvar \mapsto (1- \pmape{\sigma}{\psvar})$ is also lsc from $\pspace$ to $[0,1]$, uniformly in $\sigma$ from $T$. Then uniformly in $\sigma$ from $T$ one has that $\prior\mapsto \int_{\pspace}  (1-\pmape{\sigma}{\psvar}) \; d\prior(\psvar)$ is lsc from $\mathcal{P}(\pspace)$ to $[0,1]$ by \cite[Proposition 4.3.1]{Hoyrup2009-pl}. Hence, uniformly in $\sigma$ from $T$ one has that $\prior\mapsto \int_{\pspace}  \pmape{\sigma}{\psvar} \; d\prior(\psvar)$ is usc (that is, upper semi-computable) from $\mathcal{P}(\pspace)$ to $[0,1]$. By the previous paragraph with $U=\pspace$, one has that uniformly in $\sigma$ from $T$ the map $\prior\mapsto \int_{\pspace}  \pmape{\sigma}{\psvar} \; d\prior(\psvar)$ is computable continuous from $\mathcal{P}(\pspace)$ to $[0,1]$.

To finish (\ref{prop:freeprep:1}), it suffices to note by the definition of the posterior $\poste{U}{\ssvarn{n}}$ in (\ref{eqn:post}) that for rational $q>0$ we have $\poste{U}{\ssvarn{n}}>q$ iff there is $\sigma$ in $T$ of length~$n$ with $\ssvar$ in $[\sigma]$ and there is rational $r>0$ such that $\int_U  \pmape{\sigma}{\psvar} \; d\prior(\psvar)>r$ and $0< \int_{\pspace}  \pmape{\sigma}{\psvar} \; d\prior(\psvar)<\frac{r}{q}$. These inequalities are c.e. open conditions in $(\prior,\ssvar)$ by the previous two paragraphs.

For (\ref{prop:freeprep:2}), any lsc function $f:\Omega\rightarrow [0,\infty]$ can be written as $f=\sup_s f_s$, where $f_s$ is a computable sequence of simple functions of the form $\sum_{i=1}^{n_s} q_{s,i}\cdot I_{U_{s,i}}$, where $n_s\geq 0$ is computable and $q_{s,i}$ is a computable sequence of non-negative rationals and $U_{s,i}$ is a computable sequence of c.e. opens in $\Omega$ (cf. proof of Hoyrup-Rojas \cite[Proposition 4.3.1]{Hoyrup2009-pl}). Hence one has $\int_{\pspace} f(\psvar^{\prime}) \; d\postws{\ssvarn{n}}(\psvar^{\prime})>r $ iff there is $s\geq 0$ such that $\sum_{i=1}^{n_s} q_{s,i} \cdot \poste{U_{s,i}}{\ssvarn{n}}>r$, which is a c.e. open condition in $(p,\ssvar)$ by (\ref{prop:freeprep:1}).

 For (\ref{prop:freeprep:3}), any lsc $g:\mathcal{P}(\pspace)\times \sspace\rightarrow [0,\infty]$ can be written as $g=\sup_s g_s$, where $g_s$ is a computable sequence of simple functions of the form $\sum_{i=1}^{n_s} q_{s,i}\cdot I_{W_{s,i}}$, where $n_s\geq 0$ is computable and $q_{s,i}$ is a computable sequence of non-negative rationals and $W_{s,i}$ is a computable sequence of c.e. opens in $\mathcal{P}(\pspace)\times \sspace$ (cf. again proof of Hoyrup-Rojas \cite[Proposition 4.3.1]{Hoyrup2009-pl})). Let $W_{s,i}=\bigcup_{j} (U_{s,i,j}\times V_{s,i,j})$ for some computable sequences $U_{s,i,j}$ of c.e. opens in $\mathcal{P}(\pspace)$ and $V_{s,i,j}$ of c.e. opens in $\sspace$.
 Then one has that $\int_{\sspace} g(p,\ssvar^{\prime}) \; d\pmap{\psvar}(\ssvar^{\prime})>r$ iff there is $s\geq 0$ and there are finite sets of natural numbers $\{J_{s,i}: i<n_s\}$ such that $p$ is in $\bigcap_{i<n_s} \bigcap_{j\in J_{s,i}} U_{s,i,j}$ and $\sum_{i=1}^{n_s} q_{s,i} \cdot \pmape{\bigcup_{j\in J_{s,i}} V_{s,i,j}}{\psvar}>r$, which is a c.e. open condition in $(p,\theta)$.


Finally, (\ref{prop:freeprep:4}) follows directly from (\ref{prop:freeprep:2})-(\ref{prop:freeprep:3}).

\begin{proofdetail}
In more detail:

 For (\ref{prop:freeprep:4}) suppose $f:\pspace\rightarrow [0,\infty]$ is lsc and $n\geq 0$. Let $g_n(\prior,\ssvar) = \int_{\pspace} f(\psvar^{\prime}) \; d\post{\ssvarn{n}}(\psvar^{\prime})$, which is lsc from $\mathcal{P}(\pspace)\times \sspace$ to $[0,\infty]$ uniformly in $n\geq 0$ and $f$ by (\ref{prop:freeprep:2}). Further, $(\prior,\psvar)\mapsto \int_{\sspace} g_n(p,\ssvar^{\prime}) \; d\pmap{\psvar}(\ssvar^{\prime})$ is lsc from  $\mathcal{P}(\pspace)\times \pspace$ to $[0,\infty]$, uniformly in $n\geq 0$ and $f$ by (\ref{prop:freeprep:3}). By substituting the definition of $g_n$, we have $(\prior, \psvar)\mapsto \int_{\sspace} \int_{\pspace} f(\psvar^{\prime}) \; d\post{\ssvarn{n}}(\psvar^{\prime}) \;d\pmap{\psvar}(\ssvar)$ is lsc from  $\mathcal{P}(\pspace)\times \pspace$ to $[0,\infty]$, uniformly in $n\geq 0$ and $f:\pspace\rightarrow [0,\infty]$.    
\end{proofdetail}
 
\end{proof}

%% file: 09-incon.tex
\section{Effective Freedman inconsistency}\label{sec:efincon}

We begin with the proof of Theorem~\ref{thm:freedman}, which effectivizes the proof in \cite[Theorem 6.12]{Ghosal2017-jk}. In this theorem, recall that the parameter space $\pspace$ is  $\mathbb{S}_{\infty}$, and so we just write $\mathbb{S}_{\infty}$ instead of $\pspace$ in this section.

Note that we have the formula $\pmape{\sigma}{\psvar} = \prod_{i<\left|\sigma\right|} \psvar(\sigma(i))$ for the likelihood, and that $\psvar \mapsto \prod_{i<\left|\sigma\right|} \psvar(\sigma(i))$ is computable continuous from $[0,1]^{\mathbb{N}}$ to $[0,1]$. By Proposition~\ref{prop:sinfty} and Remark~\ref{rmk:howused}(\ref{rmk:howused:1}), we have that $\psvar \mapsto \pmape{\sigma}{\psvar}$ is computable continuous from $\mathbb{S}_{\infty}$ to $[0,1]$, which puts us in a position to later apply Proposition~\ref{prop:freeprep}.

Further, $\mathbb{S}_{\infty}^+:=\{\psvar\in \mathbb{S}_{\infty} : \forall \; i\geq 0\; \psvar(i)>0\}$ is effectively comeager in $\mathbb{S}_{\infty}$. For, the set $\{\psvar\in \mathbb{S}_{\infty} : \psvar(i)>0\}$ is c.e. open and dense for each $i\geq 0$: for, it is $\pi_i^{-1}((0,1])$, where $\pi_i:\mathbb{S}_{\infty}\rightarrow [0,1]$ is the projection map, which is computable continuous and computable open by Proposition~\ref{prop:sinftyrojection}. By the Baire Category Theorem, $\mathbb{S}_{\infty}^+$ is dense in $\mathbb{S}_{\infty}$. By the Effective Baire Category Theorem (\cite[Theorem 5.1.2]{Hoyrup2009-pl}) we can find a uniformly computable sequence of points $\psvar_0^+, \psvar_1^+, \ldots$ in $\mathbb{S}_{\infty}^+$ which is dense in $\mathbb{S}_{\infty}$.

Further, $\mathbb{S}_{\infty}^-=\{\psvar\in \mathbb{S}_{\infty} : \exists \; i\geq 0\; \psvar(i)=0\}$ is dense since it is a superset of the countable dense set. 

Finally, let $D_i$ be the set of $\sum_{j=1}^m q_j \delta_{\psvar_j}$, where $0<q_j<1$ is rational and $\sum_{j=1}^m q_j =1$, and $\psvar_1=\psvar^+_i$ and for $1<j\leq m$ one has that $\psvar_j$ is from $\mathbb{S}_{\infty}^-$. It is not hard to see, by a simple classical argument, that this is dense in $\mathcal{P}(\mathbb{S}_{\infty})$.

\begin{proofdetail}

 Here is the classical argument. For, suppose that $\nu$ in $\mathcal{P}(\mathbb{S}_{\infty})$ is given. We can find a point of the form $\sum_{j=2}^m q_j \delta_{\psvar_j}$, where $0<q_j<1$ is rational and $\sum_{j=2}^m q_j =1$ and $\psvar_j$ is from the countable dense set of $\mathbb{S}_{\infty}$, which is arbitrarily close to $\nu$, since these finite rational averages of Diracs are dense in $\mathcal{P}(\mathbb{S}_{\infty})$ (indeed, this is the traditional choice of a countable dense set in $\mathcal{P}(\mathbb{S}_{\infty})$). Further, since $\mathbb{S}_{\infty}^-$ is a superset of the countable dense set, we have that $\psvar_j$ for $2\leq j\leq m$ is from  $\mathbb{S}_{\infty}^-$. As $\epsilon\rightarrow 0$, one has that $\big(\epsilon \delta_{\psvar_i^+}+(1-\epsilon)\cdot \sum_{j=2}^m q_j \delta_{\psvar_j}\big)\rightarrow \sum_{j=2}^m q_j \delta_{\psvar_j}$ in $\mathcal{P}(\mathbb{S}_{\infty})$. Then by taking $\epsilon$ a sufficiently small rational, we are done.

\end{proofdetail}

Then one has:
\begin{prop}\label{prop:helperpost}
Let $\psvar_0$ be from $\mathbb{S}_{\infty}^+$. Let $\prior$ be from $D_i$. Then for $\pmap{\psvar_0}$-a.s. many $\ssvar$ in $\sspace$ one has $\post{\ssvarn{n}}\rightarrow \delta_{\psvar_i^+}$ in $\mathcal{P}(\mathbb{S}_{\infty})$.
\end{prop}
\begin{proof}
Let $\prior = \sum_{j=1}^m q_j \delta_{\psvar_j}$, where this is as in the definition of $D_i$ in the previous paragraph. The non-trivial case is $m>1$. For each $1<j\leq m$, choose $k_j\geq 0$ with $\psvar_j(k_j)=0$. Let $\epsilon = \prod_{j=2}^m \psvar_0(k_j)$, which since  $\psvar_0$ is in $\mathbb{S}_{\infty}^+$ is necessarily $>0$. Let $A=\bigcap_{\ell} A_{\ell}$, where $A_{\ell} = \{ \ssvar\in \sspace: \forall \; 1<j\leq m \; \exists \; n\geq \ell \; \ssvar(n)=k_j\}$. Then $A$ is a tailset, i.e. if $\ssvar, \ssvar^{\prime}$ differ by finitely much and one is in $A$, then the other is as well. Since $\pmap{\psvar_0}$ is a product measure, by the zero-one law, $A$ either has measure zero or one. We claim it has measure one. If not, then it has measure zero. Then choose sufficiently large $\ell\geq 0$ such that $\pmape{A_{\ell}}{\psvar_0}<\epsilon$. Consider the event $B=\{\ssvar: \ssvar(\ell+2)=k_2, \ssvar(\ell+3)=k_3, \ldots, \ssvar(\ell+m)=k_m\}$. Then by independence we have $\pmape{B}{\psvar_0}= \prod_{j=2}^m \psvar_0(k_j)$, which is $\epsilon$, a contradiction to $B$ being a subset of $A_{\ell}$. Hence rather we have that $\pmape{A}{\psvar_0}=1$.

By Portmanteau it suffices to check that for all $\ssvar$ in $A$ and all open $U$ containing $\psvar_1$ one has $\liminf_n \poste{U}{\ssvarn{n}}=1$. Suppose not. Then there is $\ssvar$ in $A$ and open $U$  containing $\psvar_1$ such that $\liminf_n \poste{U}{\ssvarn{n}}< 1$. Then there is $\ell_0\geq 0$ such that for all $\ell\geq \ell_0$ one has $\inf_{n\geq \ell} \poste{U}{\ssvarn{n}}< 1$. Since $\ssvar$ is in $A_{\ell_0}$, choose $\ell>\ell_0$ such that for all $1<j\leq m$ one has that $\ssvar(\ell^{\prime})=k_j$ for some $\ell^{\prime}$ in the interval $[\ell_0, \ell)$. This has the consequence that for all $n\geq \ell$ one has that $\pmape{\ssvarn{n}}{\psvar_j}=0$ for all $1<j\leq m$. Hence for any $n\geq \ell$ and any event $C$, we have that 
\begin{equation*}
\int_C \pmape{\ssvarn{n}}{\psvar} \; d\prior(\psvar) =\sum_{j=1}^m q_j \cdot I_C(\psvar_j) \cdot \pmape{\ssvarn{n}}{\psvar_j} = q_1 \cdot I_C(\psvar_1) \cdot \pmape{\ssvarn{n}}{\psvar_1}=q_1 \cdot I_C(\psvar_i^+) \cdot \pmape{\ssvarn{n}}{\psvar_i^+}
\end{equation*}
When we set $C=\mathbb{S}_{\infty}$, then we get a value $>0$ since $\psvar_i^+$ comes from $\mathbb{S}_{\infty}^+$. For $n\geq \ell$:
\begin{equation*}
\poste{U}{\ssvarn{n}}= \frac{\int_U \pmape{\ssvarn{n}}{\psvar}\; d\prior(\psvar)}{\int_{\mathbb{S}_{\infty}} \pmapews{\ssvarn{n}}{\psvar}\; d\prior(\psvar)} = \frac{ q_1 \cdot I_U(\psvar_i^+) \cdot \pmape{\ssvarn{n}}{\psvar_i^+}   }{q_1 \cdot \pmape{\ssvarn{n}}{\psvar_i^+} }=I_U(\psvar_i^+)=1
\end{equation*}
and this contradicts that  $\inf_{n\geq \ell} \poste{U}{\ssvarn{n}}< 1$.

\end{proof}

For each $i\geq 0$ and each $\psvar_i^+$ on our list, we construct a uniformly computable continuous sequence $\chi_{i,k}: \mathbb{S}_{\infty} \rightarrow [0,1]$ such that (i) $\chi_{i,k}(\psvar_i^+)=1$ for all $k\geq 0$, and (ii) if $\psvar^+_i$ is in the open $V$, then eventually  $\chi_{i,k}\leq I_V$. Let $\epsilon_k\rightarrow 0$ be a computable sequence of positive rationals which is strictly decreasing. Let $V_{i,k} = B(\psvar^+_i, \epsilon_k)$, noting that $\overline{V_{i,k+1}}\subseteq B[\psvar^+_i, \epsilon_{k+1}]\subseteq B(\psvar^+_i, \epsilon_k)=V_{i,k}$. Since $B[\psvar^+_i, \epsilon_{k+1}]$ and $\mathbb{S}_{\infty} \setminus B(\psvar^+_i, \epsilon_k)$ are effectively closed and disjoint, by the effective Urysohn Lemma (\cite[Lemma II.7.(1), Lemma II.7.3 p. 90]{Simpson2009aa}) there are uniformly computable continuous $\chi_{i,k}: \mathbb{S}_{\infty} \rightarrow [0,1]$ such that $\chi_{i,k}\upharpoonright B[\psvar^+_i, \epsilon_{k+1}]=1$ and $\chi_{i,k}\upharpoonright (\mathbb{S}_{\infty} \setminus B(\psvar^+_i, \epsilon_k))=0$. Then $\chi_{i,k}\upharpoonright V_{i,k+1}=1$ and  $\chi_{i,k}\upharpoonright (\mathbb{S}_{\infty} \setminus V_{i,k})=0$.

By Proposition~\ref{prop:freeprep}(\ref{prop:freeprep:4}), the following are c.e. opens, where $0<\epsilon<1$ is rational: 
\begin{equation}
U_{i,\epsilon,m,k}=\bigcup_{n\geq m} \{ (\prior, \psvar): \int_{\sspace} \int_{\mathbb{S}_{\infty}} \chi_{i,k}(\psvar^{\prime}) \; d\post{\ssvarn{n}} (\psvar^{\prime}) \; d\pmap{\psvar}(\ssvar) > 1-\epsilon\}
\end{equation}

We show that each of these c.e. opens is dense in $\mathcal{P}(\mathbb{S}_{\infty})\times \mathbb{S}_{\infty}$. Suppose that a pair in $\mathcal{P}(\mathbb{S}_{\infty})\times \mathbb{S}_{\infty}$ is given. Find $(\prior, \psvar_0)$ in $D_i\times \mathbb{S}_{\infty}^+$ arbitrarily close to the given pair. We argue that the pair $(\prior, \psvar_0)$ is in $U_{i,\epsilon,m,k}$. By Proposition~\ref{prop:helperpost}, one has that $\post{\ssvarn{n}}\rightarrow \delta_{\psvar_i^+}$ for $\pmap{\psvar_0}$-a.s. many $\ssvar$ in $\sspace$. By the Portmanteau Theorem, we have $\int_{\mathbb{S}_{\infty}} \chi_{i,k}(\psvar^{\prime})\; d\post{\ssvarn{n}}(\psvar^{\prime})\rightarrow \int_{\mathbb{S}_{\infty}} \chi_{i,k}(\psvar^{\prime})\;d\delta_{\psvar_i^+}(\psvar^{\prime})$ for $\pmap{\psvar_0}$-a.s. many $\ssvar$ in $\sspace$. By the property (i) of $\chi_{i,k}$ above, we have $\int_{\mathbb{S}_{\infty}} \chi_{i,k}(\psvar^{\prime})\;d\delta_{\psvar_i^+}(\psvar^{\prime})= \chi_{i,k}(\psvar_i^+)=1$. Hence $\int_{\mathbb{S}_{\infty}} \chi_{i,k}(\psvar^{\prime})\; d\post{\ssvarn{n}}(\psvar^{\prime})\rightarrow 1$ for $\pmap{\psvar_0}$-a.s. many $\ssvar$ in $\sspace$. By applying the Dominated Converence Theorem with respect to $\pmap{\psvar_0}$, we have  $\lim_n \int_{\sspace} \int_{\mathbb{S}_{\infty}} \chi_{i,k}(\psvar^{\prime})\; d\post{\ssvarn{n}}(\psvar^{\prime}) \; d\pmap{\psvar_0}(\ssvar)=1$. This then puts the pair $(\prior, \psvar_0)$ in $U_{i,\epsilon,m,k}$. Thus indeed $U_{i,\epsilon,m,k}$ is dense. 

Let $V\subseteq \mathbb{S}_{\infty}$ be non-empty c.e. open, let $0<\epsilon<1$ be rational and consider \begin{equation}
  \{ (\prior, \psvar): \limsup_n \int_{\sspace} \int_{\mathbb{S}_{\infty}} I_V(\psvar^{\prime}) \; d\post{\ssvarn{n}} (\psvar^{\prime}) \;\; d\pmap{\psvar}(\ssvar) >1-\epsilon\} \label{eqn:thesetv} 
\end{equation}
Suppose $(\prior,\psvar)$ is in $\bigcap_{i,\epsilon, m,k} U_{i,\epsilon, m,k}$. We argue that $(\prior,\psvar)$ is in the set in (\ref{eqn:thesetv}). Since $V$ is non-empty, there is $\psvar^+_i$ in $V$. Then by the property (ii) of $\chi_{i,k}$ above we have that there is $k_0\geq 0$ such that for all $k\geq k_0$ we have $\chi_{i,k}\leq I_V$. Then we are done since $(\prior, \psvar)$ is in $U_{i,\epsilon, m,k}$ for all $m\geq 0$ and all $k\geq k_0$.

Then by Proposition~\ref{prop:thelprop}, we have finished proving Theorem~\ref{thm:freedman}. 

It remains to prove Corollary~\ref{cor:freedman}. By the Effective Baire Category Theorem (\cite[Theorem 5.1.2]{Hoyrup2009-pl}) and Theorem~\ref{thm:freedman}, in any non-empty c.e. open in $\mathcal{P}(\mathbb{S}_{\infty})\times \mathbb{S}_{\infty}$ we can find a computable point $(\prior, \psvar_0)$ in $\bigcap_{i,\epsilon, m,k} U_{i,\epsilon, m,k}$. Using the end of the proof of Theorem~\ref{thm:freedman}, starting at (\ref{eqn:thesetv}), together with Proposition~\ref{prop:thelprop} specialized to the computable prior $\prior$, we further get that $(\prior, \psvar_0)$ is not computably consistent. Hence we have finished part (\ref{cor:freedman:1}). And part (\ref{cor:freedman:2}) follows from Proposition~\ref{prop:sinftyprob} and Theorem~\ref{thm:doob}.

\stoptoc

\section*{Acknowledgments and funding}

\subsection*{Acknowledgments}

Thanks to the referees and editors for the helpful comments and feedback. Thanks also to the members of the Department of Philosophy and the Department of Statistics at the University of Toronto for the stimulating questions and discussion.

\subsection*{Funding}

The authors received no external funding for this work.


\resumetoc